\documentclass{amsart}

\usepackage{amsmath}

\usepackage{amsthm}
\usepackage{amssymb}
\usepackage{latexsym}
\usepackage[utf8]{inputenc}
\usepackage{textcomp}
\usepackage{color}
\usepackage[pdftex]{graphicx}

\usepackage{xspace}

\numberwithin{equation}{section}

\PassOptionsToPackage{obeyspaces}{url} 
\usepackage{hyperref}

\usepackage[all]{xy}
\usepackage{enumerate}

\newenvironment{enumeratea}
{\begin{enumerate}[\upshape (a)]}
{\end{enumerate}}

\newenvironment{enumeratei}
{\begin{enumerate}[\upshape (i)]}
{\end{enumerate}}

\newtheorem{thm}[equation]{Theorem}

\newtheorem{lemma}[equation]{Lemma}
\newtheorem{prop}[equation]{Proposition}
\newtheorem{cor}[equation]{Corollary}
\newtheorem{defin}[equation]{Definition}
\theoremstyle{definition}
\newtheorem{ex}[equation]{Exercise}
\newtheorem{examp}[equation]{Example}

\newtheorem{rmk}[equation]{Remark}

\newcommand{\gr}{\emph}

\newcommand{\C}{\mathcal{C}}
\renewcommand{\P}{\mathbb{P}}
\newcommand{\N}{\mathbb{N}}
\newcommand{\G}{\mathcal{G}}
\newcommand{\E}{\mathcal{E}}
\newcommand{\A}{\mathbb{A}}
\renewcommand{\L}{\mathcal{L}}
\newcommand{\Q}{\mathbb{Q}}

\newcommand{\art}{(\Art_k)}
\newcommand{\artl}{(\Art_{\Lambda})}
\newcommand{\set}{(\Set)}

\newcommand{\schs}{(\Sch_S)}
\newcommand{\fvect}{(\Fvect_k)}
\newcommand{\vect}{(\Vect_k)}
\newcommand{\fmod}{(\Fmod_A)}
\renewcommand{\mod}{(\Mod_A)}
\newcommand{\lie}{[-,-]}
\newcommand{\compl}{(\Comp_\Lambda)}

\newcommand{\eps}{\epsilon}
\newcommand{\dual}{k[\eps]}
\renewcommand{\phi}{\varphi}
\newcommand{\m}{\mathfrak{m}}
\renewcommand{\O}{\mathcal{O}}
\newcommand{\wh}{\widehat}
\newcommand{\F}{\mathcal{F}}
\newcommand{\fib}{\F\to\artl^\op}
\newcommand{\diff}{\Omega}
\newcommand{\op}{\mathrm{op}}
\newcommand{\Homsh}{\mathcal{H}om}

\newcommand{\catdef}{\mathcal{D}e\hspace{-.9pt}f}
\newcommand{\defgl}{\mathcal{D}e\hspace{-.9pt}fgl}
\newcommand{\hilb}{\mathcal{H}ilb}
\newcommand{\qcoh}{\mathcal{QC}oh}
\newcommand{\coh}{\mathcal{C}oh}
\newcommand{\Lif}{\mathcal{L}if}

\newcommand{\nin}{{n \in \N}}
\newcommand{\tor}{\mathop{\rm Tor}\nolimits}

\DeclareMathOperator{\id}{id}
\DeclareMathOperator{\Vect}{Vect}
\DeclareMathOperator{\Mod}{Mod}
\DeclareMathOperator{\Modfl}{Mod^{fl}}
\DeclareMathOperator{\Algfl}{Alg^{fl}}

\DeclareMathOperator{\Hom}{Hom}

\DeclareMathOperator{\spec}{Spec}
\DeclareMathOperator{\invlim}{\underleftarrow{\lim}}

\DeclareMathOperator{\Ext}{Ext}
\DeclareMathOperator{\Tor}{Tor}
\DeclareMathOperator{\Def}{Def}
\DeclareMathOperator{\Art}{Art}
\DeclareMathOperator{\Set}{Set}
\DeclareMathOperator{\Sch}{Sch}
\DeclareMathOperator{\Fvect}{FVect}
\DeclareMathOperator{\coker}{coker}
\DeclareMathOperator{\aut}{Aut}
\DeclareMathOperator{\Inf}{Inf}
\DeclareMathOperator{\lif}{Lif}
\DeclareMathOperator{\der}{Der}
\DeclareMathOperator{\Fmod}{FMod}
\DeclareMathOperator{\car}{char}
\DeclareMathOperator{\Comp}{Comp}
\DeclareMathOperator{\pic}{Pic}

\DeclareMathOperator{\het}{ht}

\newcommand\arr{\ifinner\to\else\longrightarrow\fi}

\newcommand\arrto{\ifinner\mapsto\else\longmapsto\fi}
\newcommand\larr{\longrightarrow}

\renewcommand{\H}{\mathrm{H}}

\newcommand\eqdef{\overset{\mathrm{\scriptscriptstyle def}} =}

\newcommand{\sym}{\operatorname{Sym}}
\newcommand\generate[1]{\langle #1 \rangle}
\DeclareMathOperator{\Hilb}{Hilb}
\newcommand{\dr}[1]{(\mspace{-3mu}(#1)\mspace{-3mu})}
\newcommand{\ds}[1]{[\mspace{-2mu}[#1]\mspace{-2mu}]}

\newcommand\dash{\nobreakdash-\hspace{0pt}}

\begin{document}

\title[Deformation theory from the point of view of fibered categories]{Deformation theory\\from the point of view of fibered categories}

\author{Mattia Talpo}
\author{Angelo Vistoli}

\address{Scuola Normale Superiore\\Piazza dei Cavalieri 7\\56126 Pisa\\Italy
}
\email[Vistoli]{angelo.vistoli@sns.it}

\email[Talpo]{m.talpo@sns.it}


\begin{abstract}
We give an exposition of the formal aspects of deformation theory in the language of fibered categories, instead of the more traditional one of functors. The main concepts are that of tangent space to a deformation problem, obstruction theory, versal and universal formal deformations. We include proofs of two key results: a version of Schlessinger's Theorem in this context, and the Ran--Kawamata vanishing theorem for obstructions. We accompany this with a detailed analysis of three important cases: smooth varieties, local complete intersection subschemes and coherent sheaves.
\end{abstract}

\maketitle
\thispagestyle{empty}

\tableofcontents

\section{Introduction}

The purpose of this article is to give an exposition of the formal aspects of deformation theory, from the point of view of fibered categories.

Deformation theory is the infinitesimal part of the study of families of algebraic or geometric objects. If we have some family of, say, algebraic varieties $X \to S$ (i.e. a proper and flat finitely presented morphism) and $s_{0}$ is a point of $S$, denote by $S_{n}$ the scheme $\spec (\O_{S,s_{0}}/\m_{s_{0}}^{n+1})$ over $S$. Information about pullbacks $X_{n}$ of $X$ to $S_{n}$ for all $n$ can tell us a lot about the family $X$ (for example, it can be shown that the subschemes $X_{n}$, with the natural embeddings $X_{n} \subseteq X_{n+1}$, determine the geometric fibers in some Zariski neighborhood of $s_{0}$ in S).

More generally, we can state the basic question of deformation theory as follows. Suppose that we are given a local artinian ring $A$, and a geometric or algebraic object $\xi_{0}$ of some kind over its residue field $k$ (for example, a scheme, a scheme with a sheaf, an algebra, \dots), what are the possible liftings (or extensions) of $\xi_{0}$ to $A$? By a \emph{lifting} we mean an object of the same kind $\xi$ defined over $A$, flat over $A$, 	together with an isomorphism of the restriction of $\xi$ to $\spec k \subseteq \spec A$ with $\xi_{0}$. If $\xi$ and $\eta$ are liftings of $\xi_{0}$, an \emph{isomorphism of liftings} will be an isomorphism $\xi \simeq \eta$ inducing the identity on $\xi_{0}$.

An archetypical theorem in deformation theory takes the following form. There are three vector spaces $U$, $V$ and $W$ on $k$, with the following property. If $A'$ is a local artinian ring with residue field $k$, $I$ is an ideal of $A'$ such that $\m_{A'}I = 0$ (i.e., $I$ is a $k$-vector space), and $\xi$ is a lifting of $\xi_{0}$ to $A \eqdef A'/I$, then

\begin{enumeratei}

\item there is an element $\omega$ of $I \otimes_{k}W$, such that $\omega = 0$ if and only if the object $\xi$ can be lifted to $A'$;

\item if $\omega = 0$, then the set of isomorphism classes of liftings has a natural structure of a principal homogeneous space under the additive group $I\otimes_{k}V$, and

\item if $\xi'$ is a lifting of $\xi$ to $A'$, the group of automorphisms of $\xi'$ inducing the identity on $\xi$ is $I \otimes_{k} U$.

\end{enumeratei}

Then one can draw conclusions on the set of possible lifting of $\xi_{0}$ to a local artinian algebra $B$, by applying induction on $n$ in the situation where $A' = B/\m_{B}^{n+1}$ and $A = B/\m_{B}^{n}$. Thus, for example, if $V = 0$ (a situation that very rarely happens in practice), we can conclude that any two liftings of $\xi_{0}$ to $B$ are isomorphic.

In all the examples we know, $U$, $V$ and $W$ are cohomology groups of the same type, with increasing indices; for example, if $X_{0}$ is a smooth variety over $k$ and $T_{X_{0}}$ its tangent bundle, then $U = \H^{0}(X_{0}, T_{X_{0}})$, $V = \H^{1}(X_{0}, T_{X_{0}})$ and $W = \H^{2}(X_{0}, T_{X_{0}})$.

Another type of result is the existence of \emph{versal deformations} (a.k.a. \emph{hulls}); the main result in this direction is due to Schlessinger. Suppose that $\Lambda$ is a noetherian complete local ring with residue field $k$, and $\xi_{0}$ is an object over $k$; we want to study liftings of $\xi_{0}$ to local artinian $\Lambda$-algebras. In many elementary treatments one assumes $\Lambda = k$; however, the more general situation is not essentially harder, and it is useful in many contexts.

Here are two typical situation where it is necessary to consider a ring $\Lambda$ different from $k$. Suppose we have a family of varieties $X \arr S$ over a base scheme $S$ over $k$, and a rational point $s_{0} \in S(k)$. Let $Y_{0}$ be a closed subscheme of the fiber $X_{0}$ over $s_{0}$. Can we deform $Y_{0}$ inside the neighboring fibers? The first step would be to set up a deformation theory problem over the ring $\wh\O_{S,s_{0}}$.

Or, suppose that we have a variety $X_{0}$ over a field $k$ of positive characteristic, and we would like to lift $X_{0}$ to a field of characteristic~$0$. In this case, obviously considering $k$-algebras will not work. The right thing to do is to take a complete local discrete valuation ring $\Lambda$ with residue field $k$ and with quotient field of characteristic~$0$ (for example, the ring of Witt vectors in $k$, if $k$ is perfect), and try to lift $X_{0}$ to a variety over $\Lambda$. For this we need to study the deformation theory of $X_{0}$ over $\Lambda$.

Schlessinger's Theorem states that under mild hypotheses there is a noetherian complete $\Lambda$-algebra $R$ with residue field~$k$, and a sequence of liftings $\rho_{n}$ of $\xi_{0}$ to $\spec R_{n}$, where $R_{n} \eqdef R/\m_{R}^{n+1}$, with an isomorphism of each $\rho_{n}$ with the restriction of $\rho_{n+1}$ (a \emph{formal object} $\rho$ over $R$), satisfying the following condition. Notice that a homomorphism of $\Lambda$-algebras $R \to A$, where $A$ is a local artinian $\Lambda$-algebra, induces an object $\xi$ over $A$, by factoring $R \arr A$ through some $R_{n}$, and pulling back $\rho_{n}$ through the corresponding morphism $\spec A \to \spec R_{n}$ (this $\xi$ is easily seen to be independent of $n$, up to a canonical isomorphism). We require the following.

\begin{enumeratea}

\item If $R \to A$ is a homomorphism of $\Lambda$-algebras, $\xi$ is the corresponding object over $A$, $A' \to A$ is a surjection of local artinian $\Lambda$-algebras, and $\xi'$ is a lifting of $\xi$ to $A'$, then $\xi'$ is isomorphic to an object coming from a lifting $R \to A'$ of the given homomorphism $R \to A$.

\item If $\xi$ is a object defined over the ring of dual numbers $k[\epsilon]$, then there exists a unique homomorphism $R \to k[\epsilon]$ inducing an object isomorphic to $\xi$.

\end{enumeratea}

The second condition is a minimality condition, ensuring that $R$ is unique, up to a non-canonical isomorphism. The first one says that knowing $R$ gives a considerable degree of control over liftings; for example, it implies that every lifting of $\xi_{0}$ to any local artinian $\Lambda$-algebra $A$ is induced by a homomorphism $R \to A$. Furthermore, if $\xi$ is a lifting of $\xi_{0}$ over some $A$, coming from a homomorphism $R \to A$, and $A' \to A$ a surjection of local artinian $\Lambda$-algebras, the existence of a lifting of $\xi$ to $A'$ is equivalent to the existence of a lifting $R \arr A'$ of the homomorphism $R \arr A$ (often a considerably easier problem).

The formal object $\rho$ over $R$ is called a \emph{miniversal deformation}, or a \emph{hull}, in Schlessinger's terminology.

The point of view in Schlessinger's paper, as in most of the literature, is that of functors; that is, we consider liftings of an object $\xi_{0}$ only up to isomorphism. A somewhat more sophisticated point of view, that of fibered categories, was introduced in \cite{Rim}; this allows to keep track of automorphisms, rather than collapsing them as we do when passing to isomorphism classes. This makes the formalism more elaborate; on the other hand, it captures more, and we hope to convince the reader that in the end it leads to a simplification. Certainly the results are cleaner and, in our opinion, more natural: for example, Schlessinger's conditions are replaced by a unique condition, the \emph{Rim--Schlessinger condition} (Definition~\ref{def:rs}), that does not distinguish between two cases. And in any case, when learning algebraic geometry it is becoming more and more important to learn about stacks, which are particular fibered categories; thus it could be argued that learning about deformation theory in the same context is not unnatural, and can be advantageous.

We stress that the object of this paper is the formal theory; we give several illustrative examples (deformation theory of schemes, subschemes and coherent sheaves), which are fundamental, but we don't aim at completeness. Thus, for example, we don't include any material on the deformation theory of maps, which is extremely important, nor on the cotangent complex \cite{Ill}, as this would have increased the size of this article considerably. Also, we don't discuss at all the approach to deformation theory using differential graded Lie algebras. Our focus is on the abstract formalism; hopefully the examples and applications that are sprinkled all over will be sufficient to keep the reader awake. 

We do not make any claim of originality, as all the material is very well known; we hope that the style of the exposition, with detailed proofs of all the results, and the content, slightly more advanced than that of \cite{Ser} or of \cite{hartshorne-deformation}, will be sufficient to attract some readers.

\subsection*{Description of content}


Section~\ref{deformation.cat.cap} contains our treatment of ``deformation categories''; these are fibered categories over the opposite of the category of local artinian $\Lambda$-algebras, which satisfy an analogue of Schlessinger's condition, which are sufficient for some of the basic constructions.

We describe Schlessinger's setup in \S\ref{deform.def}. We include no proofs, because we avoid this point of view in the rest of the paper.

In \S\ref{fibered.cat} we recall the formalism of fibered categories; our reference for all the proofs is \cite[Chapter~3]{FGA}.

The all-important definition of a \emph{deformation category} is given in \S\ref{fcadp}. This is a fibered category satisfying what we call the \emph{Rim--Schlessinger condition}, RS for short, which allows, in particular, to glue together objects along closed subschemes of spectra of local artinian rings, a basic construction that will be used all the time in the rest of our treatment.

We introduce out three basic examples in \S\ref{examp}: schemes, closed subschemes and quasi-coherent sheaves; these will accompany us all the way to the end. Here we are content with showing that they satisfy the RS condition.

We define the tangent space to a deformation category in Section~\ref{captang}. The definition is in \S\ref{tangdef}; the RS condition is used in an essential way to construct the vector space structure. In \S\ref{actionsec} we prove the all-important Theorem~\ref{actionthm}, which shows that the tangent space gives some control over the liftings. 

In \S\ref{tangent-examples} we compute the tangent spaces in our basic examples. We give an application in \S\ref{smooth.proj}: we show that any infinitesimal deformation of a smooth hypersurface of degree~$d$ in $\mathbb P^{n}$ is still a hypersurface, except for $n = 2$, $d \geq 5$ and for $n = 3$, $d = 4$.

Section~\ref{inf.aut} is a short one. Here we discuss infinitesimal automorphisms, and compute them in our basic examples.

In Section~\ref{obs} we discuss another extremely important concept, the obstruction theory. Contrary to the constructions of the previous sections, an obstruction theory is not canonical, and in fact in many concrete cases the same deformation problem has more than one obstruction theory attached to it. An obstruction theory gives for each surjective homomorphism $A' \arr A$ of $\Lambda$-algebras, whose kernel is a $k$-vector space, and every lifting $\xi$ of $\xi_{0}$ to $A$, an obstruction in a vector space whose vanishing is necessary and sufficient for the existence of a lifting of $\xi$ to $A'$. Thus, the vanishing of all obstructions means that objects can always be lifted; in this case we say that $\xi_{0}$ is \emph{unobstructed}.

After the definition, we show that the every obstruction theory contains a minimal one, which is canonical (but, unfortunately, mostly incomputable). Then we state the Ran--Kawamata Theorem, which we prove later. This gives a general condition for the vanishing of obstructions in characteristic~0; it implies, for example, that smooth proper varieties with trivial canonical bundle are unobstructed in characteristic~0.

Next, we analyze our three running examples. For the case of schemes, we construct the obstruction theory in the smooth case. We state the result for generically smooth local complete intersections, without including a proof (as far as we know there is no really simple one). We include the famous example of Kodaira and Spencer of a smooth obstructed projective variety.

Section~\ref{capform} is the heart of the paper. We define the category of formal objects, which is fibered over the opposite of the category of local complete noetherian $\Lambda$-algebras with residue field $k$. We define the Kodaira--Spencer homomorphism, which associates with every formal object a linear map, representing in some sense the differential of the formal object. Then we define universal, versal and miniversal objects, and we state and prove Schlessinger's Theorem in our context, stating that under very mild and natural conditions, a miniversal deformation exists. As an application, we prove the Ran--Kawamata Theorem. We also show how a miniversal deformation (which is unique, up to a non-unique isomorphism) determines the dimension of the minimal obstruction space. We also give some examples.

In \S\ref{algebraization} we analyze the problem of algebraization of schemes; namely, we give criteria for deciding when a formal scheme (i.e., a formal deformation, in our language) over a complete local ring comes from an actual scheme. We also give an example showing where this fails, even in the smooth case.

Section~\ref{capcurv} is an extended example, meant to illustrate many of the results; we analyze in detail the deformation theory of nodal curves. We also  use deformation theory to prove some results on the local structure of a family of nodal curves which are commonly used in the literature.

The various appendices contain several technical results and definitions that are used in several places in the text.

\subsection*{Prerequisites} These notes assume a basic knowledge of algebraic geometry, at the level of Hartshorne's book \cite{Har}. We also assume that the reader is comfortable with the fundamental concepts of commutative and homological algebra. For certain topics we assume a little more specialized knowledge, in particular concerning regular sequences and local complete intersections. These topics are very well discussed in \cite{Mat} (a great all around reference for commutative algebra) and \cite[Chapter~IV]{Fulton-Lang}.

\subsection*{Acknowledgements} This text is derived from the first author's undergraduate thesis (which can be found at \url{http://etd.adm.unipi.it/theses/available/etd-06292009-165906/}), which in turn was based on notes from two courses on moduli theory that the second author taught at the University of Bologna from 2001 to 2003. The second author would like to thank all the students that attended the courses, and particularly Alessandro Arsie, Damiano Fulghesu and Elisa Targa, who took the notes.

We are also grateful to Damiano for his assistance while the first author was drafting his thesis. Part of the work on thesis was done while both authors were visitors at MSRI; we enjoyed our stay immensely, and we are thankful for the hospitality.

We are in debt with Ian Morrison and Gavril Farkas for the invitation to write about deformation theory, for their interest in this work, and for their encouragement.

\subsection*{Some random suggestions for further readings} Unfortunately, there is no comprehensive introduction to deformation theory. Here are a few suggested references; the list is by no means supposed to be exhaustive. For many more, here are two good online resources: Charles Doran's historical bibliography (\url{http://www.math.columbia.edu/~doran/Hist Ann Bib.pdf}) and the nLab page on deformation theory (\url{http://ncatlab.org/nlab/show/deformation+theory}).

Two recent, very good elementary treatments, pitched at a slightly more elementary level than ours, are \cite{Ser} and \cite{hartshorne-deformation}. The canonical reference is \cite{Ill}, together with its second volume \cite{illusie-cotangent2}; it is quite abstract, but very well written, and amply rewards the patient reader. For the analytic treatment, one can consult \cite{kodaira-deformation}.

The other references have a narrower goal. The second author learnt a lot, in the ancient times when he was a graduate student, from \cite{artin-tata}, which focusses on isolated singularities. For deformations of analytic singularities, one can profitably consult \cite{greuel-lossen-shustin}. Koll\'ar's book \cite{kollar-rational-curves} contains a very complete discussion of the deformation theory of subschemes.

For an introduction to the more modern point of view, using differential graded Lie algebras, two good sources are \cite{manetti} and \cite{manetti-differential-formal-def}. An excellent exposition of the deformation theory of associative algebras, including an introduction to Kontsevich's theory of the deformation theory of Poisson manifolds, is in \cite{doubek-markl-zima}.

\subsection*{Notations and conventions}

All rings will be commutative with identity, and noetherian. If $A$ is a local ring, $\m_A$ will denote its maximal ideal.

The symbol $\Lambda$ will denote a noetherian local ring, which is complete with respect to the $\m_\Lambda$-adic topology, meaning that the natural homomorphism $\Lambda \to \invlim( \Lambda/\m_\Lambda^n)$ is an isomorphism. By $k$ we denote a (not necessarily algebraically closed) field; it will usually be the residue field $\Lambda/\m_\Lambda$ of $\Lambda$.

We denote by $\artl$ the category of local artinian $\Lambda$-algebras with residue field $k$. The \emph{order} of an object $A \in \artl$ will be the least $n$ such that $\m_A^{n+1}=0$. Similarly $\compl$ will be the category of noetherian local complete $\Lambda$-algebras with residue field $k$.

Notice that all homomorphisms in $\artl$ and $\compl$ are automatically local. In general, if $A$ is a local ring with residue field $k$, we will denote by $(\Art_A)$ the category of local artinian $A$-algebras with residue field $k$, and by $(\Comp_{A})$ the category of noetherian local complete $A$-algebras with residue field $k$.

We will be dealing with the opposite categories $\artl^{\op}$ and $\compl^{\op}$. This should cause no confusion regarding objects, since they are the same ones as in the categories above. So it will be equivalent to say for example ``let $A$ be an object of $\artl$'' and ``let $A$ be an object of $\artl^{\op}$''. Concerning arrows, every time $A$ and $B$ are rings, the symbol $A\to B$ will always denote a genuine homomorphism of rings, which of course becomes an arrow from $B$ to $A$ in the opposite category. When we will need to refer specifically to this last arrow, we will denote it by $(A\to B)^{\op}$ or by $\phi^{\op}$ (if $\phi$ is the homomorphism $A\to B$), and the reader should always view it as the induced morphism of affine schemes $\spec B \to \spec A$.

An important role will be played by the $\Lambda$-algebra $\dual$, that is, the $k$-algebra $k[t]/(t^2)\simeq k\oplus k\epsilon$ (the \gr{ring of dual numbers} of $k$), where $\eps \eqdef [t]$.

A generalization of the algebra of dual numbers is the following. Given a ring $A$ and an $A$-module $M$, we denote by $A \oplus M$ the $A$-algebra in which the product is defined by the rule $(a,m)(b,n) = (ab, an + bm)$. Its maximal ideal is $M$, and $M^{2} = 0$. We will call this the \gr{trivial algebra} structure on $A \oplus M$.

In particular if $V$ is a $k$-vector space, $k\oplus V$ becomes a $k$-algebra as well as a $\Lambda$-algebra, by means of the map $\Lambda \to k$.

When dealing with categories, as customary we will not worry about set-theoretic problems, so in particular the collections of objects and arrows will always be treated as sets. A functor $F\colon \mathcal{A}\to \mathcal{B}$ will always denote a covariant functor from $\mathcal{A}$ to $\mathcal{B}$; a contravariant functor from $\mathcal{A}$ to $\mathcal{B}$ will be considered as a covariant functor from the opposite category, written $F\colon \mathcal{A}^\op\to \mathcal{B}$. If $\mathcal{A}$ is a category, $A\in \mathcal{A}$ will mean that $A$ is an object of $\mathcal{A}$.

We denote by $\set$ the category of sets, by $\mod$ (resp. $\fmod$) the category of (finitely generated) modules over the ring $A$, by $\vect$ (resp. $\fvect$) the category of (finite-dimensional) $k$-vector spaces, by $\schs$ the category of schemes over a base scheme $S$. By a \emph{groupoid} we mean a category in which all arrows are invertible. A \emph{trivial groupoid} will be a groupoid in which for any pair of objects there is exactly one arrow from the first to the second; they are precisely the categories that are equivalent to the category with one object and one arrow.

All schemes we will consider will be locally noetherian, and if $f\colon X\to Y$ is a morphism of schemes, $f^\sharp\colon \O_Y\to f_*\O_X$ will denote the corresponding homomorphism of sheaves. If $X$ is a scheme, we write $|X|$ for the underlying topological space, and \emph{quasi-coherent} $\O_X$\emph{-module} as well as \emph{quasi-coherent sheaf} will always mean quasi-coherent sheaf of $\O_X$-modules. Usually, we specify the structure sheaf only when there are different schemes with the same underlying topological space. If $x \in X$ is a point of the scheme $X$, we will denote by $k(x)$ its residue field $\O_{X,x}/\m_x$.

If $X$ is a scheme over $k$, the sheaf of K\"{a}hler differentials $\diff_{X/k}$ on $X$ coming from the morphism $X\to \spec k$ will be denoted simply by $\diff_X$, and we use the same convention with the tangent sheaf; in the same fashion, the sheaf of continuous differentials $\wh{\diff}_{R/k}$ of an object $R \in (\Comp_k)$ will be denoted by $\wh{\diff}_R$ (see Appendix~\ref{appb}). Moreover by a \emph{rational point} of $X$ we always mean a $k$-rational point.

A \emph{pointed scheme} $(S, s_{0})$ over a field $k$ will be a $k$-scheme $S$ with a distinguished rational point $s_{0} \in S(k)$. A morphism $\psi\colon (R, r_{0})\to (S, r_{0})$ of pointed schemes is a morphism $\psi\colon R \arr S$ of $k$-schemes, carrying $r_{0}$ into $s_{0}$.

If $X\subseteq Y$ is a closed immersion of schemes over a ring $A$, with sheaf of ideals $I$, by \emph{the conormal sequence} associated with this immersion we will always mean the exact sequence of $\O_X$-modules
$$
\xymatrix{
I/I^2\ar[r] & \diff_{Y/A}|_X \ar[r]& \diff_{X/A}\ar[r] & 0.
}
$$

If $X$ is a scheme over a ring $A$ and $A\to B$ is a ring homomorphism, we denote by $X_B$ the base change $X\times_{\spec A}\spec B$, and we will use the same notation for pullbacks of quasi-coherent sheaves.

If $\mathcal{U}=\{U_i\}_{i \in I}$ is an open cover of a topological space $X$, we will denote by $U_{ij}$ the double intersection $U_i\cap U_k$, by $U_{ijk}$ the triple intersection $U_i\cap U_j \cap U_k$, and so on.

By a \emph{variety} we will always mean a reduced separated scheme of finite type over the field $k$. A \emph{curve} will be a variety of dimension 1.


\section{Deformation categories}\label{deformation.cat.cap}

In this section we start from a geometric deformation problem, and see how it leads naturally to a functor (the deformation functor of the problem), which is formed by taking isomorphism classes in a category fibered in groupoids. We recall briefly Schlessinger's Theorem of deformation functors to motivate the introduction of the abstract objects we will use to formalize deformations problems, that is, deformation categories. Finally, we introduce our three illustrative examples, which will be analyzed in detail throughout this work.

\subsection{Deformation functors}\label{deform.funct}

We start by describing the most basic example of a deformation problem, that of deformations of schemes.

Let $X_0$ be a proper scheme over $k$; we are interested in families having a fiber over a rational point isomorphic to $X_0$.
\begin{defin}\label{deform.def}
A \gr{(global) deformation} of $X_0$ is a cartesian diagram of schemes over $k$
$$
\xymatrix{
X_0\ar[r] \ar[d] &  X \ar[d]^f\\
\spec k \ar[r] & S
}
$$
where $f\colon X\to S$ is a flat and proper morphism, and $S$ is connected.
\end{defin}
Sometimes $X$ is called the \gr{total scheme} of the deformation, $S$ the \gr{base scheme}. We denote by $s_{0}$ the rational point given by the image of $\spec k \to S$, as we call it the \emph{base point} of the deformation; this makes $(S, s_{0})$ into a pointed scheme.

Notice that to give a deformation we can equivalently give a flat and proper morphism $f\colon X\to S$ and an isomorphism of the fiber of $f$ over a rational point $s_0\in S$ with $X_0$. We will usually refer to a deformation simply as the morphism of schemes, leaving the rational point of $S$ and the isomorphism with the fiber understood.

The flatness hypothesis in Definition~\ref{deform.def} is in some sense a continuity condition (it ensures that, locally on $X$, the fibers do not vary too wildly). The properness condition must be assumed when dealing with global deformations, to prevent constructions that are really ``wrong''. In the situation above, we consider the fibers of $X \to S$ over other rational points of $S$ to be deformations of $X_{0}$ (this is justified by the fact, for example, that many invariants are constant under deformation). But acts of vandalism, such as randomly deleting fibers over closed points of $S$ different from $s_{0}$, or deleting points on such fibers, will change them in such a way that they cannot be considered as deformations of $X_{0}$, unless one wants to end up with a completely useless notion. This problem does not arise with infinitesimal deformations, because there are no other fibers, and moreover it is important to study deformations of affine schemes with isolated singularities. Because of these reasons, we will drop the properness assumption once we focus on infinitesimal deformations. One can define global deformations of isolated singularities, but the definition is a little more subtle (see \cite{artin-stacks}). In dealing with deformations over non-noetherian schemes one usually asks the morphism to be also finitely presented.

\begin{defin}
An \gr{isomorphism} between global deformations $f\colon X\to S$ and $g\colon Y\to S$ of $X_0$ with the same distinguished point $s_{0}\in S$ is an isomorphism of $S$-schemes $F\colon X\to Y$, inducing the identity on $X_0$.
\end{defin}

Every $X_0$ has a \gr{trivial deformation} over any scheme $S$ over $k$, given by the projection $X_0\times_{\spec k} S\to S$, and we can take as distinguished fiber any fiber over a rational point of $S$, since they are all isomorphic to $X_0$. A deformation of $X_0$ over $S$ is called \gr{trivial} if it is isomorphic to a trivial deformation.

Deformations over a fixed $(S,s_{0})$ with isomorphisms form a category, which is a groupoid by definition. We call this category $\defgl_{X_0}(S,s_{0})$.

This construction is functorial in the base space: if we are given a morphism $\psi\colon (R, r_{0})\to (S, s_{0})$ of pointed schemes and a deformation $f\colon X\to S$ of $X_0$, we can form the fibered product and consider the projection $R\times_{S}X\to R$, which is a deformation of $X_0$ over $R$.
$$
\xymatrix{
X_{0} \ar[r]\ar[d] &R\times_S X\ar[r]\ar[d] & X\ar[d]^f \\
\spec k \ar[r]^{r_{0}} & R\ar[r]^\psi& S
}
$$
Moreover, if we have two isomorphic deformations over $(S,s_{0})$, say $f\colon X\to S$ and $g\colon Y\to S$ with an isomorphism $F\colon X\to Y$, then $F$ induces an isomorphism
   \[
   \id\times_S F\colon R\times_S X\stackrel{\sim}{\larr} R\times_S Y\,,
   \]
and this association gives a \gr{pullback functor} $\psi^*\colon \defgl_{X_0}(S,s_{0})\to\defgl_{X_0}(R,r_{0})$.

For a number of reasons the first step in studying deformations is considering infinitesimal ones.

\begin{defin}
An \gr{infinitesimal deformation} is a deformation such that $S=\spec A$, where $A \in \art$, and the morphism $f\colon X\to \spec A$ is not necessarily proper. A \gr{first-order deformation} is an infinitesimal deformation with $A=\dual$.
\end{defin}

In the case of infinitesimal deformations $X$ and $X_0$ have the same underlying topological space, and what changes is only the structure sheaf. This is because the sheaf of ideals of $X_0$ in $X$ is nilpotent, being the pullback of the sheaf of ideals on $\spec A$ corresponding to the maximal ideal of $A$.

We can perform the same constructions as above with infinitesimal deformations. We will denote by $\catdef_{X_{0}}(A)$ for $A \in \art$ the category of infinitesimal deformations of $X_{0}$ over $\spec A$.

Now that we have restricted our attention to infinitesimal deformations, following the classical approach, we define a ``deformation functor'' for our problem.

\begin{defin}
The \gr{deformation functor} defined by $X_0$ is the functor
   \[
   \Def_{X_0}\colon \art\larr\set
   \]
defined on objects by sending $A$ to the set $\Def_{X_0}(A)$ isomorphism classes of infinitesimal deformations of $X_0$ over  $\spec A$, and sending a homomorphism $\phi\colon A\to B$ into the function $\phi_*\colon \Def_{X_0}(A)\to \Def_{X_0}(B)$ given by the pullback.
\end{defin}

\begin{rmk}
Notice that we introduced a covariant construction $\phi\mapsto \phi_*$ and still called it pullback, and not pushforward. This is because we always want to consider $\art^\op$ as a subcategory of $(\Sch_k)$, and from this point of view the function $\phi_*$ is the pullback $\widetilde{\phi}^*$ induced by the map $\widetilde{\phi}\colon \spec B\to \spec A$ corresponding to $\phi$.

To avoid this confusion, every time we will have a homomorphism $\phi\colon A\to B$ that induces a pullback function in some way, we will still denote it by $\phi_*$, keeping in mind that it is the pullback induced by the associated map on the spectra.

We also stress the fact that $\Def_{X_0}(A)$ is the set of isomorphism classes of a groupoid $\catdef_{X_0}(A)$, and the function $\phi_*$ is the one induced by the pullback functor we defined before, along the morphism of schemes $\spec B\to \spec A$.
\end{rmk}

In conclusion, the study of infinitesimal deformations of a fixed scheme $X_0$ leads to a functor $\art\to\set$. With this motivation in mind, we make the following definition.

\begin{defin}
A \gr{predeformation functor} is a functor $F\colon \art\to\set$, such that $F(k)$ is a set with one element.
\end{defin}

The idea is of course that the element of $F(k)$ is the object that is getting deformed, and the elements of $F(A)$ are (isomorphism classes of) its deformations on $\spec A$. Nearly every geometric deformation problem can be formalized in this setting; we will see some examples of how this is done.

After their introduction by Grothendieck, these functors have been studied by Schlessinger, in \cite{Schl} (another exposition can be found in \cite[Chapter 2]{Ser}). Since we will review most of the theory using fibered categories, there is no point in describing it in detail here. An exception is the so-called Schlessinger's Theorem (which is the central result of Schlessinger's paper), which will provide a basic condition for the fibered categories we will consider. To state the Theorem we need a couple of definitions.

\begin{defin}
A predeformation functor is \gr{prorepresentable} if it is isomorphic to a functor of the form $\Hom_{k}(R,-)$ for some $R \in (\Comp_k)$.
\end{defin}

Prorepresentability corresponds to the existence of what is called a universal formal deformation, and is clearly a good thing to have, but it is also quite restrictive. A substitute when prorepresentability fails is the existence of a \gr{hull}, which is a formal deformation having a weaker universality property. Again, we will not go into details here, because we will discuss all of this later in a more general context.

\begin{defin}
A \gr{small extension} is a surjective homomorphism $\phi\colon A'\to A$ in $\art$, such that $ker\phi$ is annihilated by $\m_{A'}$, so that it is naturally a $k$-vector space. A small extension is called \gr{tiny} if $\ker\phi $ is also principal and nonzero, or equivalently if $\ker\phi \simeq k$ as a $k$-vector space.
\end{defin}

\begin{defin}
The \gr{tangent space} of a predeformation functor $F$ is $TF \eqdef F(\dual)$.
\end{defin}

This is of course only a set in general, but it has a canonical structure of $k$-vector space if $F$ satisfies condition (H2) below (see \cite{Schl}).

Let $F$ be a predeformation functor, and suppose we are given two homomorphisms $A'\to A$ and $A''\to A$ in $\art$. Then we can consider the fibered product $A'\times_A A''$ (notice that this is still an object of $\art$), and we have a natural map $f\colon F(A'\times_A A'')\to F(A')\times_{F(A)}F(A'')$ given by the universal property of the target. Schlessinger's condition are as follows:

\begin{itemize}

\item[(H1)] $f$ is surjective when $A'\to A$ is a tiny extension.

\item[(H2)] $f$ is bijective when $A'=\dual$ and $A=k$.

\item[(H3)] The tangent space $TF$ is finite-dimensional.

\item[(H4)] $f$ is bijective when $A'=A''$ and $A'\to A$ is a tiny extension.

\end{itemize}

\begin{thm}[Schlessinger]\label{schlessinger.classical}
A predeformation functor $F$ has a hull if and only if it satisfies (H1),(H2),(H3) above, and it is prorepresentable if and only if it also satisfies (H4).
\end{thm}

Conditions (H1) and (H2) are usually satisfied when dealing with functors coming from geometric deformation problems. Because of this, a predeformation functor satisfying (H1) and (H2) is called by some authors a \gr{deformation functor}. Schlessinger's terminology is a bit different, since with ``deformation functor'' he means our predeformation ones.

As we mentioned, in this paper we are not going to use the formalism of functors, but that of categories fibered in groupoids, which we introduce next.

\subsection{Categories fibered in groupoids}\label{fibered.cat}

As we have seen, the deformation functor of a scheme $X_0$ is formed by taking isomorphism classes in a certain groupoid. This is what typically happens when a geometric deformation problem is translated into a functor. But sometimes, for example when using deformation theory to study moduli problems, it is useful and natural to keep track of isomorphisms and automorphisms.

This leads us to using categories fibered in groupoids instead of functors while developing our theory. In the example introduced above, this amounts to considering a certain category, which we will denote by $\catdef_{X_0}$, that has all infinitesimal deformations of $X_0$ as objects (see below for a precise definition). We have a natural forgetful functor $\catdef_{X_0}\to \art^\op$, which makes $\catdef_{X_0}$ into a category fibered in groupoids over $\art^\op$.

Here we recall the definitions and some basic facts about fibered categories. All the proofs and more about the subject can be found in \cite[Chapter~3]{FGA}.

\subsubsection{First definitions}

In what follows we consider two categories $\F$ and $\C$ with a functor $p_\F\colon \F\to\C$. In this context, the notation $\xi \mapsto T$ where $\xi \in \F$ and $T \in \C$ will mean $p_\F\xi = T$ (and we will sometimes say that $\xi$ is over $T$). Moreover we will call a diagram like this
$$
\xymatrix{
\xi \ar[r]^f \ar@{|->}[d] & \eta \ar@{|->}[d] \\
T \ar[r]^\phi & S
}
$$
\emph{commutative} if $p_\F f = \phi$ (and we will sometimes say that \emph{$f$ is over $\phi$}).

\begin{defin}
An arrow $f\colon \xi \to \eta$ of $\F$ is \gr{cartesian} if the following universal property holds: every commutative diagram
$$
\xymatrix@R-20pt@+5pt{
\nu \ar@{|->}[dd] \ar@/^1pc/[rrd]^g \ar@{-->}[rd]_h & &\\
 &\xi \ar@{|->}[dd] \ar[r]_f & \eta \ar@{|->}[dd] \\
U \ar[dr]^\psi & & \\
& T \ar[r]^{p_{\F}f} & S
}
$$
can always be filled with a dotted arrow, in a unique way.

In other words, given any two arrows $g\colon \nu \to \eta$ in $\F$ and $\psi\colon U\to T$ in $\C$, where $U=p_\F \nu $ and $T=p_\F \xi$, such that $p_\F f \circ \psi = p_\F g$, there exists exactly one arrow $h\colon \nu\to \xi$ over $\psi$ such that $f\circ h=g$.
\end{defin}

It is very easy to see that if we have two cartesian arrows $f\colon \xi\to\eta$ and $g\colon \nu\to\eta$ in $\F$ over the same arrow of $\C$, then there is a canonically defined isomorphism $h\colon \xi\simeq\nu$, coming from the universal property, and compatible with the two arrows, meaning that $g\circ h=f$.

\begin{defin}
Let $\C$ be a category. A \gr{fibered category over $\C$} is a functor $p_{\F}\colon \F\to \C$, such that for every object $\eta$ of $\F$ and every arrow $\phi\colon T\to p_\F\eta$ of $\C$, there exists a cartesian arrow $f\colon \xi \to \eta$ of $\F$ over $\phi$.
\end{defin}

Sometimes we will also say that $\F$ is a fibered category over $\C$.

In the situation above we say that $\xi$ is a \gr{pullback} of $\eta$ to $T$ along the arrow $\phi$. So fibered categories are basically categories in which we can always find pullbacks along arrows of $\C$. The existence of some sort of pullback is a very common feature when dealing with geometric problems, so it seems convenient to use the formalism of fibered categories in this context.

By the remark above, pullbacks are unique, up to a unique isomorphism.

\begin{defin}
If $T$ is an object of $\C$, we can define a \gr{fiber category}, which we denote by $\F(T)$: its objects are objects $\xi$ of $\F$ such that $p_\F\xi = T$, and its arrows are arrows $f\colon \xi\to\eta$ of $\F$ such that $p_\F f = \id_T$.

A fibered category $\F\to\C$ is a \gr{category fibered in groupoids} if for every object $T$ of $\C$ the category $\F(T)$ is a groupoid, i.e. every arrow of $\F(T)$ is an isomorphism.
\end{defin}

In the following we will always use categories fibered in groupoids.

We have the following criterion to decide whether a functor $\F\to \C$ gives a category fibered in groupoids.

\begin{prop}[{\cite[Chapter~3, Proposition~3.22]{FGA}}]\label{groupoid}
Consider a functor $\F\to \C$. Then $\F$ is a category fibered in groupoids over $\C$ if and only if the following conditions hold:

\begin{enumeratei}

\item Every arrow of $\F$ is cartesian.

\item Given an arrow $T\to S$ of $\C$ and an object $\eta \in \F(S)$, there exists an arrow $\xi\to\eta$ of $\F$ over $T\to S$ (which is automatically cartesian because of the previous condition).

\end{enumeratei}

\end{prop}

So a fibered category $\F\to\C$ is fibered in groupoids if and only if every arrow of $\F$ gives a pullback.

The ambiguity in the choice of a pullback is sometimes annoying when defining things that seem to depend on it. However, in these cases the constructions one ends up with are independent of the choice in some way (the construction of the pullback functors we will see shortly is an example). To avoid this annoyance, we make the choice of a pullback of any object along any arrow once and for all.

\begin{defin}
A \gr{cleavage} for a fibered category $\F\to \C$ is a collection of cartesian arrows of $\F$, such that for every object $\xi$ of $\F$ and every arrow $T\to S$ in $\C$, such that $\xi \in \F(S)$, there is exactly one arrow in the cleavage with target $\xi$ and over $T\to S$.
\end{defin}

We can use some appropriate version of the axiom of choice to see that every fibered category has a cleavage. Fixing a cleavage in a fibered category is somewhat like choosing a basis for a vector space: sometimes it is useful because it makes things clearer and more concrete, but usually one would like to have constructions that are independent of it.

In what follows we will always assume that we have a fixed cleavage when we are dealing with fibered categories. If we have an arrow $\phi\colon T\to S$ of $\C$ and an object $\xi \in \F(S)$, we will denote the pullback given by the cleavage by $\phi^*(\xi)$, or $\xi|_T$ when no confusion is possible.

Now suppose we have $\phi\colon T\to S$ an arrow of $\C$. We can define a \gr{pullback functor} $\phi^*\colon \F(S)\to\F(T)$ in the following way: an object $\xi$ goes to $\phi^*(\xi)$, the pullback along $\phi$, and an arrow $f\colon \xi \to \eta$ in $\F(S)$ goes to the unique arrow that fills the commutative diagram
$$
\xymatrix@!C@!R@R-10pt@C-14pt{
\phi^*(\xi)\ar[rr] \ar@{|->}[dd] \ar@{-->}[rd]^{\phi^*(f)}& & \xi \ar|\hole@{|->}[dd]\ar[rd]^f &\\
&\phi^*(\eta) \ar@{|->}[dd] \ar[rr]& & \eta \ar@{|->}[dd]\\
T \ar@{=}[dr] \ar|\hole[rr]& & S\ar@{=}[dr] & \\
& T \ar[rr] & & S.
}
$$
As with objects, when no confusion is possible we write $f|_T$ instead of $\phi^*(f)$.

It is very easy to see that a choice of a different cleavage will give another pullback functor, but the two will be naturally isomorphic. From now on we will leave this type of comment understood when doing constructions that use a cleavage.

Sending an object $T$ of $\C$ into the category $\F(T)$, and an arrow $\phi\colon T \to S$ into the pullback functor $\phi^*\colon \F(S)\to \F(T)$, seems to give a contravariant functor from $\C$ to the category of categories. This is not quite correct, because it could well happen that, if $\psi\colon S\to U$ is another arrow in $\C$, the functors $\phi^*\circ \psi^*$ and $(\psi\circ\phi)^*$ are not equal, but only canonically isomorphic. In this case we obtain a \gr{pseudo-functor} (\cite[Chapter~3, Definition~3.10]{FGA}).

Taking isomorphism classes in the fiber categories clearly fixes this problem: given a category fibered in groupoids $\F \to \C$ we have a functor $F\colon \C^\op \to \set$ that sends an object $T$ of $\C$ into the set of isomorphism classes in the category $\F(T)$, and an arrow $\phi\colon T\to S$ to the obvious pullback function $\phi^*\colon F(S)\to F(T)$.

\begin{defin}
We will call $F$ the \gr{associated functor} of $\F$.
\end{defin}

In general we cannot recover a category fibered in groupoids, up to equivalence. from its associated functor. This is possible for categories fibered in equivalence relations (see later in this section). See \cite{lieblich-osserman} for some other highly non-trivial cases in which reconstruction is possible.

\begin{examp}
As we will see, the categories $\catdef_{X_0}(A)$ introduced above can be put together as fiber categories of a category fibered in groupoids $\catdef_{X_0} \to \art^{\op}$. The deformation functor $\Def_{X_0}\colon \art\to \set$ is then precisely the associated functor of this category fibered in groupoids.
\end{examp}

\subsubsection{Morphisms and equivalence}

Suppose $p_\F\colon \F\to \C$ and $p_\G\colon \G\to\C$ are two categories fibered in groupoids.

\begin{defin}
A \gr{morphism} of categories fibered in groupoids from $\F$ to $\G$ is a functor $F\colon \F\to \G$ which is base-preserving, i.e. such that $p_\G\circ F=p_\F$.
\end{defin}

\begin{rmk}\label{natural}
If $T$ is an object of $\C$, the functor $F$ will clearly induce a functor $\F(T)\to \G(T)$ which we denote by $F_T$. In particular $F$ will induce a natural transformation between the associated functors of $\F$ and $\G$.
\end{rmk}

With this definition of morphism comes a notion of isomorphism between fibered categories, but as it often happens when dealing with categories, this notion is too strict.

\begin{defin}
Given two morphisms $F$, $G\colon \F \to \G$, a natural transformation $\alpha\colon F\to G$ is said to be \gr{base-preserving} if for every object $\xi$ of $\F$ the arrow $\alpha_\xi\colon F(\xi) \to G(\xi)$ is in $\G(T)$, where $T=p_\F(\xi)$.

An \gr{isomorphism} between $F$ and $G$ is a base-preserving natural equivalence.
\end{defin}

\begin{defin}
Two categories fibered in groupoids $\F\to \C$ and $\G \to \C$ are said to be \gr{equivalent} if there exist two morphisms $F\colon \F\to \G$ and $G\colon \G\to \F$, with an isomorphism of $F\circ G$ with the identity functor of $\G$ and of $G\circ F$ with the one of $\F$.
\end{defin}

In this case we will say that $F$ is an \gr{equivalence} between $\F$ and $\G$, and that $F$ and $G$ are \gr{quasi-inverse} to each other.

We have a handy criterion to decide whether a morphism of fibered categories is an equivalence.

\begin{prop}\label{fibered.equiv}{\cite[Chapter~3, Proposition~3.36]{FGA}}
A morphism of categories fibered in groupoids $F\colon \F\to \G$ is an equivalence if and only if $F_T\colon \F(T)\to \G(T)$ is an equivalence for every object $T$ of $\C$.
\end{prop}

\subsubsection{Categories fibered in sets}

A particularly simple class of fibered categories is that of categories fibered in sets.

\begin{defin}
A \gr{category fibered in sets} is a fibered category $\F \to \C$ such that $\F(T)$ is a set for any object $T$ of $\C$.
\end{defin}

Here we see a set as a category whose only arrows are the identities. The following proposition says that in a category fibered in sets pullbacks are ``strictly'' (meaning ``not only up to isomorphism'') unique, and this feature characterizes them.

\begin{prop}{\cite[Chapter~3, Proposition~3.25]{FGA}}
Let $\F \to \C$ be a functor. The category $\F$ is fibered in sets over $\C$ if and only if for every arrow $T\to S$ of $\C$ and every object $\xi$ of $\F(S)$ there exists a \gr{unique} arrow in $\F$ over $T\to S$ and with target $\xi$.
\end{prop}

Because of this uniqueness, when $\F\to \C$ is fibered in sets the associated pseudo-functor is actually already a functor, which we denote by $\Phi_\F\colon \C^\op\to \set$. Moreover any morphism $F\colon \F\to \G$ of categories fibered in sets over $\C$ will give a natural transformation $\phi_F\colon \Phi_\F\to \Phi_\G$, as in Remark~\ref{natural}. This association gives a functor from the category of categories fibered in sets over $\C$ and the category of functors $\C^\op\to \set$.

\begin{prop}{\cite[Chapter~3, Proposition~3.26]{FGA}}
The functor defined above is an equivalence of categories.
\end{prop}

We sketch briefly the inverse construction. Let $F\colon \C^\op\to\set$ be a functor, and consider the following category, which we call $\F_F$: as objects take pairs $(T,\xi)$, where $T$ is an object of $\C$ and $\xi \in F(T)$, while an arrow $f\colon (T,\xi)\to(S,\eta)$ will be an arrow $f\colon T\to S$ such that $F(f)(\eta)=\xi$. Then $\F_F$ is a category fibered in sets over $\C$.

Given a natural transformation $\alpha\colon F\to G$ between two functors $\C^\op\to \set$, we construct a functor $H_\alpha\colon \F_F\to \F_G$, as follows: an object $(T,\xi)$ of $\F_F$ goes to the object $(T,\alpha(T)(\xi))$ of $\F_G$, and an arrow $f\colon (T,\xi)\to (S,\eta)$ simply goes to itself (as an arrow $f\colon T\to S$ of $\C$). It can be shown that this gives a functor, which is a quasi-inverse to the one considered above.

\begin{examp}\label{scheme}
In particular if $X$ is a scheme over $S$, we can see it as a functor $h_X\colon \schs^\op\to \set$ (by the classical Yoneda Lemma), and also as a category fibered in groupoids $( \schs/X )\to \schs$ (by the preceding Proposition). To avoid this cumbersome notation we will write $X$ for $h_X$ and also for $( \schs/X )$.
\end{examp}

Another class of simple fibered categories are the ones fibered in equivalence relations. We say that a groupoid is an \gr{equivalence relation} if for any pair of objects there is at most one arrow from the first one to the second. Another way to say this is that the only arrow from any object to itself is the identity.

\begin{defin}
A fibered category $\F\to \C$ is said to be \gr{fibered in equivalence relations} if for every object $T$ of $\C$ the fiber category $\F(T)$ is an equivalence relation.
\end{defin}

The name ``equivalence relation'' comes from the fact that if a groupoid $\F$ is an equivalence relation, and we call $A$ and $O$ its sets of arrows and objects respectively, the map $A\to O\times O$ that sends an arrow into the pair $(\text{source},\text{target})$ is injective, and gives an equivalence relation on the set $O$.

We have the following fact, which characterizes categories fibered in equivalence relations.

\begin{prop}{\cite[Chapter~3, Proposition~3.40]{FGA}}
A fibered category $\F\to \C$ is fibered in equivalence relations if and only if it is equivalent to a category fibered in sets.
\end{prop}

Because of this, sometimes categories fibered in equivalence relations are called \gr{quasi-functors}.

Now suppose that $T$ in an object of $\C$, and consider the \gr{comma category} $(\C/T)$, defined as follows: its objects are arrows $S\to T$ of $\C$ with target $T$, and an arrow from $f\colon S\to T$ to $g\colon U\to T$ is an arrow $h\colon S\to U$ of $\C$, such that $g\circ h =f$.  We have a functor $(\C/T)\to \C$ that sends $S\to T$ into $S$, and an arrow as above to the arrow $h\colon S\to U$ of $\C$.

$(\C/T)\to \C$ is a category fibered in sets: given an arrow $S\to U$ of $\C$ and an object over $U$, that is, an arrow $U\to T$, the only possible pullback to $S$ is the composite $S\to U\to T$. It is also easy to see that this category fibered in sets is the one associated with the functor $h_T\colon \C^\op\to\set$ represented by $T$ (up to equivalence of course).

\begin{defin}
A category fibered in groupoids $\F\to\C$ is called \gr{representable} if it is equivalent to a category fibered in groupoids of the form $(\C/T)$.
\end{defin}

Clearly if $\F\to \C$ is representable, then it is fibered in equivalence relations.

%
%

\subsection{Fibered categories as deformation problems}\label{fcadp}

Now suppose that $\F\to \schs$ is a category fibered in groupoids coming from a geometric deformation problem, where $S=\spec k$ or some other base scheme (we will see how this association is carried out in some examples). Here the idea is that objects of the category $\F(T)$ are families with base scheme $T$, which parametrize some kind of algebro-geometric object.

A particular role is played by the objects of $\F(\spec k)$ where $k$ is a field over $S$ (i.e. with a morphism $\spec k\to S$), which correspond, by a version of the Yoneda Lemma for fibered categories, to morphisms of categories fibered in groupoids $\spec k \to \F$. These morphisms, in analogy with what happens with schemes, can be considered as ``points'' of the deformation problem (compare with the concept of points of an algebraic stack, for example).

The first step in the study of a deformation problem in the form of the category $\F$ as above is an infinitesimal study of the local geometry of $\F$ around a fixed ``point'' $\xi_{0} \in \F(\spec k)$; in other words, again in analogy with the case of schemes, one studies morphisms $\spec A\to \F$, where $A \in \art$, such that the composition $\spec k \to \spec A\to \F$ corresponds to the fixed object $\xi_{0}$. Again by the analogue of the Yoneda Lemma, these morphisms correspond to objects of $\F(\spec A)$.

This leads us to restricting the fibered category $\F\to \schs$ to the full subcategory $\art^{\op}\subseteq  \schs$, where $k$ is some field over $S$, that is, to considering infinitesimal deformations. Usually we will also concentrate our attention to objects restricting to a given $\xi_{0}$ over $k$, but it is not necessary to do so right away.

Actually, as explained in the introduction, it is sometimes useful to have a theory for deformations over artinian algebras over a complete noetherian local ring $\Lambda$ with residue field $k$. So with this motivation in mind, from now on we will study categories fibered in groupoids $\fib$, where $\Lambda$ is as above. We will turn back to ``global'' deformations only occasionally.

We stress once again that we will always identify $\artl^\op$ with the corresponding full subcategory of $(\text{Sch}_\Lambda)$. From now on we will drop the notation $\F(\spec A)$ and write simply $\F(A)$ for the fiber category of $\F$ over $\spec A$, and if $\fib$ is a category fibered in groupoids, $\phi\colon A'\to A$ a homomorphism in $\artl$, and $f\colon \xi \to \eta$ an arrow of $\F$ with $\xi \in \F(A)$ and $\eta \in F(A')$, we will say that $f$ is \gr{over} $\phi$ if its image in $\artl^\op$ is the arrow $\phi^{\op}$ from $A$ to $A'$ corresponding to $\phi$. \label{over_alg}

We will also draw some strange-looking commutative diagrams like this
$$
\xymatrix{
\xi \ar[r]^f\ar@{|->}[d] & \eta \ar@{|->}[d]\\
A & A' \ar[l]_\phi
}
$$
which should of course be read as
$$
\xymatrix{
\xi \ar[r]^f\ar@{|->}[d] & \eta \ar@{|->}[d]\\
\spec A\ar[r]^{\widetilde{\phi}} & \spec A'.
}
$$

\subsubsection{The Rim--Schlessinger condition}

For everything that follows we need to impose a very important gluing condition on a category fibered in groupoids $\fib$. Suppose that we have two homomorphisms $\pi'\colon A'\to A$, $\pi''\colon A''\to A$ in $\artl$, the second one being surjective. Consider the fibered product $A'\times_A A''$, which is easily seen to be an object of $\artl$ as well. We have two pullback functors $\F(A'\times_A A'')\to \F(A')$ and $\F(A'\times_A A'')\to \F(A'')$, such that the composites
$$
\xymatrix{
\F(A'\times_A A'')\ar[r] &\F(A') \ar[r] & \F(A)
}
$$
and
$$
\xymatrix{
\F(A'\times_A A'')\ar[r] &\F(A'') \ar[r] & \F(A)
}
$$
with the pullback functors to $A$ are isomorphic. We get an induced functor
$$
\Phi\colon \F(A'\times_A A'')\to \F(A')\times_{\F(A)}\F(A'')
$$
(see Appendix~\ref{appc} for the definition of fibered products of categories).

More explicitly, $\Phi$ sends an object $\xi$ into $(\xi|_{A'},\xi|_{A''},\theta)$ where $\theta\colon (\xi|_{A'})|_{A}\to (\xi|_{A''})|_{A}$ is the canonical isomorphism identifying the pullbacks of $\xi|_{A'}$ and $\xi|_{A''}$ to $A$ as pullbacks of $\xi$, and an arrow $f\colon \xi \to \eta$ is mapped to the pair $(f|_{A'},f|_{A''})$ of induced arrows on the pullbacks.

\begin{defin}\label{def:rs}
A category fibered in groupoids $\fib$ satisfies the \gr{Rim--Schlessinger condition} (RS from now on) if $\Phi$ is an equivalence of categories for every $A,A',A'' \in \artl$ and maps as above.
\end{defin}

This condition, which was first formulated by D. S. Rim in \cite{Rim}, resembles very much Schlessinger's, and actually implies (H1) and (H2) for the associated functor, as is very easy to see ((H4) is a little more subtle, see Proposition 2.1.12 of \cite{Oss}).

Despite the fact that RS is somewhat stronger than (H1)+(H2), when one proves that a given category fibered in groupoids (or rather its associated functor) satisfies the latter, he or she usually proves that the category satisfies RS (or could do so with little extra effort). Moreover all categories fibered in groupoids coming from reasonable geometric deformation problems seem to have the stated property, so we will take it as a starting point.

\begin{defin}
A \gr{deformation category} over $\Lambda$ is a category fibered in groupoids over $\artl^{\op}$ that satisfies RS. 
\end{defin}

Deformation categories are called ``homogeneous groupoids'' in \cite{Rim}, and ``deformation stacks'' in \cite{Oss}.

From now on when we have a deformation category $\fib$ with $A,A',A''$ artinian algebras as above, and objects $\xi' \in \F(A')$ and $\xi'' \in F(A'')$ with a fixed isomorphism of the pullbacks to $A$, we denote by $\{\xi',\xi''\}$ an induced object over the fibered product $A'\times_A A''$. When the isomorphism over $A$ or the choice of such an object is relevant, we will be more specific.

\begin{examp}
As a trivial example, we consider the category fibered in groupoids $X\to \artl^\op$ given by a scheme $X$ over $\spec\Lambda$.

If $X=\spec R$ is affine, for every $B \in \artl$ we have a natural bijection $X(B)\simeq\Hom_\Lambda(R,B)$, and if we take $A,A',A''\in \artl$ and maps as above, the map $X(A'\times_A A'')\to X(A')\times_{X(A)} X(A'')$ is a bijection because of the properties of the fibered product. When $X$ is not affine one reduces to the affine case by noticing that the image of the morphisms involved is a point of $X$, and taking an affine neighborhood.
\end{examp}

A \gr{morphism} of deformation categories will simply be a morphism of categories fibered in groupoids.

Given a deformation category $p_{\F}\colon \fib$ and an object $\xi_0$ over $\spec k$, we can construct another deformation category $\F_{\xi_0}$ that contains only objects of $\F$ that restrict to $\xi_0$ over $\spec k$ (and in this sense are deformations of $\xi_0$), taking the (dual) comma category:

\begin{description}
\item[Objects]{arrows $f\colon \xi_0\to \xi$ of $\F$, or equivalently pairs $(\xi,\phi)$ where $\xi$ is an object of $\F$ and $\phi$ is an arrow in $\F(k)$ between $\xi_0$ and the pullback of $\xi$ to $\spec k$.}
\item[Arrows]{from $f\colon \xi_0\to \xi$ to $g\colon \xi_0\to \eta$ are arrows $h\colon \xi\to\eta$ of $\F$ such that $h\circ f=g$, or equivalently the arrow $\xi_0\to\xi_0$ induced by $h$ is the identity.}
\end{description}

We have also an obvious functor $\F_{\xi_0}\to \artl^\op$ sending $(\xi,\phi)$ into $p_{\F}\xi$. The following will be useful when we have to consider deformations of a fixed object over $k$.

\begin{prop}\label{deformation.cat}
If $\fib$ is a deformation category and $\xi_0 \in \F(k)$, then $\F_{\xi_0}\to \artl^\op$ is also a deformation category.
\end{prop}

\begin{ex}
Prove Proposition~\ref{deformation.cat}.
\end{ex}


\subsection{Examples}\label{examp}

Now we introduce three examples of deformation problems that will show up systematically in the following, providing concrete examples to our abstract constructions. In each of these examples some additional hypotheses may be required (on the ambient scheme over $\Lambda$ in the case of deformations of closed subschemes, for example) to make things work out sometimes. We will specify these hypotheses case by case.

Each of these examples has also a classical associated deformation functor, which can be simply obtained by taking the associated functor of the deformation category we will introduce for the problem.

\subsubsection{Schemes}

The most basic example is the one already introduced, that of deformations of schemes without additional structure.

Let us consider the following category fibered in groupoids, which we will denote by $\catdef \arr \artl^\op$:

\begin{description}
\item[Objects]{flat morphisms of schemes $X\to \spec A$, where $A \in \artl$.}
\item[Arrows]{from $X\to \spec A$ to $Y\to \spec B$ are pairs $(\phi,f)$ where $\phi\colon B\to A$ is a homomorphism of $\Lambda$-algebras and $f\colon X\simeq Y_A$ is an isomorphism of schemes (recall that $Y_A$ denotes the base change $Y\times_{\spec B}\spec A$).}
\end{description}

Given two arrows $(\phi,f)$ from $X\to \spec A$ to $Y\to \spec B$ and $(\psi,g)$ from $Y\to \spec B$ to $Z\to \spec C$ the composite $(\psi,g)\circ(\phi,f)$ is $(\rho,h)$ where $\rho\colon C\to A$ is simply $\phi\circ \psi$, and if we call $g_A\colon Y_A \simeq (Z_B)_A$ the isomorphism induced by $g\colon Y\simeq Z_B$ by base change, then $h\colon X\simeq Z_A$ is given by the composite
$$
X\stackrel{f}{\larr}Y_A \stackrel{g_A}{\larr} (Z_B)_A\simeq Z_A
$$
where the last isomorphism is the canonical one.

We have a natural forgetful functor $\catdef\to \artl^\op$, and by the properties of the fibered product and the way we defined arrows we see that the conditions of Proposition~\ref{groupoid} are satisfied, so that $\catdef\to \artl^\op$ is a category fibered in groupoids. Notice that if $X_0 \in \catdef(k)$ and $A \in \artl$, the category $\catdef_{X_0}(A)$ is exactly the one we defined at the beginning, of flat schemes over $\spec A$ with an isomorphism of the closed fiber with $X_0$.


\begin{prop}\label{schemes.rs}
The category fibered in groupoids $\catdef\to \artl^\op$ is a deformation category.
\end{prop}

The proof of this is postponed to Appendix~\ref{appd}.

In what follows we will mostly be interested in schemes of finite type. However, if $X_{0}$ is a scheme of finite type over $k$, it is easy to see that any lifting of $X_{0}$ to any $A \in \artl$ is of finite type over $A$, so we don't need to restrict $\catdef$.

\subsubsection{Closed subschemes}

For our second example we want to consider, given a closed immersion of schemes $Y_0\subseteq X$ over $k$, families of subschemes of $X$ including the given $Y_0$ as a fiber over a rational point.

In our setting, given a scheme $X$ over $\spec \Lambda$, we consider the following category, which we will denote by $\hilb^{X}$:

\begin{description}
\item[Objects]{pairs $(A,Y)$ where $A \in \artl$ and $Y$ is a closed subscheme of $X_A$, flat over $A$.}
\item[Arrows]{from $(A,Y)$ to $(B,Z)$ are homomorphisms $B\to A$, such that the induced closed subscheme $Z_A\subseteq (X_B)_A$ corresponds to $Y\subseteq X_A$ under the canonical isomorphism $(X_B)_A\simeq X_A$.}
\end{description}

Composition is given by the usual composition of ring homomorphisms, and it is easily checked that this is well defined: that is, if we have $\phi\colon (A,Y)\to(B,Z)$ and $\psi\colon (B,Z)\to (C,W)$ arrows as above, then the composite $\phi\circ \psi\colon  C\to A$ is still an arrow in our category, i.e. the induced closed subscheme $W_A\subseteq (X_C)_A$ corresponds to $Y\subseteq X_A$ with respect to the canonical isomorphism $(X_C)_A\simeq X_A$.

We have a natural forgetful functor $\hilb^X\to \artl^\op$, and again by the properties of fibered products and definition of the arrows we easily see that we can apply Proposition~\ref{groupoid}, so that $\hilb^X\to\artl^\op$ is a category fibered in groupoids.

There is an important difference between this example and the previous one, namely the fact that in $\hilb^X$ arrows are uniquely determined by their image in $\artl$. This means that the associated pseudo-functor of $\hilb^X\to \artl^\op$ is actually a functor, and our fibered category is fibered in sets.

We will see later on that this is equivalent to saying that our deformation problem has no non-trivial infinitesimal automorphisms (see Proposition~\ref{equiv.rel}).

The notation $\hilb$ comes from the fact that the deformation category is the category fibered in sets associated with the Hilbert functor, if the ambient scheme $X$ is quasi-projective over $\Lambda$.

\begin{prop}\label{hilb.rs}
The category fibered in groupoids $\hilb^X\to \artl^\op$ is a deformation category.
\end{prop}

\begin{ex}
Using Proposition~\ref{schemes.rs}, prove Proposition~\ref{hilb.rs}.
\end{ex}

If $Y_0\subseteq X_0=X\times_{\spec \Lambda}\spec k$ is a closed subscheme, objects of $\hilb^X_{Y_0}$ are often called \gr{embedded} (infinitesimal) deformations of $Y_0$.

\subsubsection{Quasi-coherent sheaves}

For our last example, suppose we are given a quasi-coherent sheaf $\E_0$ on a scheme $X$ over $k$, and we want to consider families of quasi-coherent sheaves on $X$ having a fiber over a rational point isomorphic to $\E_0$.

Once again, we formulate the problem in terms of fibered categories. Given a scheme $X$ over $\Lambda$, we construct the category $\qcoh^X$ as follows:

\begin{description}
\item[Objects]{pairs $(A,\E)$, where $A \in \artl$ and $\E$ is a quasi-coherent sheaf on $X_A$, flat over $A$.}
\item[Arrows]{from $(A,\E)$ to $(B,\F)$ are pairs $(\phi,f)$, with $\phi\colon B\to A$ a homomorphism and $f\colon \E\simeq\F_A$ an isomorphism of quasi-coherent sheaves on $X_A$, where $\F_A$ is the pullback of $\F$ along the natural morphism $X_A\to X_B$.}
\end{description}

Composition is defined as in the first example: given $(\phi,f)\colon (A,\E)\to (B,\F)$ and $(\psi,g)\colon (B,\F)\to(C,\G)$, their composite $(\psi,g)\circ(\phi,f)$ is $(\rho,h)$, where $\rho\colon C\to A$ is the usual composite $\phi\circ \psi$, and if we denote by $g_A\colon \F_A\simeq (\G_B)_A$ the isomorphism induced by $g\colon \F\simeq \G_B$ by base change, then $h\colon \E\simeq \G_A$ is given by
$$
\E\stackrel{f}{\larr}\F_A\stackrel{g_A}{\larr}(\G_B)_A\simeq \G_A
$$
where the last isomorphism is the canonical one.

As before we have a forgetful functor $\qcoh^X\to \artl^\op$, and by our definition of arrows and properties of the pullback of quasi-coherent sheaves, we can use Proposition~\ref{groupoid}, and so $\qcoh^X\to \artl^\op$ is a category fibered in groupoids.

\begin{prop}\label{qcoh.rs}
The category fibered in groupoids $\qcoh^X\to \artl^\op$ is a deformation category.
\end{prop}

\begin{ex}
Prove Proposition~\ref{qcoh.rs}.
\end{ex}

\section{Tangent space}\label{captang}

This section is devoted to the introduction and study of the tangent space of a deformation category. This concept generalizes the corresponding ones for schemes and deformation functors.

After defining the tangent space and discussing its action on isomorphism classes of liftings, we will calculate it in our three main examples, and give an application to deformations of smooth hypersurfaces in $\P^n_k$.

\subsection{Definition}\label{tangdef}

Let $\fib$ be a deformation category and suppose $\xi_0 \in \F(k)$. We start by defining the tangent space as a set.

\begin{defin}
The \gr{tangent space} of $\F$ at $\xi_0$ is the set
$$
T_{\xi_0}\F=\{\text{isomorphism classes of objects in } \F_{\xi_0}(\dual)\}\,.
$$
\end{defin}

Now we want to justify the name of tangent space, showing that there is a canonical structure of $k$-vector space on $T_{\xi_0}\F$. To do so, we consider the functor $F\colon \fvect\to\set$ defined as follows: given a $V\in \fvect$, we take the trivial $\Lambda$-algebra $k\oplus V$ (where we recall that multiplication is defined by $(x,v)(y,w)=(xy,xw+yv)$) and associate with $V$ the set
$$
F(V)=\{\text{isomorphism classes of object in } \F_{\xi_0}(k \oplus V)\}\,.
$$
If $V\to W$ is a $k$-linear map, we get a homomorphism $k \oplus V\to k \oplus W$, and by pullback (in the fibered category $\F_{\xi_0}$) an arrow $F(V)\to F(W)$. Clearly $F(k)=T_{\xi_0}\F$.

Now we notice that $F$ has a lifting $\widetilde{F}\colon \fvect\to \vect$ to the category of $k$-vector spaces, so that each $F(V)$ (in particular $F(k)=T_{\xi_0}\F$) will have a natural structure of $k$-vector space. As shown in Appendix~\ref{appa}, to show this it suffices to check that $F$ preserves finite products (Definition~\ref{def:preserve-products}). This follows easily from the fact that $\F_{\xi_0}\to \artl^{\op}$ satisfies RS, as stated in Proposition~\ref{deformation.cat}.

\begin{ex}
Check that $F$ preserves finite products.
\end{ex}

We describe briefly the vector space structure that we obtain on each $F(V)$: first of all $F(0)$ has exactly one element, which is simply the isomorphism class of the identity $\xi_0\to \xi_0$ in $\F_{\xi_0}(k)$. Moreover every $V \in \fvect$ has a natural map $0\to V$ that induces $F(0)\to F(V)$; the zero element of $F(V)$ is then the image of this map. In our particular case this corresponds to the isomorphism class of the ``trivial'' pullback of $\xi_0$ along the inclusion homomorphism $k \to k \oplus V$.

Addition is defined by the composite
$$
F(V)\times F(V)\simeq F(V\oplus V) \xrightarrow{F(+)}F(V)
$$
where $+\colon V\oplus V \to V$ is the addition of $V$. Similarly multiplication by $a \in k$ is simply $F(\mu_a)\colon F(V)\to F(V)$, where $\mu_a\colon V\to V$ is multiplication by $a$.

From now on we will consider $F$ as a functor $\fvect\to \vect$.

\begin{rmk}\label{tangiso}
Suppose we have another object $\eta_0 \in \F(k)$, such that there is an arrow $f\colon \xi_0 \to \eta_0$ (which must be an isomorphism). It is clear then that $f$ will induce a bijection $T_{\xi_0}\F\to T_{\eta_0}\F$, which is an isomorphism of $k$-vector spaces.

So isomorphic objects over $k$ will have isomorphic tangent spaces.
\end{rmk}

As discussed in Appendix~\ref{appa} this canonical lifting $\fvect\to \vect$ is a $k$-linear functor, so we can apply Proposition~\ref{app.natural} and conclude that for every $V \in \fvect$ we have
$$
F(V) \simeq V\otimes_k F(k)= V\otimes_k T_{\xi_0}\F\,.
$$

\begin{examp}\label{dimcoarse}
If $\fib$ is the category fibered in groupoids coming from a scheme $X$ over $\Lambda$, then one sees easily that the tangent space $T_{\xi_0}\F$ is the usual Zariski tangent space of $X$ at the $k$-rational point corresponding to $\xi_0 \in X(k)$.

In particular if we have a (fine) moduli space $M$ representing a functor $F\colon \schs^\op\to \set$ (and so the corresponding deformation category satisfies RS over every point of $M$), we can get informations on the tangent space of $M$ at a point $m \in M$ by studying that of the deformation category associated with $F$ at the corresponding object over the residue field $k(m)$.

If a moduli problem has only a coarse moduli space, sometimes we can get informations about its dimension from the tangent spaces of the associated deformation categories. More precisely, if $\F\to \schs$ is a Deligne-Mumford stack with finite inertia, so that it has a coarse moduli space $\pi\colon \F\to F$, and it satisfies RS over a point $\xi_{0}:\spec k\to \F$ (i.e. it satisfies RS when restricted to $\art$, and considering only deformations of $\xi_{0}$), then the dimension of $F$ at the point $\spec k\to F$ is (the dimension of the stack $\F$ at the point $\xi_{0}$, and then) at most $\dim_{k}T_{\xi_{0}}\F$. If moreover $\F$ is smooth at the point $\xi_{0}$, then we have equality (see for example \ref{curvdim}).
\end{examp}

As expected, along with the concept of tangent space comes that of differential of a morphism.

Let $H\colon \F\to \G$ be a morphism of deformation categories, and suppose $\xi_0 \in \F(k)$. Then we have an induced morphism $H_{\xi_0}\colon \F_{\xi_0}\to \G_{H(\xi_0)}$ defined in the obvious way. If we call $F$, $G\colon \fvect \to \set$ the two functors involved in the construction of the tangent spaces of $\F$ at $\xi_0$ and $\G$ at $H(\xi_0)$ respectively, $H_{\xi_0}$ will induce a natural transformation $\phi\colon F\to G$.

Since $F$ and $G$ are $k$-linear functors, from Proposition~\ref{app.nat.linear} we see that $\phi$ is automatically $k$-linear. In particular $\phi(k)\colon F(k)\longrightarrow G(k)$ will be a $k$-linear map.

\begin{defin}
The \gr{differential} of $H$ at $\xi_0$ is the $k$-linear map
$$
d_{\xi_0}H=\phi(k)\colon T_{\xi_0}\F\to T_{H(\xi_0)}\G\,.
$$
\end{defin}

Concretely, given $a \in T_{\xi_0}\F$ and an object $\xi \in \F_{\xi_0}(\dual)$ in the isomorphism class $a$, the image $d_{\xi_0}H(a)$ is the isomorphism class of $H(\xi) \in \G_{H(\xi_0)}(\dual)$.

As one expects the differential of the composite of two morphisms of deformation categories is the composite of the differentials, as is very easy to see. Moreover if a morphism $H\colon \F\to \F$ is isomorphic to the identity, then the differential $d_{\xi_0}H\colon T_{\xi_0}\F\to T_{H(\xi_0)}\F$ is an isomorphism.

If in particular $H\colon \F\to \G$ is an equivalence, then $d_{\xi_0}H\colon T_{\xi_0}\F\to T_{H(\xi_0)}\G$ is an isomorphism too. This is because in this case $H$ has a quasi-inverse $K\colon \G \to \F$, and the composites $H\circ K$ and $K \circ H$ are isomorphic to the identities; this implies that
$$
d_{H(\xi_0)}K\circ d_{\xi_0}H\colon T_{\xi_0}\F\arr T_{\xi_0}\F
$$
and
$$
d_{\xi_0}H\circ d_{H(\xi_0)}K\colon T_{H(\xi_0)}\G \arr T_{H(\xi_0)}\G
$$
are isomorphisms, and so $d_{\xi_0}H$ will be too. Here actually $K(H(\xi_0))$ needs only to be isomorphic to $\xi_0$, so we use the isomorphism of Remark~\ref{tangiso} to identify $T_{K(H(\xi_0))}\F$ and $T_{\xi_0}\F$ in the composites above.

\subsection{Extensions of algebras and liftings}\label{sec.extensions}

Now we pause shortly to state some standard facts about extensions of algebras that will be used very frequently from now on.

\begin{defin}
Let $A \in \artl$. An \gr{extension} of $A$ is a surjection $A'\to A$ in $\artl$ with square-zero kernel $I=\ker\phi \subseteq A'$. We also say that $A'\to A$ is an extension of $A$ by $I$.
\end{defin}
An extension as above is usually pictured as the exact sequence of $\Lambda$-modules
$$
\xymatrix{
0\ar[r] & I\ar[r] & A'\ar[r] & A\ar[r] & 0.
}
$$
In this situation $I$ is an $A$-module in a natural way: given $a \in A$ and $i \in I$ we take an element $a' \in A'$ in the preimage of $a$ and define $a\cdot i$ as $a'i \in I$. This is well defined because $I^2=(0)$.

\begin{examp}
If $M$ is an $A$-module, there is a \gr{trivial extension} of $A$ by $M$, which we obtain by considering $A\oplus M$ as an $R$-algebra by the trivial algebra structure mentioned in the introduction (the product is defined by $(a,m)(a',m')=(aa',am'+a'm)$, so that in particular $M^2=(0)$). The homomorphism $A\oplus M\to A$ is the projection. 

In particular if $A=k$ is a field and $M\simeq k$, we obtain the $k$-algebra $k[t]/(t^2)\simeq k\oplus k\eps$, which is the usual ring of dual numbers $\dual$ (where $\eps=[t]$).
\end{examp}

The following fact will be used later.

\begin{prop}\label{ext.der}
Let $A'\to A$ be an extension in $\artl$ with kernel $I$, $B$ a $\Lambda$-algebra, and $f$, $g\colon B\to A'$ two homomorphisms of $\Lambda$-algebras such that the composites with $A'\to A$ coincide. Then the difference $f-g\colon B\to I$ is a $\Lambda$-derivation.

Conversely, if $f\colon B \to A'$ is a homomorphism of $R$-algebras and $d\colon B\to I$ is a $\Lambda$-derivation, then the map $f+d\colon B\to A'$ is a homomorphism of $R$-algebras such that the composite with $A'\to A$ coincides with $B\xrightarrow{f} A'\to A$.
\end{prop}

The proof is easy and left to the reader.

Suppose that we have two extensions of $R$-algebras $A'\to A$ and $B'\to B$, with kernels $I$ and $J$ respectively, and a homomorphism of $R$-algebras $\phi\colon A'\to B'$, such that $\phi(I)\subseteq J$. Then $\phi$ will induce $\overline{\phi}\colon A\to B$ and $\phi|_I\colon I\to J$, which fit together with $\phi$ in a commutative diagram.

\begin{defin}
A \gr{homomorphism} between two extensions of $R$-algebras $A'\to A$ and $B'\to B$ with kernels $I$ and $J$ respectively is a triplet of homomorphisms $(f,g,h)$, where $f\colon I\to J, g\colon A'\to B', h\colon A\to B$, such that the diagram
$$
\xymatrix{
0\ar[r] & I\ar[r]\ar[d]^f & A'\ar[r]\ar[d]^g & A\ar[r]\ar[d]^h & 0\\
0\ar[r] & J\ar[r] & B'\ar[r] & B\ar[r] & 0
}
$$
is commutative.
\end{defin}

So a homomorphism $\phi$ as above induces a morphism $(\phi|_I,\phi,\overline{\phi})$ between the two extensions.

\begin{defin}
A \gr{splitting} of an extension of $R$-algebras $A'\to A$ is a homomorphism of $R$-algebras $\phi\colon A\to A'$ such that the composite $A\xrightarrow{\phi} A'\to A$ is the identity.
\end{defin}

Standard arguments show that an extension admits a splitting if and only if it is isomorphic to a trivial extension.

The following type of extensions will play a particularly important role.

\begin{defin}
An extension $A'\to A$ in $\artl$ is said to be \gr{small} if the kernel $I$ is annihilated by the maximal ideal $\m_{A'}$, so that it is naturally a $k$-vector space.

A small extension is called \gr{tiny} if $I\simeq k$ as a $k$-vector space, or equivalently if $I$ is principal and nonzero.
\end{defin}

\begin{prop}\label{comp.small}
A surjection  $A'\to A$ in $\artl$ can be factored as a composite of tiny extensions
$$
A'=A_0\to A_1 \to \cdots\to A_n=A\,.
$$
\end{prop}

\begin{ex}
Prove Proposition~\ref{comp.small}.
\end{ex}

Now we come to liftings of objects of a deformation category. The idea is that if we want to study the deformations over $A \in \artl$ of a given object over $k$, we should do this inductively using the factorization of the surjection $A\to k$ given by the preceding Proposition to reduce to the case of small extensions.

\begin{defin}
Let $\fib$ be a deformation category, $\phi\colon A'\to A$ a surjection in $\artl$, and $\xi \in \F(A)$. A \gr{lifting} of $\xi$ to $A'$ is an arrow $\xi\to \xi'$ over $\phi$.
\end{defin}

Equivalently, a lifting of $\xi$ to $A'$ is an object $\xi' \in \F(A')$ together with an isomorphism of its pullback $\phi_*(\xi')$ with $\xi$ in $\F(A)$. Sometimes we will refer to a lifting only by means of the object $\xi'$ over $A'$, leaving the arrow from $\xi$ understood.

Generalizing the construction of the category $\F_{\xi_0}$, it is easy to see that given $\phi$ and $\xi$ as above, the liftings of $\xi$ to $A'$ are the objects of a category $\Lif(\xi,A')$, in which arrows from $f\colon \xi\to \xi'$ to $g\colon \xi\to \xi''$ are arrows $h\colon \xi'\to \xi''$ of $\F(A)$ such that $h\circ f=g$. We will call $\lif(\xi,A')$ the set of isomorphism classes of liftings of $\xi$ to $A'$.

Both $\Lif(\xi,A')$ and $\lif(\xi,A')$ clearly depend also on the homomorphism $A'\to A$, but we will not specify it in the notation, since it will always be clear from the context which homomorphism we are considering.

In the following we will make some constructions starting with an isomorphism class $[\xi']$ of a lifting and possibly pick one of its elements in the process, without mentioning that the final result will not depend on this choice (because we will be often taking isomorphism classes again in the end).

In particular if we have an element $a \in I\otimes_k T_{\xi_0}\F$, we will also write $a$ for an object of $\F_{\xi_0}(k \oplus I)$ belonging to the isomorphism class $a$.

\subsection{Actions on liftings}\label{actionsec}

The tangent space gives some control on the set of isomorphism classes of liftings of objects of $\F$ along small extensions, as the following Theorem shows.

\begin{thm}\label{actionthm}
Let $\fib$ be a deformation category, $A'\to A$ a small extension with kernel $I$, and take $\xi_0 \in \F(k)$, $\xi \in \F_{\xi_0}(A)$. If\/ $\lif(\xi,A')$ is not empty, then there is a free and transitive action of $I\otimes_k T_{\xi_0}\F$ on it.
\end{thm}

\begin{proof}
Let $\xi\to\xi_1'$ and $\xi\to\xi_2'$ be two liftings of $\xi$ to $A'$, and notice that together they give an object of the category $\F(A')\times_{\F(A)}\F(A')$. By RS, they give rise to a lifting $\xi\to\{\xi_1',\xi_2'\}$ of $\xi$ to the fibered product $A'\times_A A'$, and this construction gives a bijection between pairs of isomorphism classes of liftings of $\xi$ to $A'$ and liftings of $\xi$ to $A'\times_A A'$.

We have an isomorphism of rings $f\colon A'\times_A A' \simeq A'\oplus I$ given by $f(a_1,a_2)=(a_1,a_2-a_1)$, which commutes with the projections on the first factor $A'$. Moreover, if we call $\pi\colon A\to k$ the quotient map, there is an isomorphism $A' \oplus I\simeq A'\times_k k \oplus I$, defined by $(a,v) \mapsto (a,\pi(a)\oplus v)$, which also commutes with the projections on the first factor, and so as before we have a bijection between the isomorphism classes of liftings of $\xi$ to $A' \oplus I$ and pairs of isomorphism classes of liftings of $\xi$ to $A'$ and of $\xi_0$ to $k \oplus I$.

This gives a bijection $\Phi$, defined as the composite
$$
\lif(\xi,A')\times \lif(\xi,A')\longrightarrow \lif(\xi,A')\times \lif(\xi_0,k \oplus I)\simeq \lif(\xi,A')\times(I\otimes_k T_{\xi_0}\F)\,.
$$
By construction if $\pi_1\colon \lif(\xi,A')\times \lif(\xi,A')\to \lif(\xi,A')$ is the projection to the first factor, then $\pi_1\circ \Phi^{-1}$ is also the projection to the first factor $\lif(\xi,A')\times(I\otimes_k T_{\xi_0}\F)\to\lif(\xi,A')$. Let us consider now the function
$$
\mu=\pi_2\circ \Phi^{-1}\colon \lif(\xi,A')\times(I\otimes_k T_{\xi_0}\F)\arr \lif(\xi,A')
$$
where $\pi_2\colon \lif(\xi,A')\times \lif(\xi,A')\to \lif(\xi,A')$ is the projection on the second factor. It is easy to check that $\mu$ gives an action of $I\otimes_k T_{\xi_0}\F$ on $\lif(\xi,A')$ (we leave the details to the reader), and this is free and transitive since $\Phi$ is bijective.
\end{proof}

\begin{rmk}\label{action}
From now on we will drop the notation $\mu\colon \lif(\xi,A')\times(I\otimes_k T_{\xi_0}\F)\to \lif(\xi,A')$ for the action, and we will simply write it by addition, as $([\xi], g) \arrto [\xi] + g$.

Furthermore, if $[\xi'],[\xi'']$ are two isomorphism classes of liftings of $\xi \in \F_{\xi_0}(A)$ to $A'$, where $A'\to A$ is a small extension with kernel $I$, we will denote by $[\xi'] - [\xi'']$ the element $g \in I \otimes_k T_{\xi_0}\F$ such that $[\xi''] + g =[\xi']$.

\end{rmk}

The following corollary is a straightforward application of Proposition~\ref{actionthm}, easily proved using \ref{comp.small} and induction.

\begin{cor}\label{trivial}
Let $\fib$ be a deformation category, and $\xi_0 \in \F(k)$. If $T_{\xi_0}\F=0$, then there is at most one isomorphism class in $\F_{\xi_0}(A)$, for every $A \in \artl$.
\end{cor}

This action has two natural functoriality properties, which we now discuss. The first one is functoriality with respect to the small extension. Let $\fib$ be a deformation category, and $A'\to A$, $B'\to B$ two small extensions, with kernels $I\subseteq A'$ and $J\subseteq B'$. Suppose we also have a homomorphism $\phi\colon A'\to B'$ such that $\phi(I)\subseteq J$, and thus inducing $\overline{\phi}\colon A\to B$ and $\phi|_I\colon I\to J$. In other words, we have a homomorphism of extensions
$$
\xymatrix{
0\ar[r] & I\ar[r]\ar[d]^{\phi|_I} & A' \ar[r]\ar[d]^\phi & A \ar[r]\ar[d]^{\overline{\phi}} & 0\\
0\ar[r] & J\ar[r] & B' \ar[r] & B\ar[r] & 0.
}
$$

Let us also have $\xi_0 \in \F(k)$, $\xi \in \F_{\xi_0}(A)$ and assume $\lif(\xi,A')$ is nonempty (so that $\lif(\overline{\phi}_*(\xi),B')$ is nonempty). We have a $k$-linear map
$$
\phi|_I\otimes \id \colon  I\otimes_k T_{\xi_0}\F\arr J\otimes_k T_{\xi_0}\F
$$
(which corresponds to the pullback function $\lif(\xi_0,k \oplus I)\to \lif(\xi_0,k\oplus J)$ induced by $\id\oplus \phi|_I$), and a pullback function on isomorphism classes of liftings
$$
\phi_*\colon \lif(\xi,A') \to \lif(\overline{\phi}_*(\xi),B')\,.
$$

\begin{prop}\label{funct}
We have
$$
\phi_*([\xi'] +  a)=\phi_*([\xi']) + (\phi|_I\otimes \id)(a)
$$
for every $a \in I\otimes_k T_{\xi_0}\F$ and $[\xi'] \in \lif(\xi,A')$.
\end{prop}

Using the notation of Remark~\ref{action}, we can equivalently say that if $[\xi'],[\xi''] \in \lif(\xi,A')$ we have
$$
[\phi_*(\xi')] -[\phi_*(\xi'')]=(\phi|_I\otimes \id)([\xi']-[\xi''])\,.
$$

\begin{examp}\label{universal-first-order}
Suppose that $\fib$ is a deformation category and $\xi_0 \in \F(k)$. Set $T \eqdef T_{\xi_{0}}\F$. Consider liftings of $\xi_{0}$ to artinian $\Lambda$-algebras of the form $k \oplus V$, where $V$ is a finite-dimensional $k$-vector space. There is a trivial lifting $\xi_{0}^{V}$, coming from functoriality from the homomorphism $k \arr k \oplus V$. From Theorem~\ref{actionthm} we see that $\lif(\xi_{0}, k\oplus V)$ is in bijective correspondence with $V \otimes_k T$, by applying each element of $V \otimes_k T$ to $\xi_{0}^{V}$.

In particular, when $V = T^{\vee}$ we have a bijective correspondence of $\lif(\xi_{0}, k\oplus T^{\vee})$ with $T^\vee \otimes_k T$; we denote by $\xi^{(1)}$ a lifting of $\xi_{0}$ that corresponds to the identity in $T^\vee \otimes_k T = \Hom_{k}(T, T)$. This $\xi^{(1)}$ is unique up to (a non unique) isomorphism; we call it \emph{the universal first order lifting} of $\xi_{0}$.

This name is justified by the following universal property: for each lifting $\xi$ of $\xi_{0}$ to an artinian $k$-algebra of the form $k \oplus V$, then there exists a unique $k$-linear homomorphism $T^{\vee} \arr V$, such that the corresponding homomorphism of $\Lambda$-algebras $k \oplus T^{\vee} \arr k \oplus V$ sends $\xi^{(1)}$ into a lifting of $\xi_0$ isomorphic to $\xi$. This follows easily from functoriality.
\end{examp}

The second one is functoriality with respect to the deformation category. Let $\fib$ and $\G\to \artl^\op$ be two deformation categories with a morphism $F\colon \F\to \G$, $A'\to A$ a small extension with kernel $I$, and let $\xi_0 \in \F(k)$, $\xi \in \F_{\xi_0}(A)$. Assume also that $\lif(\xi,A')$ is nonempty (so that $\lif(F_{\xi_0}(\xi),A')$ is also nonempty).

There is a $k$-linear map
$$
\id \otimes d_{\xi_0}F\colon  I\otimes_k T_{\xi_0}\F\arr I \otimes_k T_{F(\xi_0)}\G
$$
induced by the differential $d_{\xi_0}F\colon T_{\xi_0}\F\to T_{F(\xi_0)}\G$, and we still denote by
$$
F\colon \lif(\xi,A')\arr \lif(F(\xi),A')
$$
the induced function on isomorphism classes of liftings.

\begin{prop}\label{actionfunct}
We have
$$
F([\xi'] +  a)=F([\xi']) + (\id \otimes d_{\xi_0}F)(a)
$$
for every $a \in I\otimes_k T_{\xi_0}\F$ and $[\xi']\in \lif(\xi,A')$.
\end{prop}

As before we can reformulate this result using the notation of Remark~\ref{action}, and obtain
$$
[F(\xi')] - [F(\xi'')])=(\id\otimes d_{\xi_0}F)([\xi'] - [\xi''])
$$
for every $[\xi'],[\xi''] \in \lif(\xi,A')$.

The proofs of the last two Propositions are a simple matter of drawing diagrams and chasing pullbacks.

\begin{ex}
Prove Propositions \ref{funct} and \ref{actionfunct}.
\end{ex}

There is a generalization of the previous constructions, which we will need later to state the Ran--Kawamata Theorem about vanishing of obstructions (Theorem~\ref{rk}). Given $A \in \artl$, $\xi \in \F(A)$, we consider the liftings of $\xi$ to trivial $\Lambda$-algebras $A \oplus M$ where $M \in \fmod$ (and the homomorphism $A \oplus M\to A$ is the projection).\label{generalization}

We have a functor $F_\xi\colon \fmod\to \set$ defined on objects by
$$
F_\xi(M)=\{ \text{isomorphism classes of liftings of } \xi \text{ to } A \oplus M\}
$$
and sending an $A$-linear map $M\to N$ into the pullback function $F_\xi(M)\to F_\xi(N)$.

Since $\fib$ satisfies RS, one can readily show that the functor $F_\xi$ preserves finite products, and so by Proposition~\ref{app.lifting} it has a canonical lifting $\fmod\to\mod$, which we still call $F_\xi$. Unlike the case $A = k$, in the present situation the functor $F_\xi$ need not be exact. Nevertheless, one can easily prove using RS that it is half-exact, that is, if
$$
\xymatrix{
0\ar[r]& M'\ar[r] & M\ar[r] & M''\ar[r] & 0
}
$$
is an exact sequence of finitely generated $A$-modules, then the sequence
$$
\xymatrix{
F_\xi(M')\ar[r] & F_\xi(M)\ar[r] & F_\xi(M'')
}
$$
is exact.

The following Proposition can be proved in the exact same way as Theorem~\ref{actionthm}.

\begin{prop}
If $A'\to A$ is a surjection in $\artl$ with kernel $I$ such that $I^2=(0)$ (so that $I$ is an $A$-module), and $\xi \in \F(A)$, then $\lif(\xi,A')$ is either empty, or there is a free and transitive action of $F_\xi(I)$ on it.
\end{prop}

\subsection{Examples}\label{tangent-examples}

Now we calculate the tangent space in each of the examples introduced in Section~\ref{deformation.cat.cap}, and give an application to infinitesimal deformations of smooth hypersurfaces in $\P^n_k$.

\subsubsection{Schemes}

We first consider the deformation category $\catdef \to \artl^\op$ corresponding to deformations of schemes.

\begin{thm}\label{kodaira}
Let $X_0$ be a reduced and generically smooth scheme of finite type over $k$. There is an isomorphism (sometimes called the \gr{Kodaira--Spencer} correspondence)
$$
T_{X_0}\catdef\simeq\Ext^1_{\O_{X_0}}(\diff_{X_0},\O_{X_0})\,.
$$
\end{thm}

\begin{proof}
Call $F\colon \fvect \to \set$ the functor defined on objects by
$$
F(V)=\{\text{isomorphism classes of objects in } \catdef_{X_0}(k \oplus V)\}
$$
and that sends a $k$-linear map $V\to W$ into the pullback function $F(V)\to F(W)$. We need to construct a functorial bijection
$$
F(V)\simeq V\otimes_k \Ext^1_{\O_{X_0}}(\diff_{X_0},\O_{X_0})
$$
that will give a $k$-linear natural isomorphism between the functors $F$ and $-\otimes_k  \Ext^1_{\O_{X_0}}(\diff_{X_0},\O_{X_0})$, so in particular we will get an isomorphism
$$
T_{X_0}\catdef=F(k)\simeq\Ext^1_{\O_{X_0}}(\diff_{X_0},\O_{X_0})\,.
$$
We sketch briefly the main steps of the construction.

First of all we define a function
$$
\phi_V\colon F(V)\arr V\otimes_k\Ext^1_{\O_{X_0}}(\diff_{X_0},\O_{X_0})\,.
$$
Take an object $X \in \catdef_{X_0}(k \oplus V)$, which is a flat scheme of finite type over $k \oplus V$ with an isomorphism
$$
X\times_{\spec k \oplus V}\spec k\simeq X_0
$$
(in particular $\O_X\otimes_{k \oplus V} k\simeq \O_{X_0}$), and consider the conormal sequence of $X_0\subseteq X$
$$
\xymatrix{
V\otimes_k \O_{X_0}\ar[r]^-d & \diff_X|_{X_0}\ar[r] & \diff_{X_0}\ar[r] & 0
}
$$
where $d$ is the homomorphism induced by the universal derivation $\O_X\to \diff_X$, and we identify $V\otimes_k \O_{X_0}$ with the sheaf of ideals of $X_0$ in $X$. Using the fact that $X_0$ is generically smooth and reduced one sees that in this case $d$ is injective, and so we have an exact sequence of $\O_{X_0}$-modules
$$
\xymatrix{
0\ar[r] &V\otimes_k \O_{X_0}\ar[r] & \diff_X|_{X_0}\ar[r] & \diff_{X_0}\ar[r] & 0
}
$$
whose isomorphism class in an element of $V\otimes_k\Ext^1_{\O_{X_0}}(\diff_{X_0},\O_{X_0})$. It is also clear that isomorphic objects of $\catdef_{X_0}(k \oplus V)$ will give isomorphic extensions, and so we have our function
$$
\phi_V\colon F(V)\to V\otimes_k\Ext^1_{\O_{X_0}}(\diff_{X_0},\O_{X_0})\,.
$$

Now we construct a function
$$
\psi_V\colon V\otimes_k\Ext^1_{\O_{X_0}}(\diff_{X_0},\O_{X_0})\to F(V)
$$
in the other direction. We start then with an element of $\Ext^1_{\O_{X_0}}(\diff_{X_0},V\otimes_k\O_{X_0})$, represented by an extension
$$
\xymatrix{
0\ar[r] & V \otimes_k \O_{X_0}\ar[r] & E \ar[r]^f & \diff_{X_0}\ar[r] & 0
}
$$
of $\O_{X_0}$-modules. We define then a sheaf $\O(E)$ on $|X|$ by
$$
\O(E)=\O_{X_0}\times_{\diff_{X_0}} E\subseteq \O_{X_0}\oplus E
$$
where the morphism $\O_{X_0}\to \diff_{X_0}$ is the universal derivation $d_0$.

This is a priori only a sheaf of $k$-vector spaces, but one can see that $\O(E)$ has a natural structure of sheaf of (flat) $k \oplus V$-algebras. Moreover its stalks are local rings, so that $X(E)=(|X_0|,\O(E))$ is a locally ringed space, and it is a scheme (flat over $k \oplus V$), because if $U=\spec A$ is an open affine subscheme of $X_0$, then $(U,\O(E)|_U)$ is isomorphic to $\spec(A\times_{\diff_{A}}E(U))$.

We note now that $X(E)\times_{\spec k \oplus V} \spec k \simeq X_0$, and that this construction is independent (up to isomorphism) of the representative chosen for the element of $\Ext^1_{\O_{X_0}}(\diff_{X_0},V\otimes_k\O_{X_0})$, so we get a function
$$
\psi_V\colon V\otimes_k\Ext^1_{\O_{X_0}}(\diff_{X_0},\O_{X_0})\arr F(V)\,.
$$

To complete the proof, we leave it to the reader to check that $\phi_V$ and $\psi_V$ are inverse to each other, and that $\phi_V$ is functorial in $V$, or in other words, given a $k$-linear map $f\colon V \to W$, the diagram
$$
\xymatrix@C+20pt{
F(V)\ar[r]^-{\phi_V}\ar[d]_{(\id\oplus f)_*} & V\otimes_k\Ext^1_{\O_{X_0}}(\diff_{X_0},\O_{X_0})\ar[d]^{f\otimes \id}\\
F(W)\ar[r]^-{\phi_W} & W\otimes_k\Ext^1_{\O_{X_0}}(\diff_{X_0},\O_{X_0})
}
$$
is commutative.
\end{proof}

Given $\xi \in \F_{\xi_0}(\dual)$, the element of $\Ext^1_{\O_{X_0}}(\diff_{X_0},\O_{X_0})$ associated with $\xi$ is sometimes called its \gr{Kodaira--Spencer class}, from the names of the two mathematicians who first studied the deformation theory of complex manifolds.

\begin{rmk}\label{nonsing}
If $X_0$ is also smooth over $k$, then the tangent space $T_{X_0}\catdef$ is isomorphic to $\Ext^1_{\O_{X_0}}(\diff_{X_0},\O_{X_0})\simeq \H^1(X_0,T_{X_0})$, where $T_{X_0}=\diff_{X_0}^\vee$ is the tangent sheaf of $X_0$.

In particular we see that every first-order deformation of a smooth and affine variety $X_0$ is trivial, because in this case $\H^1(X_0,T_{X_0})$ vanishes. So smooth affine varieties are rigid, in the sense that they do not have any non-trivial first-order deformation.
\end{rmk}

In the general case, in which $X_0$ is not necessarily reduced and generically smooth, one has to resort to the cotangent complex $L_{X_0/k}$ associated with the structure morphism $X_0\to\spec k$; the general result, which can be found in \cite[III, 2.1.7]{Ill}, states that there is a canonical isomorphism
$$
T_{X_0}\catdef\simeq \Ext^1_{\O_{X_0}}(L_{X_0/k},\O_{X_0})\,.
$$
This implies Theorem~\ref{kodaira}, since if $X_0$ is reduced and generically smooth the cotangent complex is the sheaf $\diff_{X_0}$.

\begin{ex}\label{affine.hypersurf}
A simple example in which one can calculate the Kodaira--Spencer correspondence explicitly is that of hypersurfaces in $\A^n_k$. Take a hypersurface $X_0=\spec A\subseteq \A^n_k$, where $A=k[x_1,\hdots,x_n]/(f)$.

\begin{enumeratei}

\item Show that $\Ext^1_{\O_{X_0}}(\diff_{X_0},\O_{X_0})\simeq k[x_1,\hdots,x_n]/(f,\partial f/\partial x_1,\hdots, \partial f/\partial x_n)$.

\item Show that if $X \in \catdef_{X_0}(\dual)$, then $X\simeq\spec(\dual[x_1,\hdots,x_n]/(f+\eps g))$ for some $g \in k[x_1,\hdots,x_n]$.

\item Show that the Kodaira--Spencer class of $X$ in $T_{X_0}\catdef\simeq \Ext^1_{\O_{X_0}}(\diff_{X_0},\O_{X_0})$ is the class of $g$ in $k[x_1,\hdots,x_n]/(f,\partial f/\partial x_1,\hdots, \partial f/\partial x_n)$.

\end{enumeratei}

\end{ex}

\subsubsection{Smooth varieties}\label{smoothvar}

Now suppose $X_0$ is a smooth variety over $k$. We describe the isomorphism
$$
T_{X_0}\catdef\simeq \Ext^1_{\O_{X_0}}(\diff_{X_0},\O_{X_0})\simeq \H^1(X_0,T_{X_0})
$$
in a more explicit way, using \v{C}ech cohomology.

Consider an object $X \in \catdef_{X_0}(\dual)$ and take an open affine cover $\mathcal{U}=\{U_i\}_{i \in I}$ of $X_0$. Because of Remark~\ref{nonsing}, the induced deformation $X|_{U_i}$ of $U_i$ is trivial for every index $i$, and from this we get a collection $\{\theta_i\}_{i \in I}$ of isomorphisms of deformations 
$$
\theta_i\colon U_i \times_{\spec k} \spec\dual \to X|_{U_i}\,.
$$
Now set $\theta_{ij} \eqdef (\theta_i|_{U_{ij}})^{-1} \circ (\theta_j|_{U_{ij}})$; the $\theta_{ij}$ are automorphisms of the trivial deformations $U_{ij}\times_{\spec k} \spec\dual$ that restrict to the identity on the closed fiber $U_{ij}$.

Now we need to use results from Section~\ref{inf.aut}. Precisely, it follows from Proposition~\ref{infinit} that there is an isomorphism between the group of automorphisms of the deformation $U_{ij}\times_{\spec k} \spec\dual$ of $U_{ij}$ that induce the identity on the closed fiber, and the group $\der_k(B_{ij},B_{ij})=\Gamma(U_{ij},T_{X_0})$, where $U_{ij}=\spec B_{ij}$.

Using this, for each $\theta_{ij}$ we get an associated element $d_{ij} \in \Gamma(U_{ij},T_{X_0})$. Furthermore, on $U_{ijk}$ we have for each triplet of indices the cocycle condition
$$
\theta_{ij}\circ\theta_{jk}=\theta_{ik}
$$
on automorphisms, which translates into the relation $d_{ij}+d_{jk}-d_{ik}=0$. This in turn says that the family $\{d_{ij}\}_{i,j \in I}$ is a \v{C}ech 1-cocycle for $T_{X_0}$, and so defines an element $[\{d_{ij}\}_{i,j \in I}]$ of $\check{H}^1(\mathcal{U},T_{X_0})\simeq \H^1(X_0,T_{X_0})$.

We it leave to the reader to check that this element does not depend on the open affine cover $\mathcal{U}$, and that isomorphic deformations give the same cohomology class. One can also check that this construction gives the same element one gets by using Theorem~\ref{kodaira} and the canonical isomorphism $\Ext^1_{\O_{X_0}}(\diff_{X_0},\O_{X_0})\simeq \H^1(X_0,T_{X_0})$.

The inverse function is defined as follows: given an element of $\H^1(X_0,T_{X_0})$, we can represent it as a $1$-cocycle $\{d_{ij}\}_{i,j \in I}$ for some open affine cover $\mathcal{U}=\{U_i\}_{i \in I}$ of $X_0$. The $d_{ij}$ correspond to automorphisms of the trivial deformation $U_{ij}\times_{\spec k} \spec\dual$, and the cocycle condition says exactly that these automorphisms can be used to glue the schemes $U_i\times_{\spec k} \spec\dual$ along the subschemes $U_{ij}\times_{\spec k} \spec\dual$, to get a flat scheme $X$ over $\dual$. It is easy to see that this construction does not depend (up to isomorphism) on the affine cover, and on the cocycle we choose in the cohomology class. Finally it is clear that the two constructions are inverse to each other, so we have the bijection above.

\subsubsection{Closed subschemes}

Next we consider the case of deformations of closed subschemes. Given an object of $\hilb^X(k)$, i.e. a closed subscheme $Z_0\subseteq X_0=X\times_{\spec \Lambda}\spec k$, call $I_0$ the ideal sheaf of $Z_0$ in $X_0$, and consider the normal sheaf $\mathcal{N}_0\eqdef \Homsh(I_0/I_0^2,\O_{Z_0})$.

\begin{thm}\label{closed}
There is an isomorphism
$$
T_{Z_0}\hilb^X\simeq \H^0(Z_0,\mathcal{N}_0) = \Hom_{\O_{Z_0}}(I_0/I_0^2,\O_{Z_0})\,.
$$
\end{thm}

\begin{proof}
We consider the functor $F\colon \fvect\to \set$ defined on objects by
$$
F(V)=\{\text{objects in } \hilb^X_{Z_0}(k \oplus V)\}
$$
and sending a $k$-linear map $V\to W$ into the associated pullback function $F(V)\to F(W)$. We will construct a functorial bijection
$$
F(V)\simeq  V\otimes_k\Hom_{\O_{X_0}}(I_0,\O_{Z_0})
$$
that will give a $k$-linear natural transformation, and in particular an isomorphism
$$
T_{Z_0}\hilb^X=F(k)\simeq \Hom_{\O_{X_0}}(I_0,\O_{Z_0})
$$
(notice that $\Hom_{\O_{X_0}}(I_0,\O_{Z_0})\simeq\Hom_{\O_{Z_0}}(I_0/I_0^2,\O_{Z_0})$).

Let us define a function
$$
\phi_V\colon F(V)\arr  V\otimes_k\Hom_{\O_{X_0}}(I_0,\O_{Z_0})\,.
$$
Take an object $Z \in \hilb^X_{Z_0}(k \oplus V)$, that is, a closed subscheme $Z\subseteq X_V$, where $X_V=X_0\times_{\spec k}\spec (k \oplus V)$ is the trivial deformation of $X_0$ over $k \oplus V$, which is flat over $k \oplus V$, and whose restriction to $X_{0}$ is $Z_0$; call $I\subseteq \O_{X_V}$ its sheaf of ideals.

By tensoring the exact sequence of $k \oplus V$-modules
$$
\xymatrix{
0\ar[r] & V \ar[r] & k \oplus V \ar[r] & k \ar[r] & 0
}
$$
with $\O_Z$ and $\O_{X_V}$, and
$$
\xymatrix{
0\ar[r] & I \ar[r] & \O_{X_V} \ar[r] & \O_Z \ar[r] & 0
}
$$
with $k$, we get a commutative diagram of $\O_{X_V}$-modules
\begin{equation}\label{morph}
\xymatrix{
 & & 0\ar[d] & 0\ar[d]&\\
 & & I \ar[r] \ar[d] & I_0 \ar[r] \ar[d]^{i} & 0 \\
0 \ar[r] & V\otimes_k \O_{X_0}\ar[r]\ar[d] & \O_{X_V} \ar[d]_q\ar[r]_{p_X} & \O_{X_0}\ar[r]\ar[d]^{q_0}\ar@/_10pt/[l]_\sigma & 0\\
0 \ar[r] & V\otimes_k \O_{Z_0}\ar[r] & \O_Z \ar[r]_{p_Z} \ar[d] & \O_{Z_0} \ar[r]\ar[d]  & 0\\
& & 0 & 0 &
}
\end{equation}
with exact rows and columns (by flatness). Since
$$
O_{X_V}\simeq \O_{X_0}\otimes_k k \oplus V\simeq \O_{X_0}\oplus (V\otimes_k \O_{X_0})
$$
as an $\O_{X_0}$-module, the map $p_X$ has an $\O_{X_0}$-linear section, which we call $\sigma$, simply defined by $\sigma(s)=(s,0)$, where $s$ is a section of $\O_{X_0}$.

The composite
$$
f\colon I_0\stackrel{i}{\larr} \O_{X_0}\stackrel{\sigma}{\larr} \O_{X_V} \stackrel{q}{\larr} \O_Z
$$
factors through $V\otimes_k \O_{Z_0}\to \O_Z$, so we have an $\O_{X_0}$-linear morphism $I_0\to V\otimes_k \O_{Z_0}$, which is then an element of $\Hom_{\O_{X_0}}(I_0,V\otimes_k\O_{Z_0})\simeq V\otimes_k\Hom_{\O_{X_0}}(I_0,\O_{Z_0})$. This gives us a function
$$
F(V)\arr  V\otimes_k\Hom_{\O_{X_0}}(I_0,\O_{Z_0})
$$
that we call $\phi_V$.

Now we define a function in the other direction. Take a homomorphism of $\O_{X_0}$-modules $f\colon I_0\to V\otimes_k\O_{Z_0}$, and consider the subsheaf $I_f$ of $\O_{X_V}\simeq \O_{X_0}\oplus (V\otimes_k \O_{X_0})$ given on an open subset $U$ of $|X_V|$ by
$$
I_f(U)=\{(s,t) \in \O_{X_V}(U) \colon  s \in I_0(U)\subseteq \O_{X_0}(U) \text{ and } f(s)+(\id\otimes q_0)(t)=0\}\,.
$$
where $q_0\colon \O_{X_0}\to \O_{Z_0}$ is the quotient map.

It is easy to check that $I_f$ is a coherent sheaf of ideals of $\O_{X_V}$, and so it defines a closed subscheme of $X_V$ that we call $Z_f\subseteq X_V$. Furthermore we have that
$$
Z_f\times_{\spec k \oplus V}\spec k \subseteq X_V\times_{\spec k \oplus V}\spec k\simeq X_0
$$
is the closed subscheme $Z_0$ and $Z_f$ is flat over $k \oplus V$. We only check the last assertion in detail: using the local flatness criterion (Theorem~\ref{local-flatness}), we have to show that $\Tor^{k \oplus V}_1(\O_Z,k)=0$.

We have an exact sequence of $\O_{X_V}$-modules
$$
\xymatrix{
0\ar[r] & I_f \ar[r] & \O_{X_V}\ar[r] & \O_Z \ar[r] & 0
}
$$
from which, taking the $\Tor$ exact sequence (tensoring with $k$), we get
$$
\xymatrix@C-10pt{
\Tor^{k \oplus V}_1(\O_{X_V},k) \ar[r] & \Tor^{k \oplus V}_1(\O_Z,k) \ar[r] & I_f \otimes_{k \oplus V} k \ar[r] & \O_{X_0}\ar[r] & \O_{Z_0}\ar[r] & 0.
}
$$
Since $X_V$ is flat over $k \oplus V$ we have $\Tor^{k \oplus V}_1(\O_{X_V},k)=0$, so we only need to show that the map $I_f\otimes_{k \oplus V} k \to I_0 \subseteq \O_{X_0}$ is injective, and this is clear.

This gives us a function
$$
V\otimes_k\Hom_{\O_{X_0}}(I_0,\O_{Z_0})\arr F(V)
$$
that we call $\psi_V$.

To conclude the proof, easy verifications show that $\phi_V$ and $\psi_V$ are inverse to each other, and that $\phi_V$ is functorial in $V$ (the reader can check the details).
\end{proof}

\subsubsection{Smooth hypersurfaces in $\P^n_k$}\label{smooth.proj}

We give an application of the previous constructions to deformations of smooth hypersurfaces of $\P^n_k$. Take $\Lambda=k$, and suppose we have a smooth hypersurface $Z_0\subseteq \P^n_k$ of degree $d$, with $n\geq 2$, $d\geq1$.

We can ask the following question: given a deformation $Z$ of $Z_0$ over $\spec A$, where $A\in \art$, can we find a closed immersion $Z\subseteq \P^n_{A}$ that extends $Z_0\subseteq \P^n_k$? Or more concisely: is every infinitesimal deformation of $Z_0$ embedded?

We can rephrase this question by using the forgetful morphism $F\colon \hilb^{\P^n_k} \to \catdef$ that sends an object $Y\subseteq \P^n_A$ of $\hilb^{\P^n_k}(A)$ into the flat morphism $Y\to \spec A$, which is an object of $\catdef(A)$, and acts on the arrows in the obvious way. The question above reads: is $F$ (restricted to deformations of $Z_0$) essentially surjective?

We can consider first the case of deformations over algebras of the form $k \oplus V$ for $V \in \fvect$. The existence of an immersion as above for every deformation of $Z_0$ over every $k \oplus V$ is equivalent to the surjectivity of the differential of the forgetful morphism at $Z_0$
$$
d_{Z_0}F\colon \H^0(Z_0,\mathcal{N}_0)\arr \H^1(Z_0,T_{Z_0}),
$$
since $\hilb^{\P^n_k}_{Z_0}(k \oplus V)\simeq V \otimes_k T_{Z_0}\hilb^{\P^n_k}$ and analogously for $\catdef$.

Now notice that the conormal sequence
$$
\xymatrix{
0\ar[r] &I_0/I_0^2 \ar[r]^-d & \diff_{\P^n_k}|_{Z_0} \ar[r] &\diff_{Z_0} \ar[r] & 0
}
$$
induced by the closed immersion $Z_0\subseteq \P^n_k$ (with sheaf of ideals $I_0$) is also exact on the left (since $Z_0$ is smooth), and dualizing it we get the exact sequence
$$
\xymatrix{
0\ar[r] &T_{Z_0} \ar[r] & T_{X_0}|_{Z_0} \ar[r] &\mathcal{N}_0 \ar[r] & 0.
}
$$
Taking cohomology we get a coboundary map $\delta\colon \H^0(Z_0,\mathcal{N}_0)\to \H^1(Z_0,T_{Z_0})$.

\begin{ex}
Show that the differential $d_{Z_0}F$ of the forgetful morphism coincides with the map $\delta$.
\end{ex}

Now we can study the surjectivity of the map $\delta$, using standard cohomology calculations.

\begin{prop}\label{hypers}
The map $\delta$ is surjective exactly in the following cases:
\begin{itemize}
\item $n=2$, $d\leq 4$.
\item $n=3$, $d\neq 4$.
\item $n\geq 4$, any $d$.
\end{itemize}
\end{prop}

\begin{rmk}
Of course, from a more advanced point of view than the one we are assuming in these notes, the exceptions $n = 2$, $d \geq 5$ and $n = 3$, $d =4$ are easily justified.

In the first case we are dealing with curves of genus $g = (d-1)(d-2)/2 \geq 6$; these curves have a moduli space of dimension $3g-3$, while plane curves form a locally closed subspace of dimension $(d^{2}+3d-16)/2$, which is smaller; thus a plane curve will have non-plane deformations.

Quartics in $\mathbb P^{3}$, on the other hand, are K3 surfaces, which are well known to have a 20-dimensional deformation space, while the ones who are algebraic of fixed genus only form an 19-dimensional family.
\end{rmk}

\begin{proof}
We start with a piece of the cohomology exact sequence
\begin{equation}\label{succ}
\xymatrix@C-5pt{
\H^0(Z_0,\mathcal{N}_0) \ar[r]^-\delta & \H^1(Z_0,T_{Z_0})\ar[r] & \H^1(Z_0,T_{\P^n_k}|_{Z_0})\ar[r] & \H^1(Z_0,\mathcal{N}_0)
}
\end{equation}
induced by the dual of the conormal sequence of $Z_0\subseteq \P^n_k$.

The first step is to prove
\begin{lemma}
$\coker(\delta)\simeq \H^2(\P^n_k,T_{\P^n_k}(-d))$.
\end{lemma}

\begin{proof}
First, we notice that $\H^1(Z_0,\mathcal{N}_0)=0$. This is because $\mathcal{N}_0\simeq \O_{Z_0}(d)$ (since $I_0\simeq \O_{\P^n_k}(-d)$), and $\H^1(\P^n_k,\O_{\P^n_k}(d))=\H^2(\P^n_k,\O_{\P^n_k})=0$, so from the cohomology exact sequence induced by
$$
\xymatrix{
0\ar[r] & \O_{\P^n_k} \ar[r]^-{f\cdot} & \O_{\P^n_k}(d) \ar[r] & \O_{Z_0}(d) \ar[r] & 0
}
$$
where $f$ is an equation for $Z_0$, we get $\H^1(Z_0,\mathcal{N}_0)=\H^1(\P^n_k,\mathcal{N}_0)=0$ (because $\mathcal{N}_0$ has support contained in $Z_0$).

From (\ref{succ}) we deduce then that $\coker(\delta)\simeq \H^1(Z_0,T_{\P^n_k}|_{Z_0})$, which is the same as $\H^1(\P^n_k,T_{\P^n_k}|_{Z_0})$, again because $T_{\P^n_k}|_{Z_0}$ has support contained in $Z_0$.

Tensoring the exact sequence
$$
\xymatrix{
0\ar[r] & \O_{\P^n_k}(-d) \ar[r]^-{f\cdot} & \O_{\P^n_k} \ar[r] & \O_{Z_0} \ar[r] & 0
}
$$
with $T_{\P^n_k}$ we get
\begin{equation}\label{tp}
\xymatrix{
0\ar[r] & T_{\P^n_k}(-d) \ar[r] & T_{\P^n_k} \ar[r] & T_{\P^n_k}|_{Z_0} \ar[r] & 0.
}
\end{equation}
Now we notice that $\H^i(\P^n_k,T_{\P^n_k})=0$ for $i\geq 1$: this follows from $\H^i(\P^n_k,\O_{\P^n_k})=\H^i(\P^n_k,\O_{\P^n_k}(1))=0$, using the cohomology exact sequence coming from the dual of the Euler sequence
$$
\xymatrix{
0\ar[r] & \O_{\P^n_k} \ar[r] & \O_{\P^n_k}(1)^{\oplus(n+1)} \ar[r] & T_{\P^n_k} \ar[r] & 0.
}
$$
From (\ref{tp}) we get an isomorphism
\begin{equation*}
\H^1(\P^n_k,T_{\P^n_k}|_{Z_0})\simeq \H^2(\P^n_k,T_{\P^n_k}(-d)).\qedhere
\end{equation*}
\end{proof}
To understand $\H^2(\P^n_k,T_{\P^n_k}(-d))$, we consider the exact sequence
$$
\xymatrix{
0\ar[r] & \O_{\P^n_k}(-d) \ar[r] & \O_{\P^n_k}(1-d)^{\oplus(n+1)} \ar[r] & T_{\P^n_k}(-d) \ar[r] & 0
}
$$
obtained by twisting the dual of the Euler sequence by $\O_{\P^n_k}(-d)$, and the following piece of its cohomology exact sequence
\begin{equation}\label{euler}
\xymatrix@C-7pt@R-9pt{
\H^2(\P^n_k,\O_{\P^n_k}(-d)) \ar[r] & \H^2(\P^n_k,\O_{\P^n_k}(1-d))^{n+1} \ar[r] & \H^2(\P^n_k,T_{\P^n_k}(-d)) \ar[dll]\\
\H^3(\P^n_k,\O_{\P^n_k}(-d))\ar[r] &\H^3(\P^n_k,\O_{\P^n_k}(1-d))^{n+1}.
}
\end{equation}

Now suppose $n\geq 4$. In this case we have
$$
\H^2(\P^n_k,\O_{\P^n_k}(1-d))^{n+1}=\H^3(\P^n_k,\O_{\P^n_k}(-d))=0
$$
and then from (\ref{euler}) we obtain $\coker(\delta)\simeq \H^2(\P^n_k,T_{\P^n_k}(-d))=0$, so that $\delta$ is surjective.

Now take $n=2$. Then we have $\H^3(\P^2_k,\O_{\P^2_k}(-d))=0$ and so again from (\ref{euler}) we get
$$
\H^2(\P^2_k,T_{\P^2_k}(-d))\simeq \coker\left(\H^2(\P^2_k,\O_{\P^2_k}(-d))\stackrel{\phi}{\longrightarrow} \H^2(\P^2_k,\O_{\P^2_k}(1-d))^{3}\right)
$$
where the map $\phi$ is induced by
$$
\begin{array}{rcl}
\O_{\P^2_k}(-d)& \longrightarrow & \O_{\P^2_k}(1-d)^{\oplus3}\\
\\
f & \longmapsto & f\begin{pmatrix} x_0 \\ x_1 \\ x_2 \end{pmatrix}
\end{array}
$$
where the $x_i$'s are homogeneous coordinates on $\P^2_k$ (seen as sections of the sheaf $\O_{\P^2_k}(1)$ of course).

By Serre's duality we see that $\H^2(\P^2_k,\O_{\P^2_k}(-d)) \simeq  \H^0(\P^2_k,\O_{\P^2_k}(d-3))^\vee$ and $\H^2(\P^2_k,\O_{\P^2_k}(1-d)) \simeq  \H^0(\P^2_k,\O_{\P^2_k}(d-4))^\vee $, and the adjoint map
$$
\H^0(\P^2_k,\O_{\P^2_k}(d-4))^3 \stackrel{\phi^\vee}{\longrightarrow} \H^0(\P^2_k,\O_{\P^2_k}(d-3))
$$
is given by scalar multiplication by the vector $(x_0,x_1,x_2)$.

Now $\coker\phi\simeq\ker\phi^\vee$. If $d\leq 3$ the source of $\phi^\vee$ is trivial, so certainly $\ker\phi^\vee=0$. If $d=4$, we have $\H^0(\P^2_k,\O_{\P^2_k}(d-4))\simeq k$, and the map $\phi^\vee$ is injective, because the sections $x_0,x_1,x_2$ are linearly independent over $\O_{\P^2_k}$, so that $\delta$ is surjective. On the other hand when $d\geq 5$ clearly $\phi^\vee$ is it not injective anymore, and so $\delta$ will not be surjective.

Now suppose that $n=3$. Then $\H^2(\P^2_k,\O_{\P^2_k}(-d))=0$, and using (\ref{euler}) once again we get
$$
\H^2(\P^n_k,T_{\P^2_k}(-d))\simeq \ker\left(\H^3(\P^3_k,\O_{\P^3_k}(-d))\stackrel{\phi}{\longrightarrow} \H^3(\P^3_k,\O_{\P^3_k}(1-d))^{4}\right)
$$
where $\phi$ is the analogue of the one we had in the preceding case. Again using Serre's duality we have to study
$$
\coker\Bigl(\H^0(\P^3_k,\O_{\P^3_k}(d-5))^4 \stackrel{\phi^\vee}{\longrightarrow} \H^0(\P^3_k,\O_{\P^3_k}(d-4))\Bigr)\,.
$$
If $d\leq 3$ the target is trivial, so that certainly $\coker(\phi^\vee)=0$, and if $d \geq 5$ the map $\phi^\vee$ is surjective, because every homogeneous polynomial of positive degree in variables $x_0,x_1,x_2,x_3$ can be written as a linear combination of the variables $x_i$'s, with homogeneous polynomials of one degree less as coefficients. In these cases then $\delta$ will be surjective.

The only case in which $\phi^\vee$ is not surjective (and so $\delta$ will not be too) is $d=4$, when the source is trivial and the target is not.
\end{proof}

We will examine the case $n=3$, $d=4$ further in Section~\ref{capform}, where it will give a counterexample to algebraizability of deformations of surfaces.

Now we go back to answering the first question we posed. We use the following result, which will be proved in Section~\ref{obs}, page~\pageref{prf-vanish}.

\begin{prop}\label{vanish}
If $Z_0$ is a smooth hypersurface of $\P^n_k$ of degree $d$, with $n\geq 1$ and $d\geq 1$, any object $Z\subseteq \P^n_A$ of $\hilb^{\P^n_k}_{Z_0}(A)$ can be lifted along any small extension $A'\to A$.
\end{prop}

We say that smooth hypersurfaces in $\P^n_k$ have unobstructed embedded deformations.

\begin{prop}\label{immers}
Let $Z_0$ be a smooth hypersurface of $\P^n_k$ of degree $d\geq 1$ and with $n\geq 2$, and $Z$ an object of $\catdef_{Z_0}(A)$, where $A \in \art$. Then there is a closed immersion $Z\subseteq \P^n_A$ inducing $Z_0 \subseteq \P^n_k$ in all cases except $n=2$, $d\geq 5$ and $n=3$, $d=4$.
\end{prop}

In other words the forgetful morphism $F\colon \hilb^{\P^n_k}_{Z_0}\to \catdef_{Z_0}$ is essentially surjective exactly in the cases above. 

\begin{proof}
We already know from the preceding discussion that in cases $n=2$, $d\geq 5$ and $n=3$, $d=4$ there are counterexamples.

Suppose then that we are not in one of the cases above, and take the given $Z\in \catdef_{Z_0}(A)$. We consider a factorization of the homomorphism $A\to k$ as a composite of small extensions
$$
A=A_0\to A_1 \to \hdots \to A_n=k\,.
$$
and proceed by induction on $n(A)$, the least $n$ with such a factorization.

If $n(A)=0$ there is nothing to prove. Suppose we know the result for $n(A)-1$, and consider the extension $A\to A_1$ with kernel $I$. The pullback $Z|_{A_1}\in \catdef_{Z_0}(A_1)$ of $Z$ to $A_1$ admits then a closed immersion $Z|_{A_1}\subseteq \P^n_{A_1}$ because of the induction hypothesis.

From the discussion above we also know that the differential of the forgetful morphism $d_{Z_0}F\colon T_{Z_0}\hilb^{\P^n_k}\to T_{Z_0}\catdef$ is surjective, and in particular
$$
\id\otimes d_{Z_0}F\colon I\otimes_k T_{Z_0}\hilb^{\P^n_k}\arr I \otimes_k T_{Z_0}\catdef
$$
will be surjective too.

Because of Proposition~\ref{vanish} we can find a lifting $Z'\subseteq \P^n_A$ of $Z|_{A_1}\subseteq \P^n_{A_1}$ to $A$; both $\lif(Z|_{A_1},A)$ and $\lif(Z_{A_1}\subseteq \P^n_{A_1},A)$ will then be nonempty, and by Theorem~\ref{actionthm} we have free and transitive actions on them, respectively of $I \otimes_k T_{Z_0}\catdef$ and $I\otimes_k T_{Z_0}\hilb^{\P^n_k}$.

The object $Z' \in \catdef_{Z_0}(A)$ is a lifting of $Z|_{A_1}$, as is $Z$, so by transitivity of the action we have an element $g \in I\otimes_k T_{Z_0}\catdef$ such that $[Z'] + g = [Z]$; take then $h \in I \otimes_k T_{Z_0}\hilb^{\P^n_k}$ such that $(\id\otimes d_{Z_0}F)(h)=g$.

Then using Proposition~\ref{actionfunct} we have
$$
F((Z'\subseteq \P^n_A) + h) = [Z'] + (\id\otimes d_{Z_0}F)(h) = [Z'] + g=[Z]\,.
$$
In other words the object $(Z'\subseteq \P^n_A) + h$ is (after possibly composing with an isomorphism of schemes over $\spec A$) a closed immersion $Z\subseteq \P^n_A$ that induces $Z_0\subseteq \P^n_k$ on the closed fiber, which is what we were looking for.
\end{proof}

The only things we really used in this proof were surjectivity of the differential and existence of liftings in the source deformation category. Every time these two facts hold in an abstract setting we can repeat the same argument to deduce that every object of the target deformation category is isomorphic to the image of an object of the source.

\subsubsection{Quasi-coherent sheaves}

Now suppose $X$ is a scheme over $\spec \Lambda$, and consider the deformation category $\qcoh^X\to\artl^\op$ of deformations of quasi\dash coherent sheaves on $X$. Let $\E_0 \in \qcoh^X(k)$.

\begin{prop}
There is an isomorphism
$$
T_{\E_0}\qcoh^X\simeq \Ext^1_{\O_{X_0}}(\E_0,\E_0)\,.
$$
\end{prop}

\begin{proof}
Consider the functor $F\colon \fvect\to \set$ defined on objects by
$$
F(V)=\{\text{isomorphism classes of objects in } \qcoh^X_{\E_0}(k \oplus V)\}
$$
and sending a $k$-linear map $f\colon V\to W$ into the corresponding pullback function $F(V)\to F(W)$. We show that there is a functorial bijection
$$
F(V)\simeq V\otimes_k \Ext^1_{\O_{X_0}}(\E_0,\E_0)
$$
that will give as usual a $k$-linear natural transformation, and in particular an isomorphism
$$
T_{\E_0}\qcoh^X=F(k)\simeq \Ext^1_{\O_{X_0}}(\E_0,\E_0)\,.
$$
To show this, the reader can check that the category $\qcoh^X_{\E_0}(k \oplus V)$ of quasi-coherent $\O_{X_V}$-modules $\E$ on $X_V=X_0\times_{\spec k}\spec k \oplus V$ with an isomorphism $\E\otimes_{k \oplus V} k\simeq \E_0$, is equivalent to the category whose objects are extensions of quasi-coherent $\O_{X_0}$-modules
$$
\xymatrix{
0\ar[r] & V\otimes_k \E_0 \ar[r] & \E \ar[r] & \E_0 \ar[r] & 0
}
$$
and arrows defined in the obvious way. This automatically gives us the bijection $\phi_V$ we want by taking isomorphism classes.

To conclude the proof, all is left is to show that $\phi_V$ is functorial in $V$, which we leave to the reader as usual.
\end{proof}

In particular if $\E_0$ is locally free, then we have
$$
T_{\E_0}\qcoh^X\simeq \Ext^1_{\O_{X_0}}(\E_0,\E_0)\simeq \H^1(X_0,\mathcal{E}nd_{\O_{X_0}}(\E_0))
$$
and moreover if $\E_0$ is invertible, then $\mathcal{E}nd_{\O_{X_0}}(\E_0)\simeq \E_0\otimes_{\O_{X_0}}(\E_0)^\vee \simeq \O_{X_0}$, so that
$$
T_{\E_0}\qcoh^X\simeq \H^1(X_0,\O_{X_0})
$$
which does not depend on $\E_0$. This particular case is more easily proved from the long cohomology sequence of the exact sequence of sheaves on $X_{0}$
   \[
   \xymatrix@C=30pt{
   1 \ar[r]& V\otimes_{k}\O_{X_{0}} \ar[r]^-{f \mapsto 1+f}&
   \O^{*}_{X_{V}} \ar[r]& \O^{*}_{X_{0}} \ar[r] & 0\,.
   }
   \]

\section{Infinitesimal automorphisms}\label{inf.aut}

In this section we introduce and discuss the so-called group (or space) of infinitesimal automorphisms of a deformation category at an object $\xi_0 \in \F(k)$, which is a feature we see only if we use categories instead of functors. It is particularly significant in the context of algebraic stacks, where wether the objects have or do not have non-trivial infinitesimal automorphisms determines if the stack is an Artin or a Deligne-Mumford stack; this is because there are non-trivial infinitesimal automorphisms if and only if the diagonal of the stack is ramified, and an algebraic stack is Deligne-Mumford if and only if its diagonal is unramified (see \cite[Th\'eor\`eme~8.1]{laumon-moret-bailly}).

We will see that this space gives a measure of the ``rigidity'' of a deformation problem, and tells us how far our deformation category is from its corresponding deformation functor. After the definition, we will examine some of its properties, and finally calculate it in some examples.

\subsection{The group of infinitesimal automorphisms}

Suppose $\fib$ is a deformation category, and $\phi\colon A'\to A$ is a small extension. Fix $\xi \in \F(A)$, and let $\xi' \in \F(A')$ be a lifting of $\xi$ to $A'$. We denote by $\aut_A(\xi)$ the set of automorphisms of the object $\xi$ in the category $\F(A)$, and analogously for $\xi'$. Recall that $\phi$ induces a pullback functor $\phi_*\colon \F(A')\to \F(A)$. In particular we have a ``restriction'' function $\aut_{A'}(\xi')\to \aut_A(\xi)$, which is a homomorphism of groups.

\begin{defin}
An \gr{infinitesimal automorphism} of $\xi'$ (with respect to $\xi$) is an automorphism of $\xi'$ in $\F(A')$ in the kernel of the homomorphism above (in other words, inducing the identity on $\xi$).
\end{defin}

Infinitesimal automorphisms are automorphisms of $\xi'$ in the category of liftings of $\xi$ over $A'$. They form a group, which we call the \gr{group of infinitesimal automorphisms} of $\xi'$ (with respect to $\xi$). We will see that this group depends only on $\ker\phi $ and on the pullback of $\xi$ to $\spec k$.

We start by taking $A=k$, $A'=\dual$ in the situation above. Notice that if $A$ is a $k$-algebra we have a trivial deformation of $\xi_0$ over $A$, which we denote by $\xi_0|_A$, given by the pullback of $\xi_0$ along the structure homomorphism $k\to A$.
\begin{defin}
The \gr{group of infinitesimal automorphisms} of $\xi_0$ is the group defined above, where we take $\xi=\xi_0 \in \F(k)$ and $\xi'=\xi_0|_{\dual} \in \F(\dual)$. We denote it by $\Inf(\xi_0)$.
\end{defin}
When we need to specify the category $\F$ in the notation, we will write $\Inf_{\xi_0}(\F)$ instead of $\Inf(\xi_0)$.

The group of infinitesimal automorphisms has also a canonical $k$-vector space structure, coming from the fact that it is the tangent space of a deformation category.

Consider the functor $\aut(\xi_0)\colon \art\to \set$ that sends an object $A \in \art$ into $\aut_A(\xi_0|_A)$, and an arrow $A'\to A$ into the function $\aut_{A'}(\xi_0|_{A'})\to \aut_A(\xi_0|_A)$ introduced above. This functor gives a category fibered in sets over $\art^\op$, and from the fact that $\F$ satisfies RS (precisely from the ``fully faithful'' part), we get that $\aut(\xi_0)$ does too.

Then we can consider the tangent space $T_{\id_{\xi_0}}\aut(\xi_0)$, which,  as a set, is easily seen to be exactly $\Inf(\xi_0)$ defined above.

\begin{ex}\label{sum.comp}
Prove that the addition coming from this vector space structure and the operation of composition of automorphisms coincide in $\Inf(\xi_0)$. In particular this will always be an abelian group with respect to composition.
\end{ex}

From the fact that $\Inf(\xi_0)$ is the tangent space of a deformation category, using Theorem~\ref{actionthm} we deduce the following corollary.
\begin{cor}\label{autom}
Let $\fib$ be a deformation category, $A'\to A$ a small extension with kernel $I$, and $f \in \aut(\xi_0)_{\id_{\xi_0}}(A)$; in other words $f$ is an automorphism of $\xi_0|_A$ that induces the identity on $\xi_0$. If\/ $\lif(f,A')$ is not empty, then there is a free and transitive action of $ I\otimes_k \Inf(\xi_0)$ on it.
\end{cor}

Now we state the aforementioned formula for the group of infinitesimal automorphisms of an arbitrary object of $\F$.

\begin{prop}\label{inf}
Let $\fib$ be a deformation category, $A'\to A$ a small extension with kernel $I$, $\xi_0 \in \F(k)$, $\xi \in \F_{\xi_0}(A)$ and $\xi'$ a lifting of $\xi$ to $A'$. Then we have an isomorphism
$$
\ker\left(\aut_{A'}(\xi')\to \aut_{A}(\xi)\right)\simeq I \otimes_k \Inf(\xi_0)\,.
$$
\end{prop}

\begin{ex}
Prove Proposition~\ref{inf}. (\emph{Sketch of proof}: define a functor
   \[
   \aut(\xi')\colon (\Art_{A'})\arr \set
   \]
that generalizes $\aut(\xi_0)$ above, and show that the resulting category fibered in sets satisfies RS. Then notice that $K=\ker\left(\aut_{A'}(\xi')\to \aut_{A}(\xi)\right)$ is the set of liftings of $\id_{\xi}$ to $A'$ in the category $\aut(\xi')$, and deduce that there is an isomorphism $I \otimes_k T_{\id_{\xi_0}}\aut(\xi') \simeq K$. To conclude show that $T_{\id_{\xi_0}}\aut(\xi')\simeq \Inf(\xi_0)$.)
\end{ex}

\begin{rmk}\label{autominf}
Suppose we have two liftings of $\xi$ to $A'$, say $\xi_1, \xi_2 \in \F_{\xi_0}(A')$, and an isomorphism of liftings $f\colon \xi_1\to \xi_2$. Take an infinitesimal automorphism $g_1 \in \aut_{A'}(\xi_1)$ of $\xi_1$, and consider $g_2=f\circ g_1\circ f^{-1} \in \aut_{A'}(\xi_2)$, which is an infinitesimal automorphism of $\xi_2$. Then it is clear from the preceding construction that the elements of $I\otimes_k \Inf(\xi_0)$ corresponding to $g_1$ and $g_2$ with respect to the isomorphism constructed above are the same.
\end{rmk}

As an application of Proposition~\ref{inf}, a straightforward induction using Proposition~\ref{comp.small} gives the following corollary.

\begin{cor}\label{equiv}
Let $\fib$ be a deformation category, and $\xi_0 \in \F(k)$. If $\Inf(\xi_0)=0$, then for every $A \in \artl$ and $\xi \in \F_{\xi_0}(A)$ the homomorphism $\aut_A(\xi)\to \aut_k(\xi_0)$ is injective.
\end{cor}

Furthermore we see that the group of infinitesimal automorphisms gives a measure of the ``rigidity'' (in the sense of closeness to being a functor) of our deformation problem.

\begin{prop}\label{equiv.rel}
Let $\fib$ be a deformation category and $\xi_0 \in \F(k)$. Then $\Inf(\xi_0)=0$ if and only if $\F_{\xi_0}\to \artl^\op$ is a category fibered in equivalence relations.
\end{prop}

\begin{proof}
Recall that a groupoid is an equivalence relation if and only if the only automorphisms are the identities.

Suppose that $\Inf(\xi_0)=0$, and consider the category $\F_{\xi_0}$. By Corollary~\ref{equiv} we have that for every $A \in \artl$ and object $\xi_0\to\xi \in \F_{\xi_0}(A)$, the induced homomorphism $\aut_A(\xi)\to \aut_k(\xi_0)$ is injective, and in particular $\aut_A(\xi_0\to \xi)$ (which is the preimage of $\id_{\xi_0}$) has at most one element (it will have exactly one, namely $\id_{\xi}$).

It follows that $\F_{\xi_0}(A)$ is an equivalence relation for every $A \in \artl$, and so $\F_{\xi_0}\to \artl^\op$ is fibered in equivalence relations. The converse is trivial.
\end{proof}

\subsection{Examples}

Now we analyze the group of infinitesimal automorphisms in our three examples.

\subsubsection{Schemes}

Consider the category $\catdef \to \artl^\op$ of flat deformations of schemes, and $X_0 \in \catdef(k)$.

\begin{prop}\label{infinit}
We have an isomorphism
$$
\Inf_{X_0}(\catdef)\simeq \der_k(\O_{X_0},\O_{X_0}) \simeq \Hom_{\O_{X_0}}(\diff_{X_0},\O_{X_0})\,.
$$
\end{prop}

\begin{proof}
We have to understand the functor $F\colon \fvect \to \set$ that takes $V\in\fvect$ to
$$
F(V) = \ker\bigl(\aut_{k \oplus V}(X_0|_{k \oplus V})\arr \aut_k(X_0)\bigr)
$$
where $X_0|_{k \oplus V}$ is the trivial deformation $X_V=X_0 \times_{\spec k} \spec k \oplus V$. In particular as topological spaces $|X_V|=|X_0|$, and on the structure sheaves we have
$$
\O_{X_V}=\O_{X_0}\otimes_k k \oplus V\simeq \O_{X_0}\oplus (V \otimes_k \O_{X_0})\,.
$$
An element $\phi \in \aut(X_0)(V)$ will clearly be the identity as a map between topological spaces, so we turn to the morphism $\phi^\sharp\colon \O_{X_V}\to \O_{X_V}$ on the structure sheaf, which is an automorphism of sheaves of $k \oplus V$-algebras inducing the identity on $\O_{X_0}$.

Using the analogue of Proposition~\ref{ext.der} for extensions of sheaves, with respect to the extension
$$
\xymatrix{
0\ar[r] & V\otimes_k \O_{X_0} \ar[r] & \O_{X_V}\ar[r] & \O_{X_0}\ar[r] & 0
}
$$
we see that $\phi^\sharp$ differs from the identity of $\O_{X_V}$ by a derivation
$$
D_\phi \in \der_k(\O_{X_0},V\otimes_k \O_{X_0})\,.
$$
Conversely every $\phi$ as above can be obtained in this way, and so for each $V \in \fvect$ we get a bijection
$$
F(V)\simeq \der_k(\O_{X_0},V \otimes_k \O_{X_0}) \simeq V\otimes_k\der_k(\O_{X_0},\O_{X_0})\,.
$$
These maps are also functorial in $V$ (as is readily checked), so the corresponding natural transformation is $k$-linear, and in particular we have an isomorphism
\begin{equation*}
\Inf_{X_0}(\catdef)=F(k)\simeq \der_k(\O_{X_0},\O_{X_0}).\qedhere
\end{equation*}
\end{proof}


In particular if $X_0$ is smooth the vector space $\Hom_{\O_{X_0}}(\diff_{X_0},\O_{X_0})$ coincides with $\H^0(X_0,T_{X_0})$, so that infinitesimal automorphisms correspond to sections of the tangent sheaf, or vector fields, which is an old intuitive idea from differential geometry.

Specializing further, if $X_{0}$ is a smooth projective curve of genus $g\geq 2$, then it has no non-trivial infinitesimal automorphism. This reflects the fact that the stack $\mathcal{M}_{g}$ of smooth curves of genus $g\geq 2$ is Deligne-Mumford (see the discussion at the very beginning of this section).

\subsubsection{Closed subschemes}

Now we turn to deformations of closed subschemes. It was already mentioned that in this case the space of infinitesimal automorphisms is trivial.

\begin{prop}
$\Inf_{Z_0}(\hilb^X)$ is trivial for every $Z_0 \in \hilb^X(k)$.
\end{prop}

\begin{proof}
This is immediate from the fact that, for source and target fixed, the arrows in $\hilb^X$ are uniquely determined by their image in $\artl^\op$. In particular an object of $\hilb^X(\dual)$ can only have one automorphism (because they map to the identity of $\dual$ in $\artl^\op$), which is the identity.
\end{proof}

\subsubsection{Quasi-coherent sheaves}

Finally let us consider the infinitesimal automorphisms of $\E_0 \in \qcoh^X(k)$ in the deformation category $\qcoh^X\to \artl^\op$.

\begin{prop}
We have an isomorphism
$$
\Inf_{\E_0}(\qcoh^{X})\simeq \Hom_{\O_{X_0}}(\E_0,\E_0)
$$
\end{prop}

\begin{proof}
We have to study the functor $F\colon \fvect \to \set$ defined by
$$
F(V)=\ker\left(\aut_{k \oplus V}(\E_V) \to \aut_k(\E_0) \right)
$$
where, if $V \in \fvect$, the sheaf $\E_V$ is the trivial lifting
$$
\E_V=\pi_V^*(\E_0)\simeq \E_0\otimes_k k \oplus V\simeq \E_0\oplus (V \otimes_k \E_0)
$$
(where $\pi_V\colon X_0\times_{\spec k}\spec k \oplus V \to X_0$ is the projection).

Consider an automorphism $\phi\colon \E_0\oplus (V \otimes_k \E_0) \to \E_0\oplus (V \otimes_k \E_0)$ of $\O_{X_V}$-modules that induces the identity on $\E_0$. Using $k \oplus V$-linearity and $V^2=(0)$, we see that $\phi$ restricts to the identity on $V\otimes_k \E_0$, and if we write $\phi(f)=f+G_\phi(f)$ for a section $f$ of the summand $\E_0\subseteq \E_V$, then $G_\phi\colon \E_0\to V\otimes_k \E_0$ is a homomorphism of $\O_{X_0}$-modules, and determines $\phi$ completely.

Conversely, given an $\O_{X_0}$-module homomorphism $G \in \Hom_{\O_{X_0}}(\E_0,V\otimes_k \E_0)$, we can define a homomorphism of $\O_{X_V}$-modules
$$
\phi_G\colon \E_0\oplus (V \otimes_k \E_0)\arr \E_0\oplus (V \otimes_k \E_0)
$$
by $\phi_G(f+\alpha)=f+G(f)+\alpha$, where $f$ is a section of $\E_0$ and $\alpha$ one of $V \otimes_k \E_0$. Moreover $\phi_G$ will be an automorphism, with inverse $\phi_{-G}$.

These two correspondences are inverse to each other, so that for each $V\in\fvect$ we have a bijection
$$
F(V)\simeq \Hom_{\O_{X_0}}(\E_0,V\otimes_k \E_0)\simeq V\otimes_k \Hom_{\O_{X_0}}(\E_0,\E_0)\,.
$$
These maps are easily seen to be functorial in $V$, so the resulting natural transformation will be $k$-linear, and we have an isomorphism
\begin{equation*}
\Inf_{\E_0}(\qcoh^X)\simeq \Hom_{\O_{X_0}}(\E_0,\E_0).\qedhere
\end{equation*}
\end{proof}

\section{Obstructions}\label{obs}

The present section is about obstruction theories, which tell us whether we can lift a given object along a small extension or not. In opposition to tangent spaces and groups of infinitesimal automorphisms, which are canonically defined, there can well be more than one obstruction theory for a given problem, and the choice of a particular one is important in some cases.

In particular we will focus on minimal obstruction spaces and their properties, and state a Theorem on the vanishing of obstructions that will be proved in Section~\ref{capform}. We will then show a particular obstruction theory for each one of our examples, and give a classical example of a variety over $\mathbb{C}$ with non-trivial obstructions.

\subsection{Obstruction theories}

Now we focus on the problem of existence of liftings. Given a deformation category $\fib$ and a small extension $A'\to A$, with an object $\xi \in \F(A)$, we would like to have a procedure to decide whether there is a lifting of $\xi$ to $A'$.

\begin{defin}
An \gr{obstruction theory} for $\xi_0 \in \F(k)$ is a pair $(V_\omega,\omega)$, where $V_\omega$ is a $k$-vector space and $\omega$ is a function that assigns to every small extension $A'\to A$ with kernel $I$ and every $\xi \in \F_{\xi_0}(A)$ an element
$$
\omega(\xi,A')\in I \otimes_k V_\omega
$$
called the \gr{obstruction} to lifting $\xi$ to $A'$, in such a way that:
\begin{itemize}
\item $\omega(\xi,A')=0$ if and only if there exists a lifting of $\xi$ to $A'$.
\item We have the following functoriality property: if $B'\to B$ is another small extension with kernel $J$, $\phi\colon A'\to B'$ is a homomorphism such that $\phi(I)\subseteq J$, and $\overline{\phi}\colon A\to B$, $\phi|_I\colon I\to J$ are the induced homomorphisms, then
$$
(\phi|_I \otimes \id)(\omega(\xi,A'))=\omega(\overline{\phi}_*(\xi),B') \in J\otimes_k V_\omega\,.
$$
\end{itemize}
\end{defin}

The space $V_\omega$ called an \gr{obstruction space} for $\xi_0$. If the association $\omega$ is identically zero (that is, every object can be lifted along any small extension), we say that $\xi_0$ (or the deformation problem associated with $\F_{\xi_0}$) is \gr{unobstructed}; otherwise, we say it is \gr{obstructed}.

\begin{examp}
If $\xi_0 \in \F(k)$ has the property that any object of $\F$ restricting to $\xi_0$ on $k$ can be lifted along any small extension, then it obviously admits a ``trivial'' obstruction theory, with $V_\omega=0$ and $\omega$ the only possible function. In particular such a $\xi_0$ is unobstructed.
\end{examp}

Notice that the functoriality property implies in particular that if $\omega(\xi,A')=0$ (i.e. $\xi$ admits a lifting to $A'$), then surely $\omega(\overline{\phi}_*(\xi),B')=0$ (i.e. $\overline{\phi}_*(\xi)$ admits a lifting to $B'$). But this is clear a priori, because the pullback along $\phi$ of a lifting of $\xi$ to $A'$ will be a lifting of $\overline{\phi}_*(\xi)$ to $B'$.

When dealing with concrete problems, it is usually possible to construct an obstruction theory, and sometimes the obstruction space is a cohomology group of a quasi-coherent sheaf on a certain noetherian scheme (usually one degree higher than the one yielding the tangent space of the deformation problem we are considering). We will see some examples of this later. In these cases in particular the obstruction will vanish locally (on affine open subschemes).

If we stick to the abstract setting, that is, if we consider an arbitrary deformation category $\fib$ and an object $\xi_0 \in \F(k)$, it is possible to construct ``abstract'' obstruction theories for $\xi_0$. In \cite{Fan} the authors define a more general notion of obstruction theory (for morphisms of deformation functors) using pointed sets, and among other results they show that, with mild hypotheses, one can always find an obstruction theory for a deformation functor (and even a universal one, in some sense).

Nevertheless notice that obstruction spaces are something that is intrinsically non-canonical, and moreover the choice of the obstruction theory one considers is very important in some cases (for example, in the theory of the virtual fundamental class \cite{bf-intrinsic}).

\subsubsection{The first obstruction}\label{first-obstruction}

Let $\Lambda = k$, and let $\F$ be a deformation category over $(\Art_k)^\op$. Let $\xi_{0}$ be an object of $\F(k)$, with tangent space $T \eqdef T_{\xi_{0}}\F$ and an obstruction theory $(V_{\omega}, \omega)$. By definition, each vector $v \in T$ defines an element $\xi_{v} \in \F(k[\epsilon])$, unique up to isomorphism. Set $k[\epsilon'] \eqdef k[t]/(t^{3})$, where $\epsilon' = [t]$. We consider $k[\epsilon']$ as a tiny extension of $k[\epsilon]$, via the homomorphism $k[\epsilon'] \arr k[\epsilon]$ sending $\epsilon'$ into $\epsilon$. With the object $\xi_{v}$ we associate an obstruction $\omega(\xi_{v}, k[\epsilon']) \in \generate{\epsilon'^{2}} \otimes_{k} V_{\omega} \simeq V_{\omega}$.

\begin{defin}\label{first-obstruction-def}
The first obstruction for $\xi_{0}$ is the function $\Phi_{\xi_{0}}\colon T \arr V_{\omega}$ that sends $v \in T$ into the obstruction $\omega(\xi_{v}, k[\epsilon'])$.
\end{defin}

We will compute the first obstruction for the deformations of a smooth variety in Proposition~\ref{first-smooth}; in this case it is given by a quadratic form. But this is not surprising, because the first obstruction is always a quadratic form. This can be shown as follows. 

Let $\sym^{*}_{k}(T^{\vee})$ be the symmetric algebra of the dual $T^{\vee}$ of the tangent space $T$; denote by $\m$ the ideal generated by elements of positive degree. Then we set
   \begin{align*}
   S_{1} &\eqdef \sym^{*}_{k}(T^{\vee})/\m^{2}\\
   &= k \oplus T^{\vee},
   \end{align*}
and
   \begin{align*}
   S_{2} &\eqdef \sym^{*}_{k}(T^{\vee})/\m^{3}\\
   &= k \oplus  T^{\vee} \oplus \sym^{2}_{k}(T^{\vee})\,.
   \end{align*}
A vector $v \in T$ can be thought of as a $k$-linear homomorphism $v\colon T^{\vee} \arr k$, so it yields a homomorphism of $k$-algebras $\sym^{2}(v)\colon \sym^{2}_{k}(T^{\vee}) \arr k$, which can be interpreted as evaluation of a quadratic form at $v$. This in turn gives a quadratic function $T \arr \Hom_{k}\bigl(\sym^{2}_{k}(T^{\vee}), k\bigr)$ sending $v$ into $\sym^{2}(v)$, and a quadratic form
   \[
   Q_{\xi_{0}}\colon T
   \arr \Hom_{k}\bigl(\sym^{2}_{k}(T^{\vee})\otimes V_{\omega},
   V_{\omega}\bigr)
   \]
sending $v \in T$ into $\sym^{2}(v) \otimes \id_{V_{\omega}}$.

Consider the universal first order lifting $\xi^{(1)}$ of $\xi_{0}$ to $S_{1}$ (Example~\ref{universal-first-order}); there is an obstruction $\omega(\xi^{(1)}, S_{2}) \in \sym^{2}_{k}(T^{\vee})\otimes V_{\omega}$ to lifting it to $S_{2}$.

The following is easy to see from functoriality of obstructions.

\begin{ex}
Check that the first obstruction $\omega(\xi_{v}, k[\epsilon']) \in V_{\omega}$ coincides with $Q_{\xi_{0}}(v)\bigl(\omega(\xi^{(1)}, S_{2})\bigr)$.
\end{ex}

Hence, the first obstruction $\Phi_{\xi_{0}}(v)$ is a quadratic function of $v$, as claimed.

\subsubsection{Minimal obstruction spaces}\label{minimalobstrsec}

Very often, an obstruction space is larger than necessary. For example, it can happen that $(V_\omega, \omega)$ is an obstruction theory for some $\xi_0 \in \F(k)$, the vector space $V_\omega$ is not zero, but nevertheless the map $\omega$ is (see the example in Proposition~\ref{undesirable}). However, there is always a minimal subspace of $V_{\omega}$ that contains all the obstructions, which gives a minimal obstruction theory.

\begin{defin}
Let $(V_\omega,\omega)$ be an obstruction theory for $\xi_0 \in \F(k)$. The \gr{minimal obstruction space} $\Omega_\omega$ of the given obstruction theory is the subspace of $V_\omega$ of elements $v \in V_\omega$ that correspond to obstructions along tiny extensions, in the following sense: there exists a tiny extension $A'\to A$, with a fixed isomorphism $I \simeq k$, and $\xi \in \F_{\xi_0}(A)$, such that $v$ is the image of the obstruction $\omega(\xi,A') \in I\otimes_k V_\omega$ under the induced isomorphism $I\otimes_k V_\omega\simeq k\otimes_k V_\omega \simeq V_\omega$.
\end{defin}

For this definition to make sense, we have to check that $\Omega_\omega$ is a vector subspace of $V_\omega$.

\begin{prop}\label{minimalobstr}
$\Omega_\omega\subseteq V_\omega$ is a vector subspace.
\end{prop}

\begin{proof}
It is easy to see that $\Omega_\omega$ contains zero and that it is closed under scalar multiplication (for the latter, just multiply the isomorphism $I\simeq k$ by the scalar under consideration). The only non-trivial part is closure under the sum operation.

Take two elements $v, w \in \Omega_\omega$, corresponding respectively to $\omega(\xi,A')$ and $\omega(\eta,B')$, with $A'\to A$ and $B'\to B$ two tiny extensions with kernels $I$ and $J$, fixed isomorphisms $f\colon I\simeq k$, $g\colon J\simeq k$, and objects $\xi \in \F_{\xi_0}(A)$, $\eta \in\F_{\xi_0}(B)$.

Then we take the fibered product $A\times_k B$, and notice that by RS $\xi$ and $\eta$ induce an object $\{\xi,\eta\}$ of $\F_{\xi_0}(A\times_k B)$ (since they restrict to $\xi_0$ over $k$). The map $A'\times_k B'\to A\times_k B$ gives a small extension, with kernel $f\oplus g\colon I\oplus J\simeq k\oplus k$; we have then an obstruction
   \begin{align*}
   \omega(\{\xi,\eta\},A'\times_k B')&\in (I\oplus J)\otimes_k V_\omega\\
   &\simeq (k\oplus k)\otimes_k V_\omega\\
   &\simeq V_\omega\oplus V_\omega
   \end{align*}
that corresponds to the pair $(u,v)$ (as the reader can check, using the functoriality property of the obstruction).

Now we take the sum $s\colon I\oplus J \simeq k\oplus k \to k$, defined by $s(i,j)=f(i)+g(j)$, and consider $K=\ker s\subseteq I\oplus J\subseteq A'\times_k B'$, an ideal. Since $s$ is surjective we have an isomorphism $h\colon (I\oplus J)/K \simeq k$.

Set $C' \eqdef (A'\times_k B')/K$. We have a tiny extension $C'\to A\times_k B$ with kernel $(I\oplus J)/K \simeq k$ (which is a sort of ``sum extension'' of the given ones), and the projection $\pi\colon A'\times_k B' \to C'$ induces a homomorphism of extensions
$$
\xymatrix{
0\ar[r] & I\oplus J\ar[r] \ar[d]^{\overline{s}} & A'\times_k B'\ar[r]\ar[d]^{\pi} & A\times_k B\ar[r]\ar[d]^{\overline{\pi}} & 0\\
0\ar[r] & (I\oplus J)/K \ar[r] & C'\ar[r] & A\times_k B \ar[r] & 0
}
$$
where the map $\overline{s}$ is the projection to the quotient, and corresponds to the addition $+\colon k\oplus k\to k$ under the isomorphisms above.

By functoriality of the obstruction we have then
$$
\omega(\overline{\pi}_*(\{\xi,\eta\}),C')=(\overline{s}\otimes \id)(\omega(\{\xi,\eta\},A'\times_k B'))\,,
$$
which corresponds to $u+v \in V_\omega$, as the reader can check using again functoriality of $\omega$. In conclusion this shows that $u+v\in \Omega_\omega$, as we want.
\end{proof}

Next, we see that $(\Omega_\omega, \omega)$ is an obstruction theory.

\begin{prop}\label{minimalobstr2}
Given a small extension $A'\to A$ with kernel $I$, and $\xi \in \F_{\xi_0}(A)$, we have
$$
\omega(\xi,A')\in I\otimes_k \Omega_\omega \subseteq I\otimes_k V_\omega\,.
$$
In particular $(\Omega_\omega,\omega)$ is an obstruction theory for $\xi_0$.
\end{prop}

\begin{proof}
Let $v_1,\hdots,v_n$ be a basis of $I$ as a $k$-vector space, and write the obstruction $\omega(\xi,A')\in I\otimes_k V_\omega$ as a sum
$$
\omega(\xi,A')=v_1\otimes w_1+\cdots+v_n\otimes w_n
$$
where $w_1,\hdots,w_n \in V_\omega$. We have to show that $w_1,\hdots,w_n$ are elements of $\Omega_\omega$.

Fix $1\leq i \leq n$, and let
$$
K_i \eqdef\ker v_i^\vee=\{v \in I : \text{if we write } v=a_1v_1+\cdots+a_nv_n \text{, then } a_i=0\}\subseteq I
$$
where $v_i^\vee\colon I\to k$ is the dual element of $v_i \in I$.

This $K_i$ is an ideal of $A'$, so set $B' \eqdef A'/K_i$. We have a tiny extension $B'\to A$ with kernel $I/K_i\simeq k$ (where the isomorphism is induced by $v_i^\vee$), and the projection $\pi\colon A'\to B'$ induces a homomorphism of extensions
$$
\xymatrix{
0\ar[r] & I\ar[r] \ar[d]^{\pi|_I} & A'\ar[r]\ar[d]^{\pi} & A\ar[r]\ar[d]^{\overline{\pi}} & 0\\
0\ar[r] & I/K_i \ar[r] & B'\ar[r] & A \ar[r] & 0\,.
}
$$
By functoriality of the obstruction we get
$$
\omega(\overline{\pi}_*(\xi),B')=(\pi|_I \otimes \id)(\omega(\xi,A'))
$$
which corresponds to the $i$-th component $w_i$ of $\omega(\xi,A')$ under the isomorphism $I \otimes_k V_\omega \simeq k^n\otimes_k V_\omega \simeq V_\omega^n$ given by the basis $v_1,\hdots,v_n$, because the diagram
$$
\xymatrix@C+10pt{
I\otimes_k V_\omega \ar[r]^\sim \ar[d]^{\pi|_I \otimes \id} & k^n\otimes_k V_\omega \ar[r]^\sim \ar[d]^{\pi_i \otimes \id} & V_\omega^n \ar[d]^{\pi_i}\\
(I/K_i)\otimes_k V_\omega \ar[r]^\sim & k\otimes_k V_\omega \ar[r]^\sim & V_\omega
}
$$
(where the horizontal isomorphisms are the ones we have already considered) is commutative.

Finally notice that $\omega(\overline{\pi}_*(\xi),B')$ is the obstruction associated with a tiny extension (since $I/K_i\simeq k$), so that $w_i \in \Omega_\omega$, and we are done.
\end{proof}

After the study of miniversal deformations in Section~\ref{capform}, we will see that we can obtain a formula for the dimension of $\Omega_\omega$ from a miniversal deformation of $\xi_0$ (provided it exists). In particular $\dim_k \Omega_\omega$ does not depend on the starting obstruction theory $(V_\omega, \omega)$.

This also follows from the next exercise, which says that minimal obstruction spaces are canonical.

\begin{ex}
Let $\fib$ be a deformation category, $(V_1,\omega_1)$ and $(V_2,\omega_2)$ be two obstruction theories for $\xi_0 \in \F(k)$, and denote by $\Omega_1$ and $\Omega_2$ the corresponding minimal obstruction spaces. Show that there is a canonical isomorphism $\phi\colon \Omega_1\simeq \Omega_2$ that preserves obstructions. Here with ``preserves obstructions'' we mean that if $A'\to A$ is a small extension with kernel $I$, and $\xi \in \F(A)$, then
$$
(\id\otimes \phi)(\omega_1(\xi,A'))=\omega_2(\xi,A') \in I\otimes_k \Omega_2\,.
$$
\end{ex}

Even though the minimal obstruction space is a good thing to have, in practice it is (in general) very hard to calculate. Because of this, in most applications it suffices to have an obstruction theory that is possibly easier to calculate and more naturally defined, as in the examples we will see later on.

\subsubsection{The Ran--Kawamata unobstructedness Theorem}

The following result, first stated and proved in \cite{Kaw}, can be applied in some cases to conclude that a deformation problem is unobstructed. In this section we assume $\Lambda=k$.

\begin{thm}[Kawamata]\label{rk}
Let $\F\to \art^\op$ be a deformation category, and take $\xi_0 \in \F(k)$. Assume that:
\begin{itemize} 
\item $T_{\xi_0}\F$ is finite-dimensional.
\item $k$ has characteristic~0.
\item If $A \in \art$ and $\xi \in \F_{\xi_0}(A)$, then the functor $F_\xi\colon \fmod\to \mod$ described on page~\pageref{generalization} is right-exact (that is, carries surjections into surjections).
\end{itemize}
Then $\xi_0$ is unobstructed.
\end{thm}

We postpone the proof to Section~\ref{capform}, page~\pageref{prf-rk}.

\begin{examp}
Let $X$ be a scheme over $k$ and consider an invertible sheaf $\L_0 \in \qcoh^X(k)$ on $X$. Suppose also that $\car k=0$ and $\H^1(X,\O_X)$ is finite-dimensional. We want to show that in this case $\L_0$ is unobstructed in $\qcoh^X$, using the Ran-Kawamata Theorem.

To do this, we consider $A \in \art$ and $\L \in \qcoh^X_{\L_0}(A)$, and we want to understand the functor $F_\L\colon \fmod\to \mod$; recall that this is defined by
$$
F_\L(M)=\{\text{isomorphism classes of liftings of } \L \text{ to } A \oplus M\}\,.
$$
Notice that if $\L_M$ denotes the trivial pullback of $\L$ to $A \oplus M$ for any $M \in \fmod$ (the one along the inclusion $A\to A \oplus M$), then there is a natural equivalence of functors $\phi\colon F_{\O_{X_A}}\simeq F_\L$ (where $X_A \eqdef X\times_{\spec k}\spec A$ is the trivial deformation, as usual); if $M \in \fmod$, the function $\phi_M\colon F_{\O_{X_A}}(M)\to F_\L(M)$ is defined by
$$
\phi_M([\E])=[\E\otimes_{\O_{X_{A \oplus M}}}\L_M]\,.
$$
The group $F_{\O_{X_A}}(M)$ is the kernel of the restriction map $\pic X_{A \oplus M} \arr \pic X_{A}$; it follows from the long cohomology exact sequence associated with the short exact sequence of sheaves on $|X|$
   \[
   \xymatrix@C=30pt{
   1 \ar[r]& M\otimes_{A}\O_{X_{A}} \ar[r]^-{f \mapsto 1+f}&
   \O^{*}_{X_{A \oplus M}} \ar[r]& \O^{*}_{X_{A}} \ar[r] & 0\,.
   }
   \]
that there is a functorial isomorphism
$$
F_{\O_{X_A}}(M)\simeq \H^1(X,M\otimes_A \O_{X_A})\,.
$$
Now $M\otimes_A \O_{X_A}\simeq M\otimes_A (A\otimes_k \O_X)\simeq M\otimes_k \O_X$, and the functor $-\otimes_k \O_X$ is exact. Consequently $\H^1(X,-\otimes_k \O_X)\simeq \H^1(X,-\otimes_A \O_{X_A})\simeq F_{\O_{X_A}}$ is exact too, and since $\car k=0$ and
$$
T_{\L_0}\qcoh^X\simeq \Ext^1_{\O_X}(\L_0,\L_0)\simeq \H^1(X,\mathcal{E}nd_{\O_X}(\L_0))\simeq \H^1(X,\O_X)
$$
is finite-dimensional, we can apply Theorem~\ref{rk}, and conclude that $\L_0$ is unobstructed.
\end{examp}

\begin{rmk}
This fails in positive characteristic: there are famous examples, due to Jun-Ichi Igusa (\cite{igusa}), of smooth projective varieties in which invertible sheaves are obstructed.
\end{rmk}

We have the following corollary of the theorem, which is useful for example when considering deformations of abelian varieties, Calabi-Yau manifolds, K3 surfaces, etc.

\begin{cor}[Ran]
Let $X_0$ be a smooth and projective scheme over $k$ (with $\car k=0$), whose canonical sheaf $\omega_{X_0}$ is trivial (i.e. isomorphic to $\O_{X_0}$). Then $X_0$ is unobstructed.
\end{cor}

\begin{proof}
Since $\car k=0$ and the tangent space $T_{X_0}\catdef\simeq \H^1(X_0,T_{X_0})$ (see Remark~\ref{nonsing}) is finite dimensional, to apply the Ran-Kawamata Theorem we only need to show, given $A \in \art$ and an object $X \in \catdef_{X_0}(A)$, that the functor $F_X\colon \fmod\to \mod$ defined by
$$
F_X(M)=\{\text{isomorphism classes of liftings of }X\text{ to }A \oplus M\}
$$
is right-exact.

If $M\in \fmod$, and we call $f\colon X\to \spec A$ the structure morphism, then one can show (using the same techniques we used to calculate the tangent space of $\catdef$, Theorem~\ref{kodaira}) that there is a functorial isomorphism
$$
F_X(M)\simeq \H^1(X,f^*M\otimes_{\O_X}T_{X/A})\simeq \H^1(X,M\otimes_A T_{X/A})\,.
$$
Using the results on base change of Appendix~\ref{appc} we will show that there is also a functorial isomorphism
\begin{equation}\label{bc}
\H^1(X,M\otimes_A T_{X/A})\simeq M\otimes_A \H^1(X,T_{X/A})
\end{equation}
which shows that the functor $F_X$ is isomorphic to $-\otimes_A \H^1(X,T_{X/A})$, and so it is right-exact.

Set $n \eqdef \dim X_0 = \dim X $; since $\omega_{X_0}=\diff_{X_0}^n\simeq \O_{X_0}$ (recall that $\diff_{X_0}^i$ denotes $\bigwedge^i\diff_{X_0}$), we have a global nowhere vanishing section $s$ of $ \diff_{X_0}^n$, which is an element of $\H^0(X_0,\diff_{X_0}^n)$. By Deligne's Theorem (Theorem~\ref{deligne}) the natural map
$$
k\otimes_A \H^0(X,\diff_{X/A}^n)\arr \H^0(X_0,\diff_{X_0}^n)
$$
is an isomorphism, and $s$ corresponds to a global section of $\diff_{X/A}^n$ that is nowhere vanishing as well, since $|X|=|X_0|$.

From the existence of this section we get that $\diff^n_{X/A} \simeq \O_X$. Moreover for each $j\leq n$ we have a bilinear nondegenerate pairing
$$
\diff^j_{X/A}\times \diff^{n-j}_{X/A}\arr \diff_{X/A}^n\simeq \O_X
$$
that induces an isomorphism $\diff^{n-j}_{X/A}\simeq (\diff^j_{X/A})^\vee$.

This implies in particular that $T_{X/A}=(\diff^1_{X/A})^\vee\simeq \diff^{n-1}_{X/A}$, which by Deligne's Theorem again satisfies base change, and then we have our functorial isomorphism (\ref{bc}). This concludes the proof, as we already remarked.
\end{proof}

\subsection{Examples}

Now we describe an obstruction theory for each of our main examples, and give a classical example of a variety with non-trivial obstructions.

\subsubsection{Schemes}

We consider the deformation category of flat schemes $\catdef\to\artl^\op$, and $X_0 \in \catdef(k)$ a local complete intersection, generically smooth scheme of finite type over $k$.

\begin{thm}\label{schemesobstr}
With the hypotheses above, there is an obstruction theory $(V_\omega, \omega)$ for $X_0$ with vector space
$$
V_\omega=\Ext^2_{\O_{X_0}}(\diff_{X_0},\O_{X_0})\,.
$$

Furthermore, let $A' \arr A$ be a small extension in $\artl$, and assume that $X_{A}$ is a lifting of $X_{0}$ to $A$. Let $Y_{A} \arr X_{A}$ be an \'etale morphism, and set $Y_{0} \eqdef Y_{A}\times_{\spec A}\spec k$. Since the induced morphism $\phi\colon Y_{0} \arr X_{0}$ is \'etale, we have $\phi_{0}^{*}\Omega_{X_{0}} \simeq \Omega_{Y_{0}}$; furthermore, $\phi_{0}$ induces a morphism
   \[
   \phi_{0}^{*}\colon \Ext^{2}_{\O_{X_{0}}}(\Omega_{X_{0}}, \O_{X_{0}})
   \arr \Ext^{2}_{\O_{Y_{0}}}(\Omega_{Y_{0}}, \O_{Y_{0}})\,.
   \]
If $I$ is the kernel of the small extension $A' \arr A$, then the homomorphism
   \[
   \id_{I}\otimes\phi_{0}^{*}\colon
   I \otimes_{k}\Ext^{2}_{\O_{X_{0}}}(\Omega_{X_{0}}, \O_{X_{0}})
   \arr I \otimes_{k}\Ext^{2}_{\O_{Y_{0}}}(\Omega_{Y_{0}}, \O_{Y_{0}})
   \]
carries the obstruction to lifting $X_{A}$ into the obstruction to lifting $Y_{A}$.
\end{thm}

This follows from the general machinery of the cotangent complex in \cite{Ill}. A proof of this Theorem that does not use the cotangent complex can be found in \cite{Vis}.

\begin{rmk}\label{affine.unobstr}
If $X_0$ is also affine, then $\Ext^2_{\O_{X_0}}(\diff_{X_0},\O_{X_0})=0$. In particular deformations of an affine $X_0$ with the hypotheses above are unobstructed.
\end{rmk}

As with the tangent space, in the general case one can still find an obstruction theory for $X_0$, by using the cotangent complex. In general $X_0$ has an obstruction theory with obstruction space
$$
V_\omega=\Ext^2_{\O_{X_0}}(L_{X_0/k},\O_{X_0})
$$
(see \cite[Chapter III, Th\'{e}or\`{e}me 2.1.7]{Ill}).

\subsubsection{Smooth varieties}

We give a proof of the Theorem above only in the case of smooth varieties, which can be studied using \v{C}ech cohomology. Consider the deformation category $\catdef\to \artl^\op$.

\begin{thm}\label{smoothobstr}
Every smooth variety $X_0 \in \catdef(k)$ has an obstruction theory $(V_\omega, \omega)$ with obstruction space
$$
V_\omega=\H^2(X_0,T_{X_0})\,.
$$
\end{thm}
\begin{proof}
Let $A'\to A$ be a small extension with kernel $I$, and $X \in \catdef_{X_0}(A)$ be a deformation of $X_0$ over $\spec A$. We show how to construct the element $\omega(X,A') \in I\otimes_k \H^2(X_0,T_{X_0})$.

Let $\mathcal{U}=\{U_i\}_{i \in I}$ be an open affine cover of $X_0$, and denote by $X|_{U_i}$ the induced deformation of $U_i$ over $A$, obtained by considering $U_i\subseteq X_0$ as an open subscheme of $X$ (recall that $|X|=|X_0|$). By Remark~\ref{affine.unobstr} we have that $U_i$ is unobstructed, and so we can find deformations $Y_i \in \catdef_{U_i}(A')$ such that the restriction of each $Y_i$ to $A$ is $X|_{U_i}$.

Now notice that, since $U_{ij}$ is affine as well, the restrictions $Y_i|_{U_{ij}}$ and $Y_j|_{U_{ij}}$ are isomorphic deformations of $U_{ij}$ over $A'$, by Remark~\ref{nonsing} and Corollary~\ref{trivial}. Then for each pair of indices we have an isomorphism of deformations
$$
\nu_{ij}\colon Y_j|_{U_{ij}}\to Y_i|_{U_{ij}}
$$
which restricts to the identity of $X|_{U_{ij}}$ on the pullback to $A$, and for each triplet of indices we can consider the composite
$$
\nu_{ijk}=\nu_{ij}\circ \nu_{jk}\circ \nu_{ik}^{-1}\,.
$$
Each $\nu_{ijk}$ is an automorphism of the deformation $Y_i|_{U_{ijk}}$ of $U_{ijk}$ over $A'$ that restricts to the identity on the pullback to $A$, and so by Propositions \ref{inf} and \ref{infinit} it corresponds to an element
$$
e_{ijk} \in I\otimes_k \Inf(U_{ijk}) \simeq I\otimes_k \der_k(A_{ijk},A_{ijk}) \simeq \Gamma(U_{ijk},I\otimes_k T_{X_0})
$$
where $\spec A_{ijk}=U_{ijk}$.

The family $\{e_{ijk}\}_{i,j,k \in I}$ is a \v{C}ech 2-cocycle for the sheaf $I\otimes_k T_{X_0}$, with respect to the cover $\mathcal{U}$: we have to show that for every quadruple of indices $i,j,k,l \in I$ we have
$$
e_{jkl}-e_{ikl}+e_{ijl}-e_{ijk}=0
$$
as elements of $\Gamma(U_{ijkl},I\otimes_k T_{X_0})\simeq I\otimes_k \Inf(U_{ijkl})$. We rewrite this as
\begin{equation}\label{cocycle}
e_{ijl}-e_{ikl}-e_{ijk}=-e_{jkl}
\end{equation}
and notice that the left-hand side corresponds (under the isomorphism of Proposition~\ref{inf}) to the infinitesimal automorphism
$$
\nu_{ijl}\circ \nu_{ikl}^{-1}\circ \nu_{ijk}^{-1}=\nu_{ij}\circ \nu_{jkl}^{-1}\circ \nu_{ij}^{-1}
$$
of the deformation $Y_i|_{U_{ijkl}}$ of $U_{ijkl}$, and the right-hand side to the infinitesimal automorphism $\nu_{jkl}^{-1}$ of the deformation $Y_j|_{U_{ijkl}}$. Moreover the restriction of $\nu_{ij}$ is an isomorphism between these two deformations, and so (\ref{cocycle}) follows from Remark~\ref{autominf}.

Then we have an element $[\{e_{ijk}\}_{i,j,k \in I}]$ of
$$
\check{H}^2(\mathcal{U},I \otimes_k T_{X_0})\simeq \H^2(X_0,I\otimes_k T_{X_0})\simeq I \otimes_k \H^2(X_0,T_{X_0})
$$
that we call $\omega(X,A')$. We leave to the reader to check that this is well-defined (it does not depend on the choice of the automorphisms $\nu_{ij}$). It is also clear that each cocycle in this cohomology class corresponds to a family of isomorphisms, by reversing this construction, and moreover one can check that the $\omega(X,A')$ defined does not depend on the open affine cover $\mathcal{U}$ of $X_0$.

Now it is clear from the construction that we will have $\omega(X,A')=0$ if and only if $X$ admits a lifting to $A'$, and the functoriality property is an easy consequence of the functoriality of the isomorphism we constructed in the proof of Proposition~\ref{infinit}.
\end{proof}

In particular if $X_0$ is a smooth curve, then $\H^2(X_0,T_{X_0})=0$ and so $X_0$ is unobstructed.

The preceding proof shows a typical pattern that can be used in other cases to construct obstructions. Here is the (rather vague) idea: if our deformation problem has an underlying scheme $X$ ($X_0$ in the case above), and it localizes naturally on this scheme, in the sense that every deformation over $X$ induces one over any of its open subschemes (just by restriction, in the case above), and moreover:
\begin{itemize}
\item Infinitesimal automorphisms form a sheaf $\mathcal{I}$ on $X$.
\item Locally we always have liftings.
\item Two liftings of the same deformation are always locally isomorphic.
\item We can reconstruct our deformations from local compatible data.
\end{itemize}
Then we can mimic the preceding proof to construct an obstruction theory, with space $\H^2(X,\mathcal{I})$ (see \cite{Oss}).

Here is an example of a smooth variety $X_0$ such that $\H^2(X_0,T_{X_0})\neq 0$, but nevertheless $X_0$ is unobstructed, so the map $\omega$ must be zero.

\begin{prop}\label{undesirable}
Let $Z_0\subseteq \P^3_k$ be a smooth surface of degree $d \geq 6$. Then $\H^2(Z_0,T_{Z_0})\neq 0$, but $Z_0$ is unobstructed.
\end{prop}

\begin{proof}
The fact that $Z_0$ is unobstructed is immediate from Propositions \ref{vanish} and \ref{immers}: given an object $Z \in \catdef_{Z_0}(A)$ and a small extension $A'\to A$, because of Proposition~\ref{immers} we have a closed immersion $Z\subseteq \P^3_A$, which by Proposition~\ref{vanish} lifts to some $Z'\subseteq \P^3_{A'}$ over $A'$, and forgetting the immersion this gives a lifting $Z'\in \catdef_{Z_0}(A')$ of $Z$ to $A'$.

The fact that $\H^2(Z_0,T_{Z_0})\neq 0$ is proved by the following calculation, similar to the ones we used in the proof of Proposition~\ref{hypers}. From the dual of the conormal sequence of $Z_0\subseteq \P^3_k$
$$
\xymatrix{
0\ar[r] & T_{Z_0}\ar[r] & T_{\P^3_k}|_{Z_0}\ar[r] & \O_{Z_0}(d) \ar[r] & 0
}
$$
we see that it suffices to show that $\H^2(Z_0,T_{\P^3_k}|_{Z_0})\neq 0$. By Serre's duality we have
$$
\H^2(Z_0,T_{\P^3_k}|_{Z_0})\simeq \H^0(Z_0,\diff_{\P^3_k}(d-4)|_{Z_0})^\vee
$$
and by the twisted and restricted Euler sequence
$$
\xymatrix{
0\ar[r] & \diff_{\P^3_k}(d-4)|_{Z_0} \ar[r] & \O_{Z_0}(d-5)^{\oplus 4} \ar[r] & \O_{Z_0}(d-4) \ar[r] & 0
}
$$
it is sufficient to show that
$$
\dim_k \H^0(Z_0,\O_{Z_0}(d-5))^4 > \dim_k \H^0(Z_0,\O_{Z_0}(d-4))
$$
or, using known formulas for the dimensions above, that
$$
4\begin{pmatrix}d-2\\3\end{pmatrix}>\begin{pmatrix}d-1\\3\end{pmatrix}
$$
for $d\geq 6$, which is easy to check.
\end{proof}

\subsubsection{An obstructed variety}

Here is a classical example due to Kodaira and Spencer of a smooth projective variety over $\mathbb{C}$ with non-trivial obstructions.

Let $X$ be a smooth variety over $k$. The tangent sheaf $T_X = \Homsh(\diff_X,\O_X)$ has a natural structure of sheaf of Lie algebras over $k$: over an open affine subset $U=\spec A$, we have $T_X(U)=\der_k(A,A)$, and given two derivations $D,E\colon A\to A$ we can define $[D,E]\colon A\to A$ by
$$
[D,E](x)=D(E(x))-E(D(x))\,.
$$
It is immediate to check that $[D,E]$ is still a $k$-derivation, and that this product gives a structure of Lie algebra over $k$ to $T_X(U)$. Moreover this local construction gives a global Lie product $\lie\colon T_X\times T_X\to T_X$, which, being $k$-bilinear, induces a linear map $T_X\otimes_k T_X\to T_X$, which we still denote by $\lie$.

Moreover if we fix an open affine cover of $X$, say $\mathcal{U}=\{U_i\}_{i \in I}$, then there is a product
$$
\check{H}^1(\mathcal{U},T_X)\times \check{H}^1(\mathcal{U},T_X)\arr \check{H}^2(\mathcal{U},T_X\otimes_k T_X)
$$
in \v{C}ech cohomology, induced by the tensor product (see for example \cite[II, $\S$~6]{God}).

From this product and its properties we get a quadratic form $\check{H}^1(\mathcal{U},T_X)\to \check{H}^2(\mathcal{U},T_X)$, defined as the composite
$$
\check{H}^1(\mathcal{U},T_X) \to \check{H}^1(\mathcal{U},T_X)\times \check{H}^1(\mathcal{U},T_X)\to  \check{H}^2(\mathcal{U},T_X\otimes_k T_X) \to \check{H}^2(\mathcal{U},T_X)
$$
where the first map is the diagonal $v\mapsto (v,v)$, the second is induced by the tensor product, and the last one by the Lie product.
If we represent an element $v \in \check{H}^1(\mathcal{U},T_X)$ as a $1$-cocycle $\{a_{ij}\}_{i,j \in I}$, then the image of $v$ along the above map, which we denote by $[v,v]$, is given by the class of the $2$-cocycle $\{[a_{ij},a_{jk}]\}_{i,j,k \in I}$.


Finally recall that if $X_0$ is a smooth variety, then $T_{X_0}\catdef\simeq \H^1(X_0,T_{X_0})$ and we have an obstruction theory with vector space $\H^2(X_0,T_{X_0})$.

\begin{prop}\label{first-smooth}
Let $X_0$ be a smooth variety over $k$, with $\car k\neq 2$. Then the first obstruction of $X_0$ (see Definition~\ref{first-obstruction-def})
$$
\Phi\colon \H^1(X_0,T_{X_0})\to \H^2(X_0,T_{X_0})
$$
with respect to the obstruction theory described in the preceding section is given by $\Phi(v)=\frac{1}{2}[v,v]$.
\end{prop}

\begin{proof}
With the notation of Section~\ref{captang}, let us take an element $v\in \H^1(X_0,T_{X_0})$, a first-order deformation $X_v$ corresponding to $v$, and the associated 1-cocycle $\{d_{ij}\}_{i,j \in I}$ of $T_{X_0}$ for an open affine cover $\mathcal{U}=\{U_i\}_{i \in I}$ of $X_0$. Recall that $d_{ij}$ is the derivation associated with the infinitesimal automorphism $\theta_{ij}$ of $U_{ij}\times_{\spec k}\spec \dual$; in other words we have
$$
\theta_{ij}^\sharp(f+\eps g)=f+\eps(d_{ij}(f)+ g)
$$
where we see $\O_{U_{ij}}\otimes_k \dual$ as $\O_{U_{ij}}\oplus(\langle\eps\rangle\otimes_k\O_{U_{ij}})$ and $f$, $g$ are sections of $\O_{U_{ij}}$ (see Proposition~\ref{infinit}).

Now let us calculate the obstruction to lifting $X_v$ to $k[\eps']\eqdef k[t]/(t^3)$ (where $\eps'=[t]$) constructed in the proof of Theorem~\ref{smoothobstr}. In this case since $X_v|_{U_i}$ is isomorphic to the trivial deformation $U_i\times_{\spec k}\spec \dual$ of $U_i$ over $\dual$, we can take $Y_i=U_i\times_{\spec k} \spec k[\eps']$ as liftings to $k[\eps']$. Moreover one checks at once that the automorphisms of sheaves of algebras $\phi_{ij}\colon \O_{U_{ij}}\otimes_k k[\eps'] \to \O_{U_{ij}}\otimes_k k[\eps']$ defined as
$$
\phi_{ij}(f+\eps'g+\eps'^2h)=f+\eps'(d_{ij}(f)+g)+\eps'^2\left(d_{ij}(g)+\frac{1}{2}d_{ij}^2(f)+h\right)
$$
where we see $\O_{U_{ij}}\otimes_k k[\eps']$ as $\O_{U_{ij}}\oplus (\langle\eps'\rangle\otimes_k \O_{U_{ij}})\oplus (\langle\eps'^2\rangle\otimes_k \O_{U_{ij}})$ and $d_{ij}^2$ is $d_{ij}\circ d_{ij}$, give isomorphisms $\nu_{ij}\colon Y_j|_{U_{ij}}\simeq Y_i|_{U_{ij}}$ inducing the identity on $X_v|_{U_{ij}}$.

The element $e_{ijk} \in \langle\eps'^2\rangle\otimes_k \Inf(U_{ijk}) \simeq \der_k(A_{ijk},A_{ijk}) \simeq \Gamma(U_{ijk},T_{X_0})$ associated with the composite $\nu_{ijk}=\nu_{ij}\circ \nu_{jk}\circ \nu_{ik}^{-1}$ is easily seen to be
$$
\frac{1}{2}\left(d_{ij}\circ d_{jk}-d_{jk}\circ d_{ij}\right)=\frac{1}{2}[d_{ij},d_{jk}]\,.
$$
Finally this gives
$$
\omega(X_v,k[\eps'])=[\{e_{ijk}\}_{i,j,k \in I}]=\left[ \left\{ \frac{1}{2}[d_{ij},d_{jk}] \right\}_{i,j,k \in I} \right]=\frac{1}{2}[v,v]
$$
as claimed.
\end{proof}

Because of the Proposition just proved, to give an example of an obstructed variety it suffices to find one such that the map $\Phi$ above is nonzero, or equivalently such that the quadratic form given by the Lie product is not identically zero. We do this in the following exercise.

\begin{ex}
Take $X_0=\P^1_\mathbb{C} \times_{\spec \mathbb{C}} Y$ where $Y$ is an abelian variety over $\mathbb{C}$ of dimension at least $2$.
\begin{itemize}
 \item Show that the product
$$
\H^1(Y,\O_Y)\otimes_\mathbb{C} \H^1(Y,\O_Y)\arr \H^2(Y,\O_Y)
$$
induced by tensor product is not identically zero, and the same is true for the product
$$
\H^0(\P^1_\mathbb{C},T_{\P^1_\mathbb{C}})\otimes_\mathbb{C} \H^0(\P^1_\mathbb{C},T_{\P^1_\mathbb{C}}) \arr \H^0(\P^1_\mathbb{C},T_{\P^1_\mathbb{C}})
$$
induced by the Lie product.
\item Take $a,b \in \H^1(Y,\O_Y)$ such that $a b\neq 0$ and $D,E \in \H^0(\P^1_\mathbb{C},T_{\P^1_\mathbb{C}})$ such that $[D,E]\neq 0$, and show that they naturally give elements $aD$ and $bE$ of $\H^1(X_0,T_{X_0})$. (\emph{Hint}: use \v{C}ech cohomology, pull the sections back to $X$, and then multiply them.)
\item Check that $[aD,bE]=(ab)[D,E]$, and then use the K\"{u}nneth formula to show that $[aD,bE]\neq 0$.
\end{itemize}
\end{ex}

\subsubsection{Closed subschemes}\label{closedobs}

Let $X$ be a separated scheme over $\Lambda$. Assume that $X$ is flat over $\Lambda$. Consider the deformation category $\hilb^X\to \artl^\op$. Take $Z_0 \in \hilb^X(k)$, call $I_0$ its sheaf of ideals, and $\mathcal{N}_0$ the normal sheaf
   \[
   \mathcal{N}_0 \eqdef \Homsh_{\O_{Z_{0}}}(I_0/I_0^2,\O_{Z_0})\,.
   \]


\begin{thm}
Assume that the subscheme $Z_{0}$ of $X_{0}$ is a local complete intersection. Then there is an obstruction theory $(V_\omega, \omega)$ for $Z_0$, with obstruction space
$$
V_\omega=\H^1(Z_0,\mathcal{N}_0)\,.
$$
\end{thm}

\begin{rmk}
This implies, for example, that if $X_0$ is affine, the deformation theory is unobstructed. It is easy to give examples in which $X_0$ is affine, the subscheme $Z_{0}$ not a local complete intersection and the deformation theory is obstructed; this shows that the formula of the Theorem can not be correct in the general case.

In general one can show that there is an obstruction theory with vector space $\Ext^{1}_{\O_{X_0}}(I_{0}, \O_{X_{0}})$; this space contains $\H^1(Z_0,\mathcal{N}_0)$, but is often larger, even when $Z_{0}$ is a local complete intersection. This gives an example of a deformation problem with two different ``natural'' obstruction theories.
\end{rmk}

\begin{proof} Let us start with a Lemma, which takes care of the local situation. Let $A'\to A$ be a surjection in $\artl$ with kernel $I$ and take and object $Z \in \hilb^X_{Z_0}(A)$. In the Lemma that follows, we will say that a sequence of elements $f_{1}$, \dots,~$f_{r}$ of a ring $R$ \emph{is a regular sequence} if its image in the localization $R_{p}$ is a regular sequence, for all $p \in \spec R$ containing $(f_{1}, \dots, f_{r})$; this is equivalent to the exactness of the Koszul complex of the $f_{i}$. From a geometric point of view this is more useful that the standard definition of regular sequence in a non-local ring.

\begin{lemma}\label{lem:lci}
Assume that $X = \spec R$ is affine, and that the ideal of $Z$ is of the form $(f_{1}, \dots, f_{r})$, where $f_{1}$, \dots,~$f_{r}$ is a regular sequence in $R_{A}$, where $R_{A} \eqdef R \otimes_{\Lambda} A$.

Let $f'_{1}$, \dots,~$f'_{r}$ be liftings of $f_{1}$, \dots,~$f_{r}$ to $R_{A'}$. Then $R_{A'}/(f'_{1}, \dots, f'_{r})$ is flat over $A'$, its spectrum is a lifting of $Z$ to $A'$ and $f'_{1}$, \dots,~$f'_{r}$ is a regular sequence in $R_{A'}$.

Conversely, every lifting of $Z$ to $A'$ is of this form.
\end{lemma}

This implies in particular that local complete intersections can always be lifted locally.

\begin{proof}
Let $Z' = \spec S'$ be a closed subscheme of $X_{A'} \eqdef X\times_{\spec \Lambda}\spec A'$ restricting to $Z$ over $A$. The ideal of $Z$ is $J \eqdef (f_{1}, \dots, f_{r})$; call $J'\subseteq R_{A'}$ the ideal of $Z'$. Then we claim that the natural surjection $J'/IJ' \arr J$ is an isomorphism if and only if $S'$ is flat over $A'$. Tensoring the exact sequence
   \[
   \xymatrix{
   0 \ar[r] &J' \ar[r] &R_{A'} \ar[r] &S' \ar[r] &0
   }
   \]
with $A$, and using the fact that $R_{A'}$ is flat over $A'$, we see that $\tor_{1}^{A'}(S', A)$ is the kernel of the homomorphism $J'/IJ' \arr J$, and then the result follows from the local criterion of flatness (Theorem~\ref{local-flatness}).

Now, assume that $Z'$ is flat over $A'$, so it is a lifting of $Z$. Let $f'_{1}$, \dots,~$f'_{r}$ be liftings of $f_{1}$, \dots,~$f_{r}$ to $J'$. From the fact that $J'/IJ' = J$ and the fact that $I$ is nilpotent we conclude that the $f'_{i}$ generate $J'$.

Let us check that $f'_{1}$, \dots,~$f'_{r}$ is a regular sequence. Set $S \eqdef R_{A}/(f_{1}, \dots, f_{r})$ and $S' \eqdef R_{A'}/(f'_{1}, \dots, f'_{r})$. Let $K'_\bullet$ be the Koszul complex
of 
$f'_1,\dots,f'_r$; then
$K_\bullet = K'_\bullet \otimes _{A'} A$ is the Koszul
complex of $f_1,\dots,f_r$. We have a homology spectral sequence
$$
	E^2_{pq} = \tor^{A'}_p\bigl(\H_q(K'_\bullet),A\bigr)\Longrightarrow\H_{p+q}
	(K_\bullet) = 
	\begin{cases}
	0
	&\text{if }p+q>0\\ S &\text{if }p+q = 0\,.
	\end{cases}
$$
Notice that
$E^2_{p0} = 0$ for $p>0$, because
$\H_0(K'_\bullet) = S'$ is flat over $A'$.
From this, and the fact that the abutment is 0 in degree 1, we
get that
$\H_1(K'_\bullet)\otimes_{A'}A = E^2_{01} = 0$. This
implies that
$\H_1(K'_\bullet)$ is 0, and hence $E^2_{p1} = 0$ for all
$p$. Analogously one proves that
$\H_2(K'_\bullet) = 0$, and by induction on $q$ that
$\H_q(K'_\bullet) = 0$ for all
$q>0$. This proves that $f'_1,\dots,f'_r$ is a regular sequence.

Conversely, start from liftings $f'_{1}$, \dots,~$f'_{r}$ of $f_{1}$, \dots,~$f_{r}$ to $S'$, and set $J' \eqdef (f'_{1}, \dots, f'_{r})$ and $S' \eqdef R_{A'}/J'$. We claim that $\spec S'$ is a lifting of $Z$ to $A'$, that is, that $S'$ is flat over $A'$ (by the argument above, this implies that $f'_{1}$, \dots,~$f'_{r}$ is a regular sequence). This is equivalent to saying that $J'/IJ' = J$.

Let $\sum_i a'_if'_i$ be an element of $J'$ whose image $\sum_ia_if_i$ in $J$ is 0. Then,
because $f_1,\dots,f_r$ is a regular sequence we can write
$(a_1,\ldots,a_r)\in R_{A'}^n$ as a linear combination of
standard relations of the form
   \[
   (0,\ldots,0,\hskip-6pt\underbrace{f_j}_{i^{\rm th}\ \rm
   place}\hskip-6pt,0,\ldots,0,\hskip-6pt\underbrace{-f_i}_{j^{\rm th}\
   \rm place}\hskip-6pt,0,\ldots,0).
   \]
These lift to relations
   \[(0,\ldots,0,\hskip-6pt\underbrace{f'_j}_{i^{\rm th}\ \rm
   place}\hskip-6pt,0,\ldots,0,\hskip-6pt\underbrace{-f'_i}_{j^{\rm th}\
   \rm place}\hskip-6pt,0,\ldots,0)
   \]
among the $f'_{i}$. Hence $(a'_1,\ldots,a'_r)$ can be written as a relation
among the $f'_i$, plus an element $(b'_1,\ldots,b'_r)\in
(IR_{A'})^n$, so that $\sum_i a'_if'_i = \sum_i b'_if'_i\in IJ'$.
\end{proof}

In particular, by applying this to the surjection $A \arr k$, we see that any lifting of $Z_{0}$ is a local complete intersection.

Now take a family $\mathcal{U}=\{U_i\}_{i \in I}$ of open subschemes of $X$, such that $Z$ is a complete intersection in each of the corresponding open affine subschemes $V_i=U_i \times_{\spec \Lambda}\spec A$ of $X_A$ and $\{V_i\}_{i \in I}$ is a cover of $X_A$. By the Lemma, for each index $i$ we can choose a lifting $Z_i' \subseteq W_i=U_i\times_{\spec \Lambda}\spec A'$ of $Z\cap V_i\subseteq V_i$.

For each pair of indexes $i,j$ the restrictions $Z_i'\cap W_{ij}$ and $Z_j'\cap W_{ij}$ are both liftings of $Z\cap V_{ij}\subseteq V_{ij}$, so by Theorem~\ref{actionthm} we have a unique element
$$
h_{ij} \in I\otimes_k T_{Z_0\cap U_{ij}}\hilb^{U_{ij}}\simeq \Gamma(Z_0\cap U_{ij},I\otimes_k \mathcal{N}_0)
$$
such that
$$
(Z_i' \cap W_{ij}\subseteq W_{ij}) + h_{ij} = Z_j'\cap W_{ij}\subseteq W_{ij}\,.
$$
From the fact that the action of $\Gamma(Z_0\cap U_{ij},I\otimes_k \mathcal{N}_0)$ on the liftings is free (and compatible with restriction to open subsets, as can be easily checked), for every triplet of indices $i,j,k$ we have
$$
h_{ij}+h_{jk}=h_{ik}
$$
so that $\{h_{ij}\}_{i,j \in I}$ is a \v{C}ech 1-cocycle for the sheaf $I\otimes_k \mathcal{N}_0$ (notice that $\mathcal{U}\cap Z_0=\{Z_0\cap U_i\}_{i \in I}$ is an open affine cover of $Z_0$).

Using a similar reasoning one checks that the cohomology class of this cocyle is independent of the choice of the liftings $Z_i'$. We call this class $\omega(Z\subseteq X_A,A')$. As in the preceding cases it is also easy to see that this class does not depend on the choice of the open cover $\{U_i\}_{i \in I}$, and that we have a lifting $Z'\subseteq X_{A'}$ if and only if $\omega(Z\subseteq X_A,A')=0$.

In fact this corresponds exactly to the situation in which the restrictions of the liftings $Z_i'$ on the intersections $W_{ij}$ are compatible, and can be used to define a global lifting (notice that we do not have infinitesimal automorphisms, so in this case these restrictions are equal, and not only isomorphic).

Finally the functoriality property is immediate from that of the action of the tangent space.
\end{proof}

\begin{proof}[Proof of Proposition~\ref{vanish}]\label{prf-vanish}
The subscheme $Z_0\subseteq \P^n_k$ has an obstruction theory with space $\H^1(Z_0,\mathcal{N}_0)$, which in this case is trivial, as we already saw in the proof of Proposition~\ref{hypers}. In other words $Z_0$ is unobstructed, so any $Z\in \hilb^{\P^n_k}_{Z_0}(A)$ can be lifted along any small extension $A'\to A$.
\end{proof}

\subsubsection{Quasi-coherent sheaves}

Now we turn to the case of deformations of quasi-coherent sheaves. Take $\Lambda=k$ and the deformation category $\qcoh^X\to \art^\op$, and consider a quasi-coherent sheaf $\E_0 \in \qcoh^X(k)$ (notice that $X_0=X$ in this case).

\begin{thm}\label{shobs}
There is an obstruction theory $(V_\omega, \omega)$ for $\E_0$, with obstruction space
$$
V_\omega=\Ext^2_{\O_X}(\E_0,\E_0)\,.
$$
\end{thm}

\begin{proof}
Take a small extension $A'\to A$ with kernel $I$, and an object $\E\in \qcoh^X_{\E_0}(A)$, which is a quasi-coherent sheaf on $X_A=X\times_{\spec k}\spec A$ with an isomorphism $\E \otimes_A k\simeq \E_0$. We construct the obstruction $\omega(\E,A')$.

Take the exact sequence of $A'$-modules
$$
\xymatrix{
0\ar[r] & I\ar[r] & \m_{A'}\ar[r] & \m_A \ar[r] & 0
}
$$
and notice that, since $\m_{A'}\cdot I=0$, $\m_{A'}$ is also an $A$-module (and $I$ is too because $I^2=(0)$), so that the sequence above is also an exact sequence of $A$-modules. We tensor it with $\E$ to get (by flatness)
$$
\xymatrix{
0\ar[r] & I\otimes_k \E_0 \ar[r] & \m_{A'}\otimes_A \E \ar[r] & \m_A\otimes_A \E \ar[r] & 0
}
$$
(since $I\otimes_A \E\simeq I\otimes_k (k\otimes_A \E)\simeq I \otimes_k \E_0$) which is an element $e \in \Ext^1_{\O_X}(\m_A\otimes_A \E,I \otimes_k \E_0)$.

We consider then the exact sequence of $A$-modules
$$
\xymatrix{
0\ar[r] & \m_A \ar[r] & A\ar[r] & k \ar[r] & 0
}
$$
and tensor it with $\E$, getting (by flatness again)
$$
\xymatrix{
0\ar[r] & \m_A\otimes_A \E \ar[r] & \E\ar[r] & \E_0 \ar[r] & 0.
}
$$
This induces a long $\Ext$ exact sequence (taking $\Hom_{\O_X}(-,I\otimes_k \E_0)$) that contains in particular the following piece
$$
\xymatrix{
\Ext^1_{\O_X}(\E,I\otimes_k \E_0)\ar[r]^-\gamma & \Ext^1_{\O_X}(\m_A\otimes_A \E,I\otimes_k \E_0)\ar[r]^-\delta & \Ext^2_{\O_X}(\E_0, I\otimes_k \E_0).
}
$$
We take as $\omega(\E,A')$ the element $\delta(e) \in \Ext^2_{\O_X}(\E_0, I\otimes_k \E_0)\simeq I\otimes_k \Ext^2_{\O_X}(\E_0,\E_0)$.

Now we have to verify that $\E$ has a lifting to $A'$ if and only if $\delta(e)=0$. Suppose first that $\E$ has a lifting $\E'\in \qcoh^X_{\E_0}(A')$. Then notice that $\m_{A'}\otimes_{A'}\E'\simeq (\m_{A'}\otimes_{A'} A)\otimes_A \E \simeq \m_{A'}\otimes_A \E$ (because $\m_{A'}\otimes_{A'} A\simeq \m_{A'}$, since $\m_{A'}$ is already an $A$-module). Tensoring the diagram with exact rows
$$
\xymatrix{
0\ar[r] & I \ar[r]\ar@{=}[d] &\m_{A'}\ar[r]\ar[d] & \m_A \ar[r]\ar[d] & 0\\
0\ar[r] & I \ar[r] & A' \ar[r] & A \ar[r] & 0
}
$$
with $\E'$, we get
$$
\xymatrix{
0\ar[r] & I\otimes_k \E_0 \ar[r]\ar@{=}[d] &\m_{A'}\otimes_A \E \ar[r]\ar[d] & \m_A\otimes_A \E \ar[r]\ar[d]^f & 0\\
0\ar[r] & I\otimes_k \E_0 \ar[r] & \E' \ar[r] & \E \ar[r] & 0
}
$$
where the top row is the extension $e$ obtained before.

But this diagram implies that $e$ is obtained by pullback from an extension in $\Ext^1_{\O_X}(\E,I\otimes_k \E_0)$ (the bottom row), so that it is in the image of the map $\gamma$; then we have $\delta(e)=0$.

Conversely, suppose that $\delta(e)=0$. Then by exactness of the $\Ext$ long exact sequence above, $e$ is in the image of the map $\gamma$. In other words we have a commutative diagram of $\O_X$-modules with exact rows
\begin{equation}\label{sheafobs}
\xymatrix{
0\ar[r] & I\otimes_k \E_0 \ar[r]\ar@{=}[d] &\m_{A'}\otimes_A \E \ar[r]\ar[d]^g & \m_A\otimes_A \E \ar[r]\ar[d]^f & 0\\
0\ar[r] & I\otimes_k \E_0 \ar[r] & \F \ar[r]^\pi & \E \ar[r] & 0
}
\end{equation}
where $\F$ is an $\O_X$-module.

We define a structure of $\O_{X_{A'}}$-module on $\F$ in the following way: since
$$
\O_{X_{A'}}\simeq \O_X\otimes_k A'\simeq \O_X\oplus(\m_{A'}\otimes_k \O_X)
$$
(because $A'\simeq k\oplus \m_{A'}$ as a $k$-vector space)
we only need to define $x\cdot s$ where $x$ is a section of $\m_{A'}\otimes_k \O_X$ and $s$ one of $\F$. Given two such sections $x=a'\otimes t$ and $s$, we define then
$$
(a'\otimes t)\cdot s= g(a'\otimes \pi(ts)) \in \F\,.
$$
(notice that $g$ is injective, since $f$ is by flatness of $\E$). It is readily checked that this gives a structure of $\O_{X_{A'}}$-module to $\F$. Moreover $\F$ is quasi-coherent, because it is an extension of two quasi-coherent sheaves.

Finally, we notice that the natural homomorphism $\F\otimes_{A'} A \to \E$ induced by $\F\to \E$ above is an isomorphism (it suffices to tensor the second row of (\ref{sheafobs}) with $A$ over $A'$), and from the local flatness criterion (Theorem~\ref{local-flatness}) we have (as the reader can check) that $\F$ is flat over $A'$.

In conclusion, $\F$ is a lifting of $\E$ to $A'$. Functoriality of the obstruction defined is immediate from the construction.
\end{proof}

If $\E_0$ is locally free, than we have
$$
\Ext^2_{\O_{X}}(\E_0,\E_0)\simeq \H^2(X,\mathcal{E}nd_{\O_{X}}(\E_0))
$$
and in this case (with the additional hypothesis that $X$ is separated) Theorem~\ref{shobs} can be proved using \v{C}ech cohomology, in the same way as we did for Theorem~\ref{smoothobstr}.

In particular if $X$ is affine, or of dimension at most 1, then every locally free sheaf is unobstructed.

\begin{ex}
Take the affine curve $X_0\subseteq \mathbb{A}^2_k$ over $k$ defined by the equation
$$
y^2=x^{2}(x-1)
$$
so that the origin $p=(0,0)$ is a singular point of $X_0$, and set $\E_0 \eqdef \O_{p}$, the pushforward of the structure sheaf of the point $p=\spec k$ along the morphism $\spec k\to X_0$ with image $p$. Consider the tangent cone $C_pX_0$ of $X_0$ at $p$, which is a union of two lines contained properly in the tangent space $T_pX_0$, which is two-dimensional, and take a tangent vector $v \in T_pX_0 \setminus C_pX_0$.

We see $v$ as a morphism $v\colon \spec\dual\to X_0$ in the usual way, and notice that it gives a section $\spec\dual\to X_\eps$ of the structure morphism $X_\eps\to \spec\dual$ (where as usual $X_\eps$ is the trivial deformation of $X_0$ over $\dual$).
$$
\xymatrix@-5pt{
\spec\dual\ar@/^10pt/[rrd]^v\ar@/_10pt/[ddr]_{\id}\ar@{-->}[rd] & & \\
& X_\eps\ar[r]\ar[d] & X_0\ar[d]\\
& \spec\dual\ar[r] & \spec k
}
$$
Moreover the image of this section, call it $Z_1 \subseteq X_\eps$, is closed, because $X_\eps\to \spec\dual$ is a separated morphism.

Now set $\E_1 \eqdef \O_{Z_1} \in \qcoh^{X_0}_{\E_0}(\dual)$, which is again the pushforward on $X_\eps$ of the structure sheaf of $\spec\dual$. The sheaf $\E_1$ is flat over $\dual$. Arguing by contradiction, show that there exists $n\geq 2$ such $\E_1$ cannot be lifted to $R_n=k[t]/(t^{n+1})$, so that $\E_0$ must be obstructed.
\end{ex}

\section{Formal deformations}\label{capform}

After developing tools to study infinitesimal deformations, in this section we go one step further and put together infinitesimal deformations that are successive liftings of a fixed $\xi_0 \in \F(k)$ to higher and higher orders. A collection of such liftings is said to be a \emph{formal deformation}.

We will introduce particular types of formal deformations, called \emph{universal} and \emph{versal} respectively, which play a key role in the theory. We will state and prove an analogue of Schlessinger's Theorem on existence of versal deformations for deformation categories. Finally, we will give some applications to obstruction theories, and consider briefly the problem of algebraization of formal deformations.

Throughout this section we will use some notation and results about noetherian local complete $\Lambda$-algebras that can be found in Appendix~\ref{appb}.

\subsection{Formal objects}

Let $\fib$ be a deformation category, and $R \in \compl$ (recall that $\compl$ denotes the category of noetherian local complete $\Lambda$-algebras with residue field $k$). We want to consider sequences of compatible deformations on the quotients $R_n \eqdef R/\m_R^{n+1}$: the idea is that $\spec R$ should be a little piece of the base scheme $S$ of a deformation we are trying to construct or study: for example it could be the spectrum of the completion of the local ring $\O_{S,s_0}$ of that base scheme at a point $s_0$, and we consider then sequences of compatible deformations on all the ``thickenings'' $\spec(\O_{S,s_0}/\m_{s_0}^{n+1})$ of the point $s_0=\spec k(s_0)$, hoping to get an actual deformation over $\spec\wh{\O}_{S,s_0}$.

\begin{defin}
A \gr{formal object} of $\F$ over $R$ is a collection $\xi=\{\xi_n, f_n\}_\nin$, where $\xi_n$ is an object of $\F(R_n)$ and $f_n\colon \xi_{n}\to \xi_{n+1}$ is an arrow of $\F$ over the canonical projection $\pi_{n}\colon R_{n+1}\to R_n$ (or more properly over the arrow $\pi_{n}^{\op}$ of $\artl^{\op}$ corresponding to the projection, according to the notation described on page \pageref{over_alg}).
\end{defin}

Sometimes we will call $\xi_n$ the \gr{term of order} $n$ of $\xi$, and say that $R$ is the \gr{base ring} of the formal object $\xi$. If we need to specify it in the notation, we will denote a formal object as above by $(R,\xi)$.

The condition of having fixed arrows $f_n\colon \xi_n \to \xi_{n+1}$ reflects the fact that the objects $\xi_n$ are compatible, in the sense that if $n\geq m$, then the pullback of $\xi_n$ to $R_m$ along the projection $R_n\to R_m$ is isomorphic to $\xi_m$, and moreover we have a canonical isomorphism, coming from the composite $\xi_m\to \xi_{m+1}\to \cdots \to \xi_n$ of the given arrows.

We also remark that a formal deformation is determined up to a unique isomorphism by the objects $\xi_n$ and the arrows $\xi_{n} \arr \xi_{n+1}$ for $n \geq n_{0}$, for any natural number $n_{0}$. This is because $\xi_n$ determines all the $\xi_i$'s with $i \leq n$, by taking pullbacks along the projections $R_n\to R_i$.


\begin{defin}
A \gr{morphism} $\alpha\colon \xi \to \eta$ of formal objects over $R$, where $\xi=\{\xi_n,f_n\}_\nin$ and $\eta=\{\eta_n,g_n\}_\nin$ is a collection $\alpha=\{\alpha_n\}_\nin$ of arrows $\alpha_n\colon \xi_n \to \eta_n$ of $\F(R_n)$, such that for every $n$ the diagram
$$
\xymatrix{
\xi_{n}\ar[r]^{f_n}\ar[d]_{\alpha_{n}} & \xi_{n+1}\ar[d]^{\alpha_{n+1}}\\
\eta_{n}\ar[r]_{g_n} & \eta_{n+1}
}
$$
commutes.
\end{defin}

As with objects, $\alpha_n$ will sometimes be called the \gr{term of order} $n$ of $\alpha$.

Formal objects over a fixed $R$ with morphisms form a category (composition of arrows is defined as composition at each order), which we call the \gr{category of formal objects} over $R$ and denote by $\wh{\F}(R)$.

Here we used the canonical filtration $\{\m_R^n\}_\nin$, but to define a formal object we can use any filtration that defines the right topology on $R$. Let $\mathcal{A}=\{I_n\}_\nin$ be a filtration of $R$, that is, a sequence of ideals $I_n$ such that $I_n\subseteq I_m$ whenever $n\geq m$, and inducing on $R$ the same topology as the canonical filtration.

We can consider then a category $\wh{\F}_{\mathcal{A}}(R)$ of formal objects over $R$ with respect to the filtration $\mathcal{A}$, defined as above, but using the quotients $R/I_n$ instead of $R_n=R/\m_R^{n+1}$.

\begin{ex}\label{filtration}
For any $R$ and filtration $\mathcal{A}=\{I_n\}_\nin$ that defines the $\m_R$-adic topology on $R$, the categories $\wh{\F}_\mathcal{A}(R)$ and $\wh{\F}(R)$ are equivalent.
\end{ex}

This says that we can use any filtration as above to define a formal object of $\F$ over $R$.

The notation $\wh{\F}(R)$ suggests that we want to consider a fibered category $\wh{\F}\to \compl^\op$, which is indeed the case.

\begin{defin}
A \gr{morphism} $\alpha\colon (R,\xi)\to (S,\eta)$ of formal objects of $\F$, where $\xi=\{\xi_n,f_n\}_\nin$ and $\eta=\{\eta_n,g_n\}_\nin$, is a pair $(\alpha,\phi)$, where $\phi\colon S\to R$ is a homomorphism, and $\alpha=\{\alpha_n\}_\nin$ is a collection of arrows $\alpha_n\colon \xi_n\to \eta_n$ of $\F$ over $\phi_n\colon S_n\to R_n$, such that for every $n$ the diagram
$$
\xymatrix{
\xi_{n}\ar[r]^{f_n}\ar[d]_{\alpha_{n}} & \xi_{n+1}\ar[d]^{\alpha_{n+1}}\\
\eta_{n}\ar[r]_{g_n} & \eta_{n+1}
}
$$
commutes.
\end{defin}

Again, sometimes we will call $\alpha_n$ the \gr{term of order} $n$ of $(\alpha, \phi)$. We stress that a morphism between formal objects $(R,\xi)\to (S,\eta)$ determines a homomorphism in the \gr{opposite} direction $S\to R$ (the same happens in the case of infinitesimal deformations, where for $\xi \in \F(A)$ and $\eta \in \F(A')$, an arrow $\xi\to \eta$ gives a homomorphism $A'\to A$).

We define a category $\wh{\F}$, and call it the \gr{category of formal objects} of $\F$: its objects are formal objects $(R,\xi)$, and an arrow $(R,\xi)\to (S,\eta)$ is a morphism of formal objects. We have a functor $\wh{\F}\to \compl^\op$ that takes $(R,\xi)$ to $R$ and an arrow $(\alpha,\phi)\colon (R,\xi)\to (S,\eta)$ to the arrow $\phi^{\op}$ from $R$ to $S$.

\begin{prop}
$\wh{\F}\to \compl^\op$ is a category fibered in groupoids.
\end{prop}

The proof is easy, so we leave the details as an exercise: given a homomorphism $\phi\colon S\to R$ in $\compl$, and a formal object $\eta=\{\eta_n,g_n\}_\nin$ over $S$, we can define a pullback to $R$ by taking pullbacks of the single terms $\eta_n$, and pullbacks of the arrows $g_n$. This shows that $\wh{\F}\to \compl^\op$ is a fibered category. The fact that each $\wh{\F}(R)$ is a groupoid follows from the fact that $\F$ is fibered in groupoids.

If we have two deformation categories $\fib$ and $\G\to \artl^\op$, and a morphism $F\colon \F\to \G$, then there is a natural induced morphism $\wh{F}\colon \wh{\F}\to \wh{\G}$ of categories fibered in groupoids: a formal object $\xi=\{\xi_n,f_n\}_\nin$ of $\F$ over $R$ goes to the formal object $\wh{F}(\xi)=\{F(\xi_n),F(f_n)\}_\nin$ of $\G$ over the same $R$, and an arrow $\alpha=\{\alpha_n\}_\nin$ goes to the arrow $\wh{F}(\alpha)=\{F(\alpha_n)\}_\nin$. It is immediate to check that this is well-defined, and gives a morphism of categories fibered in groupoids.

Now we show that $\F$ is a subcategory of $\wh{\F}$. First notice that if $A \in \artl$, then in particular $A \in \compl$, so we can consider the fiber category $\wh{\F}(A)$.

\begin{prop}\label{subcat}
We have an equivalence of categories $\wh{\F}(A)\simeq \F(A)$. Moreover these equivalences give rise to an equivalence of categories fibered in groupoids  $F\colon \F\to \wh{\F}|_{\artl^\op}$, so that $\F$ can be regarded as a full subcategory of $\wh{\F}$.
\end{prop}

\begin{proof}
In the statement above with $\wh{\F}|_{\artl^\op}$ we mean the full subcategory of $\wh{\F}$ whose objects are formal objects $(A,\xi)$ of $\F$ with $A \in \artl$.

The idea of the proof is quite clear: since $A$ is artinian its maximal ideal $\m_A$ is nilpotent, so that there exists an $m$ such that $\m_A^{m+1}=(0)$, and then $A_i =A$ for all $i\geq m$; because of this, a formal deformation will be completely determined (up to isomorphism) by its term of order $m$.

There is an obvious functor $F\colon \wh{\F}(A)\to \F(A)$ that carries a formal object $\{\xi_n,f_n\}_\nin$ into the object $\xi_m \in \F(A)$, and an arrow $\alpha=\{\alpha_n\}\colon \{\xi_n,f_n\}_\nin\to \{\eta_n,g_n\}_\nin$ into $\alpha_m\colon \xi_m\to \eta_m$.

We construct a quasi inverse $G\colon \F(A)\to \wh{\F}(A)$ as follows: given an object $\xi \in\F(A)$, for $i\leq m-1$ we can consider the pullbacks $\xi_i$ of $\xi$ along the quotient maps $A \to A/\m_A^{i+1}=A_i$, and the canonical arrows $f_i\colon \xi_i\to \xi_{i+1}$, identifying $\xi_i$ as a pullback of $\xi_{i+1}$, and for $i\geq m$, we set $\xi_i \eqdef \xi \in \F(A)$, and $f_i=\id\colon \xi\to \xi$. Then $G(\xi)=\{\xi_n,f_n\}_\nin$ is an object of $\wh{\F}(A)$. Moreover, if $\eta$ is another object of $\F(A)$, with $G(\eta)=\{\eta_n,g_n\}_\nin$, and $\alpha\colon \xi\to \eta$ is an arrow in $\F(A)$, we define an arrow $G(\alpha)=\{\alpha_n\}_\nin\colon G(\xi)\to G(\eta)$, taking for $i\leq m-1$ the arrow $\alpha_i\colon \xi_i\to\eta_i$ that is the pullback of $\alpha\colon \xi\to \eta$ to $A_i$, and for $i\geq m$, since $\xi_i=\xi$ and $\eta_i=\eta$, we take simply $\alpha\colon \xi\to \eta$.

The reader can check that $G$ and $F$ are quasi-inverse to each other, and that varying $A \in \artl$ we get a morphism $F\colon \wh{\F}|_{\artl^\op}\to \F$, which is an equivalence of categories fibered in groupoids (use Proposition~\ref{fibered.equiv}).
\end{proof}

Because of the preceding Proposition, we can talk about arrows from an object $\xi$ over $A \in \artl$ to a formal object $\eta=\{\eta_n,g_n\}_\nin$ over $R \in \compl$, using the identification above.

In particular giving an arrow $\xi\to\eta$ corresponds to giving a homomorphism $\phi\colon R\to A$ of $\Lambda$-algebras, and an isomorphism $\xi\simeq (\phi_m)_*(\eta_m)$ in $\F(A)$, where $m$ is the order of $A$, so that $\phi_m$ is a homomorphism $\phi_m\colon R_m\to A$. A pullback of $\eta$ to $A$ is $(\phi_m)_*(\eta_m) \in \F(A)$, where $m$ is as above.

Finally, if $A \in \artl$ and $\xi \in \F(A)$, we will sometimes write $(A,\xi)$ for the corresponding object in $\wh{\F}(A)$ defined in the proof above.

\begin{rmk}\label{morphism.approx}
Here we point out (using again the identification given by Proposition~\ref{subcat}) that giving a morphism of formal objects is equivalent to giving a sequence of morphisms from the artinian quotients of the source, compatible with the projections.

Let $(R,\xi), (S,\eta)$ be two formal objects of $\F$, where $\xi=\{\xi_n,f_n\}_\nin$ and $\eta=\{\eta_n,g_n\}_\nin$; call $A$ the set of arrows $(R,\xi) \to (S,\eta)$ in $\wh{\F}$, and $B$ the set of sequences $\{h_n\}_\nin$ of morphisms of formal objects $h_n\colon (R_n,\xi_n)\to (S,\eta)$ such that for every $n$ the composite
$$
(R_n,\xi_n)\arr (R_{n+1},\xi_{n+1}) \xrightarrow{h_{n+1}} (S,\eta)
$$
coincides with $h_n$.

There is a natural map $A\to B$ sending a morphism $(R,\xi)\to (S,\eta)$ into the sequence of composites $(R_n,\xi_n)\to (R,\xi)\to (S,\eta)$. Conversely given a sequence $h_n$ as above, the arrow $h_n\colon (R_n,\xi_n)\to (S,\eta)$ of $\wh{\F}$ corresponds to an isomorphism between $\xi_n$ and the pullback of $\eta_n \in \F(S_n)$ to $R_n$, which gives an arrow $\alpha_n\colon \xi_n\to \eta_n$ of $\F$. The fact that $\{h_n\}_\nin$ has the compatibility property above ensures that $\alpha=\{\alpha_n\}_\nin$ gives an arrow of formal objects $\alpha\colon (R,\xi)\to (S,\eta)$, and this gives a map $B\to A$ that is clearly inverse to the previous one.
\end{rmk}

\subsubsection{Formal objects as morphisms}

Now we change point of view, and describe formal objects as morphisms of categories fibered in groupoids. For an object $R \in \compl$ consider the opposite $(\Art_R)^\op$ of the category of local artinian $R$-algebras with residue field $k$, or equivalently the opposite of the category of objects $A \in \artl$ with a homomorphism of $\Lambda$-algebras $R\to A$.

There is an obvious functor $(\Art_R)^\op\to \artl^\op$ that sends an object $A$ of $(\Art_R)^\op$ into the $\Lambda$-algebra $A$ defined by the composite $\Lambda\to R\to A$ (where $R\to A$ defines the structure of $R$-algebra on $A$), and an arrow $(A\to B)^{\op}$ in $(\Art_R)^\op$ to ``itself'' (that is, to the opposite of the homomorphism $A\to B$ seen as a homomorphism of $\Lambda$-algebras).

\begin{prop}\label{prorapp}
$(\Art_R)^\op\to \artl^\op$ is a deformation category. Moreover its tangent space $T_{R\to k}(\Art_R)^\op$ at the unique object $R\to k$ over $k$ is isomorphic to the vertical tangent space of $R$
$$
T_{R\to k}(\Art_R)^\op\simeq T_{\Lambda}R=(\m_R/(\m_\Lambda R+\m_R^2))^\vee\,.
$$
\end{prop}

\begin{proof}
It is easy to check that $(\Art_R)^\op\to \artl^\op$ is a deformation category, so we only calculate the tangent space.

Notice first that $(\Art_R)^\op(k)=\Hom_\Lambda(R,k)$ has precisely one element, which is the quotient map $R\to k$. To find the tangent space, we consider the functor $F\colon \fvect\to \set$ that associates to $V\in \fvect$ the set $F(V) \eqdef \Hom_\Lambda(R,k \oplus V)$, and acts on arrows by pullback.

We will show that there is a functorial bijection
$$
F(V)\simeq V\otimes_k T_\Lambda R
$$
where $T_\Lambda R$ is the vertical tangent space of $R$
$$
T_\Lambda R=(\m_R/(\m_\Lambda R+\m_R^2))^\vee\,.
$$
This will give an isomorphism $T_{R\to k}(\Art_R)^\op\simeq T_\Lambda R$.

We construct a function $F(V)\to V\otimes_k T_\Lambda R$. Take a homomorphism of $\Lambda$-algebras $\phi\colon R\to k\oplus V$. We can restrict $\phi$ to the maximal ideal $\m_R$ of $R$ to get a function $\overline{\phi}\colon \m_R\to V$, and since $\phi(\m_\Lambda R+\m_R^2)=0$ (for $\phi(\m_\Lambda R)=0$ and $\phi(\m_R^2)=0$, respectively because of $\Lambda$-linearity of $\phi$ and $V^2=(0)$), we can pass $\overline{\phi}$ to the quotient to get a $k$-linear function
$$
f_{\phi}\colon \m_R/(\m_\Lambda R+\m_R^2)\arr V
$$
which is an element of
$$
\Hom_k(\m_R/(\m_\Lambda R+\m_R^2),V)\simeq V\otimes_k (\m_R/(\m_\Lambda R+\m_R^2))^\vee\,.
$$
Conversely, suppose we have an element of $V\otimes_k (\m_R/(\m_\Lambda R+\m_R^2))^\vee$ that corresponds then to a $k$-linear function
$$
f\colon \m_R/(\m_\Lambda R+\m_R^2)\arr V\,.
$$
Since $\m_R/(\m_\Lambda R+\m_R^2)\simeq \m_{\overline{R}}/\m_{\overline{R}}^2$ (where $\overline{R}=R/\m_\Lambda R$, see Appendix~\ref{appb}), we can consider the composite $g\colon \m_{\overline{R}}\arr\m_R/(\m_\Lambda R+\m_R^2)\to V$, and define $\phi_f\colon R\to k \oplus V$ as
$$
\phi_f(r) \eqdef \pi(r)+g(\pi'(r)-\pi(r))\,.
$$
where $\pi\colon R\to k$ and $\pi'\colon R\to R/\m_\Lambda R=\overline{R}$ are the quotient maps (we are using the fact that $\overline{R}$ is a $k$-algebra, so $\pi(r) \in \overline{R}$).

It can readily be checked that these two functions are inverse to each other, and we have our bijection. Functoriality is immediate.
\end{proof}

Now consider a morphism $\xi\colon (\Art_R)^\op\to \F$ of deformation categories. From $\xi$ we get a formal object of $\F$ over $R$, taking $\xi_n=\xi(R_n)$, and $f_n=\xi((R_{n+1}\to R_n)^{\op})$ (where $R_{n+1}\to R_{n}$ is the projection), and if we have a base-preserving natural transformation $\alpha\colon \xi\to \eta$ between two morphisms $(\Art_R)^\op\to \F$ we get an arrow $\alpha\colon \{\xi_n,f_n\}_\nin\to \{\eta_n,g_n\}_\nin$ taking $\alpha_n=\alpha(R_n)\colon \xi_n\to\eta_n$.

This association gives a functor $\Phi\colon \Hom((\Art_R)^\op,\F)\to \wh{\F}(R)$, defined by $\Phi(\xi)=\{\xi_{n},f_{n}\}_\nin$ and $\Phi(\alpha)=\{\alpha_{n}\}_{\nin}$. We have the following ``Yoneda-like'' Proposition.

\begin{prop}\label{yoneda-like}
The functor $\Phi$ is an equivalence of categories.
\end{prop}

\begin{proof}
We construct a quasi-inverse $\Psi\colon \wh{\F}(R)\to \Hom((\Art_R)^\op,\F)$ to $\Phi$. Given a formal object $\xi=\{\xi_n,f_n\}_\nin$, we get a morphism $F_\xi\colon (\Art_R)^\op\to\F$ in the following way: if $A \in (\Art_R)^\op$ we associate with $A$ the pullback $\xi_A=(\phi_m)_*(\xi_m) \in \F$, where $m$ is the order of $A$ and $\phi_m\colon R_m\to A$ is the homomorphism induced by $R\to A$. On arrows, if we have a homomorphism $\phi\colon A\to B$ in $(\Art_R)$, the commutative diagram
$$
\xymatrix@!C@!R@-10pt{
\xi_{m} \ar@{|->}[dd] & & \xi_A \ar[ll]\ar|\hole@{|->}[dd] &\\
&\xi_{n} \ar@{|->}[dd] \ar[lu]& & \xi_B \ar[ll] \ar@{|->}[dd]\ar@{-->}[lu]_{\xi_\phi}\\
R_m \ar[dr] \ar|\hole[rr]& & A\ar[dr]^\phi & \\
& R_n \ar[rr] & & B.
}
$$
gives (by pullback in $\F$) an arrow $\xi_\phi\colon \xi_B \to \xi_A$ of $\F$ over $\phi$ (as an arrow in $\artl$).

This defines a morphism $F_\xi\colon (\Art_R)^\op\to\F$. With an arrow $\alpha=\{\alpha_n\}_\nin\colon \xi\to \eta$ between two formal objects over $R$, where $\eta=\{\eta_n,g_n\}_\nin$, we associate a natural transformation $F_\alpha\colon F_\xi\to F_\eta$. Given an object $A \in (\Art_R)^\op$ of order $m$, we define an arrow $F_\alpha(A)\colon F_\xi(A)\to F_\eta(A)$ simply as the pullback of $\alpha_m\colon \xi_m \to \eta_m$, along the homomorphism $R_m\to A$. Standard arguments show that this gives a natural transformation, and this completes the definition of $\Psi$.

Routine verifications using the universal property of pullbacks prove that $\Psi$ and $\Phi$ are quasi-inverse to each other, and then our result.
\end{proof}

So giving a formal object of $\F$ over $R$ is equivalent to giving a morphism of deformation categories $(\Art_R)^\op\to \F$. From now on we will use both these points of view.

In particular we will use the same symbol for a formal object and for the associated morphism, and if $\xi\colon (\Art_R)^\op\to \F$ is a formal object and $\phi\colon R\to A$ a homomorphism of $\Lambda$-algebras, we will denote by $\xi_{\phi}$ (or simply $\xi_{R\to A}$ when there is no possibility of confusion) the object $\xi(R\to A)$ of $\F(A)$.

We get the following corollary (which is an analogue of the ``weak'' Yoneda's Lemma), simply by taking $\F=(\Art_R')$.

\begin{cor}
There is a bijection $\Hom\bigl((\Art_R)^\op,(\Art_R')^\op\bigr)\simeq \Hom_{\Lambda}(R',R)$ which respects composition.

In particular $(\Art_R)^\op$ and $(\Art_R')^\op$ are equivalent as fibered categories if and only if $R$ ad $R'$ are isomorphic as $\Lambda$-algebras.
\end{cor}

By ``respects composition'' above we mean the following: if $R''$ is an object of $\compl$, and $F\colon (\Art_R)^\op\to(\Art_R')^\op$, $G\colon (\Art_R')^\op\to(\Art_R'')^\op$ are two morphisms corresponding to $\phi\colon R'\to R$ and $\psi\colon  R''\to R'$ respectively, then $G\circ F \in \Hom((\Art_R)^\op,(\Art_R'')^\op)$ corresponds to $\phi\circ \psi \in \Hom_{\Lambda}(R'',R)$.

Notice also that $\Hom((\Art_R)^\op,(\Art_R')^\op)$ is a set, since $(\Art_R')^\op$ is fibered in sets. As to the proof, bijectivity is immediate from Proposition~\ref{yoneda-like}, and the part about respecting composition is easy.

\begin{rmk}\label{morph.pback}
From this description of formal objects we get another interpretation of the pullback: if $(R,\xi)$ is a formal object of $\F$, and $\phi\colon R\to S$ is a homomorphism in $\compl$, then from the last corollary we have an associated morphism $\overline{\phi}\colon (\Art_S)^\op \to (\Art_R)^\op$, and we can consider the composite
$$
(\Art_S)^\op \stackrel{\overline{\phi}}{\larr} (\Art_R)^\op \stackrel{\xi}{\larr} \F
$$
which is a formal object of $\F$ over $S$. One can easily see that this formal object is (up to a unique isomorphism) precisely the pullback of $\xi$ to $S$ along $\phi$.
\end{rmk}

\subsubsection{The Kodaira--Spencer map}

Given a formal object $(R,\xi)$ of $\F$, we can consider the differential at the only object over $k$ of $(\Art_R)^\op$ of the corresponding morphism $\xi\colon (\Art_R)^\op\to \F$.

\begin{defin}
The $k$-linear function $d_{R\to k}\xi\colon T_{R\to k}(\Art_R)^\op\to T_{\xi_0}\F$ is called the \gr{Kodaira--Spencer map} of the formal object $(R,\xi)$. We will usually denote it by $\kappa_{\xi}\colon T_\Lambda R\to T_{\xi_0}\F$.
\end{defin}

More explicitly, the Kodaira--Spencer map can be described in the following way: if $\phi\colon R\to \dual$ is an element of $T_{R\to k}(\Art_R)^\op$ (we do not need to take isomorphism classes here, for $(\Art_R)^\op$ is fibered in sets), the image $\kappa_\xi(\phi)$ is the isomorphism class of the pullback of the formal object $\xi$ along the map $\phi$.

Notice that, since $\eps^2=0$ and by $\Lambda$-linearity, $\phi$ will factor through the quotient map $R\to\overline{R}_1$, and $\kappa_\xi(\phi)$ can be described as the isomorphism class of the pullback of $\overline{\xi}_1 \in \F(\overline{R}_1)$ along the induced map $\overline{R}_1\to \dual$, where $\overline{\xi}_1$ is the pullback of the formal object $\xi$ along the quotient map above.

There is another natural $k$-linear map $T_\Lambda R\to T_{\xi_0}\F$ associated with $(R,\xi)$, defined in the following way: consider the object $\overline{\xi}_1 \in \F(\overline{R}_1)$ as above. This is a lifting of $\xi_0$ to $\overline{R}_1$, and since $\overline{R}_1$ is a $k$-algebra, we can compare it to the trivial lifting $\xi_0|_{\overline{R}_1}$ of $\xi_0$ to $\overline{R}_1$.

Since these objects are liftings of $\xi_0$ to $\overline{R}_1$, and
$$
\xymatrix{
0\ar[r] & \m_{\overline{R}}/\m_{\overline{R}}^2\ar[r] & \overline{R}_1 \ar[r] & k\ar[r] & 0
}
$$
is a small extension, we get an element
$$
[\overline{\xi}_1] - [\xi_0|_{\overline{R}_1}] \in (\m_{\overline{R}}/\m_{\overline{R}}^2)\otimes_k T_{\xi_0}\F
$$
(using the notation introduced in Remark~\ref{action}) that we call the \gr{Kodaira--Spencer class} of $\xi$, and denote by $k_\xi$.

This element corresponds to a $k$-linear function
$$
T_\Lambda R\simeq (\m_{\overline{R}}/\m_{\overline{R}}^2)^\vee\arr T_{\xi_0}\F\,.
$$

\begin{ex}\label{kod-s}
Show that this last map coincides with the Kodaira--Spencer map $\kappa_\xi$ of $\xi$.
\end{ex}

This fact shows in particular that $\kappa_\xi$ is completely determined by the first-order term $\xi_1 \in \F(R_1)$ of $\xi$, because this determines $\overline{\xi}_1 \in \F(\overline{R}_1)$ up to isomorphism. Conversely, by the freeness of the action, $\overline{\xi}_1$ is determined (up to isomorphism) once we know $\kappa_\xi$.

The following functoriality property is an immediate consequence of Proposition~\ref{funct}. Take two objects $R, S$ of $\compl$,  a homomorphism $\phi\colon R\to S$, and $\xi=\{\xi_n,f_n\}_\nin$ a formal object of $\F$ over $R$.

Call $\overline{\phi}_1\colon \overline{R}_1\to \overline{S}_1$ the induced homomorphism, which in turn induces a morphism of small extensions
$$
\xymatrix{
0\ar[r] & \m_{\overline{R}}/\m_{\overline{R}}^2 \ar[r]\ar[d]^{\psi} & \overline{R}_1 \ar[r]\ar[d]^{\overline{\phi}_1} & k \ar[r]\ar[d] & 0\\
0\ar[r] & \m_{\overline{S}}/\m_{\overline{S}}^2  \ar[r] & \overline{S}_1 \ar[r] & k\ar[r] & 0.
}
$$
Recall also that $\phi$ induces a differential $d\phi\colon T_\Lambda S\to T_\Lambda R$ that is the adjoint of the codifferential $\psi\colon \m_{\overline{R}}/\m_{\overline{R}}^2\to \m_{\overline{S}}/\m_{\overline{S}}^2$ (see Appendix~\ref{appb}).

\begin{prop}\label{kosfunct}
We have the following relations between the Kodaira--Spencer maps and classes of $\xi \in \wh{\F}(R)$ and of the pullback $\phi_*(\xi) \in \wh{\F}(S)$.
\begin{itemize}
\item $k_{\phi_*(\xi)}=(\psi\otimes\id)(k_\xi) \in \m_{\overline{S}}/\m_{\overline{S}}^2\otimes_k T_{\xi_0}\F$.
\item $\kappa_{\phi_*(\xi)}=\kappa_{\xi}\circ d\phi\colon T_\Lambda S=(\m_{\overline{S}}/\m_{\overline{S}}^2)^\vee\arr T_{\xi_0}\F$.
\end{itemize}
\end{prop}

\subsection{Universal and versal formal deformations}

Like the classical Yoneda's Lemma, Proposition~\ref{yoneda-like} lets us speak of ``universal formal objects'' for a deformation category $\F$.

\begin{defin}
A \gr{universal formal object} over $R \in \compl$ for $\F$ is a formal object $\xi \in \wh{\F}(R)$, such that the corresponding $\xi\colon (\Art_R)^\op\to \F$ is an equivalence of categories fibered in groupoids over $(\Art_\Lambda)^\op$.
\end{defin}

Thanks to Proposition~\ref{fibered.equiv}, $\xi$ is a universal formal object if and only if the induced morphism $\xi_A\colon (\Art_R)^\op(A)\to \F(A)$ is an equivalence for every $A \in \artl$, or equivalently if and only if for every $A \in \artl$ and $\eta \in \F(A)$ there exist a unique homomorphism of $\Lambda$-algebras $R\to A$ and a unique isomorphism $\xi_{R\to A}\simeq \eta$ in $\F(A)$. This can also be restated by saying that for every $A \in \artl$ and $\eta \in \F(A)$ there is a unique arrow $(A,\eta)\to (R,\xi)$ of formal objects in $\wh{\F}$.

Using Remark~\ref{morphism.approx} we see that the above universal property can be strengthened to: for every formal object $(S,\eta)$ of $\F$ there exists a unique arrow $(S,\eta)\to (R,\xi)$. That is, every formal object can be obtained as pullback of $(R,\xi)$, in a unique way. This can easily be checked by considering the sequence of arrows $h_n\colon (S_n,\eta_n)\to (R,\xi)$ coming from the ``weak'' universal property above, and noticing that they are necessarily compatible because of uniqueness.

Using this last universal property it is easy to check that two universal deformations are canonically isomorphic.

\begin{defin}
We say that a deformation category $\fib$ is \gr{prorepresentable} if it is equivalent to a deformation category of the form $(\Art_R)^\op$ for some $R \in \compl$, or equivalently if $\F$ has a universal formal object $(R,\xi)$.
\end{defin}

Since $(\Art_R)^\op$ is a category fibered in sets, a necessary condition for a deformation category $\fib$ to be prorepresentable is that $\F$ should be fibered in equivalence relations. Other necessary conditions are that $\F(k)$ should be a trivial groupoid, because it will be equivalent to a singleton, and $\F$ should have finite-dimensional tangent space $T_{\xi_0}\F$ at any (actually it suffices that this holds for one, given the former condition) object $\xi_0 \in \F(k)$, because $\dim_k T_{\Lambda}R$ is finite.

The main result of this section is that the converse also holds.

\begin{thm}[Schlessinger]\label{schlessinger}
Let $\fib$ be a deformation category. Then $\fib$ is prorepresentable if and only if:
\begin{enumeratea}
\item $\F(k)$ is a trivial groupoid.
\item $T_{\xi_0}\F$ is finite-dimensional for any $\xi_0 \in \F(k)$.
\item $\Inf(\xi_0)$ is trivial for any $\xi_0 \in \F(k)$.
\end{enumeratea}
\end{thm}

This is an analogue of Schlessinger's Theorem~\ref{schlessinger.classical} for deformation categories, even though there are no direct implications between the two. We will prove the Theorem in \S\ref{exist+schl}, after discussing miniversal deformations.

\begin{examp}
As a simple example, suppose that $X$ is a quasi-projective scheme over $k$, set $\Lambda=k$, and consider the deformation category $\hilb^X\to \art^\op$.

Let $Z_0$ be a closed subscheme of $X$ which is proper over $\spec k$, and notice that the deformation category $\hilb^X_{Z_0}\to \art^\op$ of objects restricting to $Z_0$ over $k$ meets all hypotheses of Theorem~\ref{schlessinger}: we have already seen that $\Inf_{Z_0}(\hilb^X)=0$, clearly the only object over $k$ is $Z_0$ itself, and the tangent space $T_{Z_0}\hilb^X\simeq \H^0(Z_0,\mathcal{N}_0)$ is finite-dimensional over $k$. Then we can conclude that the deformation category $\hilb^X_{Z_0}\to \art^\op$ is prorepresentable.

We can see this in a more concrete way: the deformation category (which is fibered in sets) $\hilb^X\to \art^\op$ comes from a functor, called the \gr{Hilbert functor} of $X$, and denoted by $\Hilb^X\colon (\Sch_k)\to \set$; a theorem of Grothendieck (see for example \cite[Chapter~5]{FGA}) states that if $X$ is quasi-projective this functor is represented by a scheme, called the \gr{Hilbert scheme} of $X$, which we still denote by $\Hilb^X\in (\Sch_k)$.

The closed subscheme $Z_0$ corresponds then to a point in the Hilbert scheme, $Z_0 \in \Hilb^X$. Since $\Hilb^X$ represents the Hilbert functor, every object $Z \in \hilb^X_{Z_0}(A)$ corresponds to a morphism $\spec A\to \Hilb^X$ with image $Z_0$, which factors through $\spec\wh{\O}_{\Hilb^X,Z_0}$, by the usual argument. In particular the resulting homomorphism $\wh{\O}_{\Hilb^X,Z_0}\to A$ gives an object of $(\Art_{\wh{\O}_{\Hilb^X,Z_0}})^\op$.

This gives a morphism
$$
\hilb^X_{Z_0}\arr (\Art_{\wh{\O}_{\Hilb^X,Z_0}})^\op
$$
of deformation categories, which is easily seen to be an equivalence. Then $\hilb^X_{Z_0}$ is prorepresentable, as we already knew from Theorem~\ref{schlessinger}. The universal formal object $(R,Z)$ in this case has $R=\wh{\O}_{\Hilb^X,Z_0}$, and the term of order $n$ of the formal deformation $Z=\{Z_n,f_n\}_\nin$ over $R$ is the pullback to $\spec(\O_{\Hilb^X,Z_0}/\m_{\Hilb^X,Z_0}^{n+1}{})$ of the universal closed subscheme of $\Hilb^X$.
\end{examp}

Notice the following fact. Let $F\colon (\Art_{\Lambda}) \arr \set$ be the functor of isomorphism classes of a deformation category $\F \arr (\Art_{\Lambda})^{\op}$. If $\F$ is prorepresentable, say equivalent to $(\Art_{R})^{\op}$ for $R\in \compl$, then $F$ is equivalent to the functor $\Hom_{\Lambda}(-, R)\colon (\Art_{\Lambda})^{\op} \arr \set$, so it is prorepresentable. On the other hand, there is no reason why $F$ being prorepresentable should imply that $\Inf(\xi_0)$ is trivial. For example, consider the deformation category $\catdef_{\P^{n}_{k}} \arr (\Art_{\Lambda})^{\op}$, for $n > 0$ and any $\Lambda$. Then $\H^{1}(\P^{n}, T_{\P^{n}_{k}}) = 0$, so its tangent space is trivial, and every lifting of $\P^{n}_{k}$ to any $A$ is isomorphic to $\P^{n}_{A}$, so the functor of isomorphism classes is prorepresented by $\Lambda$, and on the other hand the infinitesimal tangent space $\H^{0}(\P^{n}, T_{\P^{n}_{k}})$ is not zero, so the isomorphism is not unique, and $\catdef_{\P^{n}_{k}}$ is not prorepresentable.

This is a very simple example in which the fibered category gives somewhat finer information than the corresponding deformation functor, which doesn't take automorphisms into account. However, we see from the following criterion that the prorepresentability of the functor of isomorphism classes has a simple interpretation in terms of the fibered category.

\begin{ex}
Let $\F \arr (\Art_{\Lambda})^{\op}$ be a deformation category. Show that the functor $F\colon (\Art_{\Lambda}) \arr \set$ of isomorphism classes of $\F$ is prorepresentable if and only if the following conditions are satisfied.

\begin{enumeratea}

\item The groupoid $\F(k)$ is connected (i.e., all objects of $\F(k)$ are isomorphic).

\item $T_{\xi_0}\F$ is finite-dimensional for any $\xi_0 \in \F(k)$.

\item If $A' \arr A$ is a small extension in $(\Art_{\Lambda})$, $\xi'$ is an object over $A'$, and $\xi$ is a pullback of $\xi'$ to $A$, then the induced group homomorphism $\aut_{A'}(\xi') \arr \aut_{A}(\xi)$ is surjective.

\end{enumeratea}

Hint: assuming (a) and (b), $F$ is prorepresentable if and only if it it satisfies the RS condition.

Condition~(c) says that for any small extension $A' \arr A$, the induced functor $\F(A') \arr \F(A)$ is full. It can also be expressed as saying that automorphisms of objects of $\F$ are unobstructed.
\end{ex}

\subsubsection{Versal objects}

The condition of not having infinitesimal automorphisms prevents many deformation categories from being prorepresentable. However, there is a very useful substitute for universal deformations, which does exist quite often.

\begin{defin}\label{versaldef}
Let $\fib$ be a deformation category. A formal object $(R,\rho)$ of $\F$ is called \gr{versal} if the following lifting property holds: for every small extension $\phi\colon A'\to A$ in $(\Art_{\Lambda})$, every diagram of formal objects
$$
\xymatrix@+10pt@C+10pt{
& (A',\xi')\ar@/_10pt/@{-->}[dl]\\
(R,\rho)& (A,\xi) \ar[l] \ar[u]
}
$$
can be filled with a dotted arrow.
\end{defin}

It is easy to check that a formal object $(R,\rho)$ is universal if and only if for any diagram as above there exists a unique dotted arrow making it commutative. So universal deformations are versal.

\begin{prop}\label{liftingvers}
Let $\fib$ be a deformation category, and $(R,\rho)$ a versal formal object. Then the lifting property of the definition above also holds for all surjections $A'\to A$ in $\compl$.
\end{prop}

\begin{proof}
First of all, it is easy to see that the lifting property will hold when $A'\to A$ is a surjection in $\artl$, as usual by factoring $A'\to A$ into a composite of small extensions and lifting the morphism successively.

Let $A'\to A$ be a surjection in $\compl$, and we write $\xi=\{\xi_n,f_n\}$ and $\xi'=\{\xi_n',f'_n\}_\nin$. Let us show inductively that for each $n$ we can find a morphism of formal objects $g_n\colon (A'_n,\xi'_n)\to (R,\rho)$ such that for all $n$ the composite
$$
(A'_n,\xi'_n)\arr (A'_{n+1},\xi'_{n+1})\xrightarrow{g_{n+1}} (R,\rho)
$$
coincides with $g_n$, and the diagram
$$
\xymatrix@+10pt@C+10pt{
& (A'_n,\xi'_n)\ar@/_10pt/[dl]_{g_n}\\
(R,\rho)& (A_n,\xi_n) \ar[l] \ar[u]
}
$$
commutes.

Suppose we have constructed $g_n$, and consider the diagram
$$
\xymatrix{
R \ar@/^10pt/[rrd] \ar@/_10pt/[ddr]\ar@{-->}[rd] & &\\
& B \ar[r]\ar[d] & A'_n \ar[d] \\
& A_{n+1}\ar[r] & A_n
}
$$
where the maps from $R$ are the homomorphism $R\to A'_n$ coming from $g_n$ and the one $R\to A_{n+1}$ associated with $(A,\xi)\to (R,\rho)$, and $B$ is the fibered product. Taking the pullback of $\rho$ to $B$ along the dotted homomorphism above we get an object $\eta \in \F(B)$ restricting to $\xi'_n$ on $A_n'$ and on $\xi_{n+1}$ on $A_{n+1}$.

Now notice that there is a homomorphism $A_{n+1}'\to B$ induced by the quotient map $A_{n+1}'\to A_n'$ and the map $A_{n+1}'\to A_{n+1}$ coming from $A'\to A$, and which gives a morphism of formal objects $(B,\eta)\to (A_{n+1}',\xi_{n+1}')$, fitting in the commutative diagram
$$
\xymatrix@+5pt{
(A_{n+1}',\xi_{n+1}')  & &\\
& (B, \eta)\ar[lu] & (A'_n,\xi_n')\ar[l]\ar@/_10pt/[ull]  \\
& (A_{n+1},\xi_{n+1})\ar[u] \ar@/^10pt/[luu] &\ar[l]\ar[u] (A_n,\xi_n).
}
$$
Moreover from the fact that $A_{n+1}'\to B$ is a surjection in $\artl$ (as is readily checked, using the surjectivity of $A_{n+1}'\to A_{n+1}$), and from the diagram
$$
\xymatrix@+10pt@C+10pt{
& (A'_{n+1},\xi'_{n+1})\ar@/_10pt/@{-->}[dl]_{g_{n+1}}\\
(R,\rho)& (B,\eta) \ar[l] \ar[u]
}
$$
by versality of $(R,\rho)$ we get the dotted morphism $g_{n+1}\colon (A'_{n+1},\xi'_{n+1})\to (R,\rho)$ that has the desired properties.

Finally, notice that by Remark~\ref{morphism.approx} the sequence $\{g_n\}_\nin$ of compatible morphisms induces a morphism of formal objects $(A',\xi')\to (R,\rho)$ that gives the desired lifting.
\end{proof}

Notice that the dotted arrow in the diagram of Definition~\ref{versaldef} will give a lifting $R\to A'$ of the map $R\to A$, and conversely the existence of such a lifting implies at least that the deformation $\xi$ will lift to $A'$ (just by taking the pullback of $\rho$). In other words in presence of a versal deformations the problem of lifting objects becomes a problem of lifting maps of $\Lambda$-algebras. From this remark we will get a criterion to decide wether a deformation problem is obstructed or not, knowing a versal deformation (see Proposition~\ref{unobst-ver}).

As in the case of deformation functors, the property of being versal can be restated as a smoothness condition.

\begin{defin}\label{formally-smooth}
Let $\fib$ and $\G\to \artl^\op$ be two deformation categories, and $F\colon \F\to \G$ be a morphism. We say that $F$ is \gr{formally smooth} if for every surjection $A'\to A$ in $\artl$ the functor $\F(A')\to \F(A)\times_{\G(A)}\G(A')$ induced by the diagram
$$
\xymatrix{
\F(A')\ar[r]\ar[d]_{F_{A'}} & \F(A)\ar[d]^{F_A}\\
\G(A')\ar[r] & \G(A)
}
$$
is essentially surjective.
\end{defin}

\begin{rmk}
The composite of smooth morphisms of deformation categories is easily seen to be formally smooth.
\end{rmk}

\begin{rmk}
It follows easily by induction that to see that a morphism is formally smooth it is enough to check the condition when $A' \arr A$ is small.
\end{rmk}

The term ``smooth'' comes from the fact that if $\F$ and $\G$ are deformation categories corresponding to the functors of points of two schemes $X$ and $Y$, then a morphism $X\to Y$ locally of finite type is smooth if and only if the corresponding morphism $\F\to \G$ is formally smooth. This is the so-called ``infinitesimal smoothness criterion'' of Grothendieck (see \cite[Expos\'{e} III, Th\'{e}or\`{e}me 3.1]{SGA1}).

\begin{ex}\label{form.smooth}
Let $\fib$ be a deformation category. Show that a formal object $(R,\rho)$ of $\F$ is versal if and only if the corresponding morphism $\rho\colon (\Art_R)^\op\to \F$ is formally smooth.
\end{ex}

Here are two immediate properties of versal deformations.

\begin{prop}\label{kssurj}
Let $\fib$ be a deformation category, and $(R,\rho)$ a versal formal object of $\F$. Then:

\begin{enumeratei}

\item For every formal object $(S,\xi)$ restricting to $\rho_0$ on $k$ there is a morphism $(S,\xi)\to (R,\rho)$ (in particular this also holds if $S \in \artl$).

\item The Kodaira--Spencer map $\kappa_\rho\colon T_\Lambda R\to T_{\rho_0}\F$ is surjective.

\end{enumeratei}

\end{prop}

\begin{proof}
The first part of the statement is immediate from Proposition~\ref{liftingvers}, where we consider as surjection the quotient map $S\to k$, and the diagram
$$
\xymatrix@+10pt@C+10pt{
& (S,\xi)\ar@/_10pt/@{-->}[dl]\\
(R,\rho)& (k,\rho_0) \ar[l] \ar[u]
}
$$
that identifies $\rho_0$ as the pullback of $\rho$ and $\xi$ over $k$.

Now we prove (ii): take a vector $v \in T_{\rho_0}\F$, the usual ring of dual numbers $\dual$, and consider the element of $(\eps)\otimes_k T_{\rho_0}\F$ corresponding to $v$. We can then find an object $\xi \in \F_{\rho_0}(\dual)$ such that
$$
[\xi] - [\rho_0|_{\dual}]=v \in (\eps)\otimes_k T_{\rho_0}\F
$$
(simply by taking a representative of $[\rho_0|_{\dual}] + v$, where this is the usual action of Theorem~\ref{actionthm}) which is the same as saying that the Kodaira--Spencer map $\kappa_\xi\colon k\simeq (\eps)^\vee \to T_{\rho_0}\F$ of the formal object $(\dual,\xi)$ sends $1$ into $v$.

By the first part of the Proposition we get a morphism of formal objects
   \[
   (\dual,\xi)\to (R,\rho)
   \]
(and in particular a homomorphism $\phi\colon R\to \dual$) that identifies $\xi$ as a pullback of $\rho$, and from the second part of Proposition~\ref{kosfunct} we get
$$
v=\kappa_\xi(1)=\kappa_\rho(d\phi(1))
$$
where $d\phi\colon k\to T_\Lambda R$ is the differential induced by $\phi$. From this we see that $v$ is in the image of $\kappa_\rho$, and thus this map is surjective.

Alternatively, this follows immediately from Exercise~\ref{form.smooth}.
\end{proof}

In particular if $\F$ admits a versal object $(R,\rho)$, then the tangent space $T_{\rho_0}\F$ is finite-dimensional.

Now we state the anticipated criterion to recognize unobstructed objects.

\begin{prop}\label{unobst-ver}
Let $\fib$ be a deformation category, and $(R,\rho)$ a versal formal object of $\F$. Then $\rho_0$ is unobstructed if and only if $R$ is a power series ring over $\Lambda$.
\end{prop}

The proof, which we leave as an exercise, is a simple application of the smoothness criterion in Theorem~\ref{app.smoothcrit}: the algebra $R \in \compl$ is a power series ring over $\Lambda$ if and only if for any homomorphism $\phi\colon R\to A$ with $A \in \artl$, and small extension $\psi\colon A'\to A$, there exists a lifting $\lambda\colon R\to A'$, that is, a homomorphism $\lambda$ such that $\psi\circ\lambda=\phi$.

\begin{ex}
Prove Proposition~\ref{unobst-ver}.
\end{ex}

\subsubsection{Miniversal objects}

The second part of Proposition~\ref{kssurj} suggests us to consider versal deformations where $R$ is as ``small'' as possible, and leads us to the following definition.

\begin{defin}
A versal formal object $(R,\rho)$ of $\F$ is called \gr{minimal} if the Kodaira--Spencer map $\kappa_\rho\colon  T_\Lambda R\to T_{\rho_0}\F$ is an isomorphism.
\end{defin}

A versal minimal formal object is in short called \gr{miniversal}; Schlessinger calls the corresponding concept for deformation functors a \gr{hull}. Sometimes we will also say that $(R,\rho)$ is a \gr{miniversal deformation} of $\rho_0 \in \F(k)$.

Now we show that all universal deformations are miniversal, and that miniversal deformations are all isomorphic.

\begin{prop}\label{miniiso}
Let $\fib$ be a deformation category. Then:

\begin{enumeratei}

\item Any universal formal object of $\F$ is miniversal.

\item Any two miniversal formal objects of $\F$ with the same pullback to $k$ are isomorphic.

\end{enumeratei}

\begin{rmk}
The isomorphism whose existence is asserted in (ii) is, unfortunately in general non-canonical.
\end{rmk}

\end{prop}

\begin{proof}
We start by proving (i): it is clear that a universal object is in particular versal (and moreover the lifting morphism in the versality property will be unique), so we only have to prove that if $(R,\rho)$ is a universal formal object of $\F$, then the Kodaira--Spencer map $\kappa_\rho\colon T_\Lambda R\to \F$ is an isomorphism, and this follows from the fact that $\kappa_\rho$ is the differential of an equivalence of deformation categories.

For the second statement, take two miniversal objects $(R,\rho)$ and $(S,\nu)$ such that $\rho_0$ and $\nu_0$ are isomorphic. By Proposition~\ref{kssurj} we have two morphisms of formal objects $(R,\rho)\to (S,\nu)$ and $(S,\nu) \to (R,\rho)$, and we call $\phi\colon S\to R$ and $\psi\colon R\to S$ the corresponding homomorphisms.

By functoriality of the Kodaira--Spencer map and minimality of $(R,\rho)$ and $(S,\nu)$, the two differentials $d\phi\colon T_\Lambda R\to T_\Lambda S$ and $d\psi:T_\Lambda S\to T_\Lambda R$ will be isomorphisms (so the codifferentials are also), and from Corollary~\ref{complll} we get that $\psi$ and $\phi$ are isomorphisms. In conclusion $(R,\rho)$ and $(S,\nu)$ are isomorphic formal objects.
\end{proof}

In particular, the rings over which two miniversal deformations of a same object over $k$ are defined, are isomorphic.

The following exercise characterizes versal deformations in term of a miniversal one.

\begin{ex}
Let $\fib$ be a deformation category, $(R,\rho)$ a miniversal formal object of $\F$, and consider the power series algebra on $n$ indeterminates $S=R\ds{x_1,\hdots,x_n} \in \compl$, with the inclusion $i\colon R\to S$. Then show that the formal object $(S,i_*(\rho))$ obtained by pullback is versal.

Conversely, show that if $(P,\xi)$ is a versal formal object of $\F$ restricting to $\rho_0$ on $k$, and the kernel of $\kappa_\xi\colon T_\Lambda P\to T_{\xi_0}\F$ has dimension $n$, then $(P,\xi)$ is isomorphic to the formal object $(S,i_*(\rho))$ above.
\end{ex}

\subsection{Existence of miniversal deformations}\label{exist+schl}

Now we give an analogue of the ``existence of hulls'' part of Schlessinger's Theorem, in the context of deformations categories.

\begin{thm}\label{miniexist}
Let $\fib$ be a deformation category, and $\xi_0 \in \F(k)$ be such that the tangent space $T_{\xi_0}\F$ is finite-dimensional. Then $\F$ admits a miniversal formal object $(R,\rho)$, with $\rho_0\simeq\xi_0$.

Moreover if $n$ is the dimension of $T_{\xi_0}\F$, then $R$ will be a quotient $P/I$ of the power series ring $P=\Lambda\ds{x_1,\hdots,x_n}$ on $n$ indeterminates, with $I\subseteq \m_\Lambda P + \m_P^2$.
\end{thm}

\begin{proof}
First of all we notice that it is sufficient to find a formal object $(R,\rho)$ such that the Kodaira--Spencer map $\kappa_\rho\colon T_\Lambda R\to T_{\xi_0}\F$ is an isomorphism, and for every small extension $A'\to A$ with a diagram
$$
\xymatrix@+10pt@C+10pt{
& (A',\xi')\\
(R,\rho)& (A,\xi) \ar[l] \ar[u]
}
$$
of formal objects, the homomorphism $R\to A$ lifts to $R\to A'$.

To show this one uses the fact that the Kodaira--Spencer map $\kappa_\rho$ is an isomorphism and the functoriality property of the action of the tangent space on liftings (Proposition~\ref{actionfunct}) to ``adjust'' the lifting homomorphism $R\to A'$, so that the pullback of $\rho$ to $A'$ is isomorphic to $\xi'$ (the details are left to the reader).

Now we will construct a formal object $(R,\rho)$ with the weaker lifting property above.

Let $E=T^\vee_{\xi_0}\F$ be the dual of the tangent space $T_{\xi_0}\F$, and let $x_1,\hdots,x_n$ be a basis of $E$ as a $k$-vector space. Set $P \eqdef\Lambda\ds{x_1,\hdots,x_n}$. Then we have $\overline{P}_1\simeq k\oplus E$; let $\overline{\rho}_1 \in \F(\overline{P}_1)$ be the universal first order lifting of $\xi_0$ (see Example~\ref{universal-first-order}); this is versal with respect to artinian $\Lambda$-algebras of the form $k\oplus V$.
%

Now we will progressively extend $\overline{\rho}_1$ to a formal object on some bigger quotient of $P$. We first define inductively a sequence of ideals $I_i\subseteq P$ and objects $\rho_i \in \F(P/I_i)$ (it is easy to check that all the quotients will be actually artinian) starting with $I_1=\m_\Lambda P + \m_P^2$ and $\rho_1=\overline{\rho}_1$.

Suppose we have $I_{n-1}$ and $\rho_{n-1} \in \F(P/I_{n-1})$. Consider the set $A$ of ideals $I\subseteq P$ such that $\m_P I_{n-1}\subseteq I\subseteq I_{n-1}$ and there exists a lifting $\rho \in \F(P/I)$ of $\rho_{n-1}$, and take $I_n$ to be the minimum element of $A$ with respect to inclusion, that is, every element of the set $A$ contains $I_n$.

To show that such an element exists, we show that $A$ is closed under intersection (it is clearly nonempty, since $I_{n-1}$ satisfies the conditions). Noticing that ideals $\m_P I_{n-1}\subseteq I\subseteq I_{n-1}$ correspond to subspaces of the finite-dimensional $k$-vector space $I_{n-1}/\m_P I_{n-1}$, we only have to show that $A$ is closed under finite (or pairwise) intersection.

So suppose $I,J \in A$, with $\eta \in \F(P/I)$ and $\sigma \in \F(P/J)$; working in the $k$-vector space $I_{n-1}/\m_P I_{n-1}$ we can find an ideal $J'$ of $P$ such that $J\subseteq J'\subseteq I_{n-1}$, $I\cap J=I\cap J'$ and $I+J'=I_{n-1}$. Then we have that
$$
P/(I\cap J')\simeq P/I\times_{P/I_{n-1}}P/J'
$$
and using RS we get a deformation over $P/(I\cap J')=P/(I\cap J)$ lifting $\rho_{n-1}$, corresponding to the objects $\eta$ on $P/I$, and the pullback of $\sigma$ along the projection $P/J\to P/J'$, on $P/J'$. Thus $I\cap J$ is in $A$ as well.

Now set $I=\bigcap_i I_i$, and $R \eqdef P/I$. Notice that $R$ is still complete in the $\m_R$-adic topology, and we have also $R\simeq\invlim (P/I_i)\simeq\invlim(R/(I_i/I))$, as one easily checks. In particular the filtration $\{I_n/I\}_\nin$ of $R$ defines the same topology as its canonical one, and so (by Exercise~\ref{filtration}) we can define a formal object $\rho$ on $R$ as $\{\rho_n,f_n\}_\nin$, where $\rho_i \in \F(R/(I_i/I))$ are the ones defined above, and $f_n\colon \rho_n\to \rho_{n+1}$ are the arrows defining $\rho_{n+1}$ as a lifting of $\rho_n$.

Let us show that the formal object $(R,\rho)$ satisfies the two properties above: clearly the Kodaira--Spencer map $\kappa_\rho\colon T_\Lambda R\simeq E^\vee\to T_{\xi_0}\F=E^\vee$ is an isomorphism, since it is nothing else than $\kappa_{\overline{\rho}_1}$.

Now for the lifting property: suppose $A'\to A$ is a small extension, and that we have a diagram of formal objects as above. We want to show that $R\to A$ lifts to $R\to A'$. We can clearly assume that $A'\to A$ is a tiny extension, because if we prove it in this case, we can lift the homomorphism form $R$ successively, using the fact that every small extension is a composite of tiny extensions.

Notice that the homomorphism $R\to A$ factors through some $P/I_i$, say $P/I_k\to A$. Let us consider the fibered product $R'=(P/I_k)\times_A A'$, and take a lifting of the homomorphism $P\to P/I_k\to A$ to $P\to A'$ (which exists since $P$ is a power series ring, see Theorem~\ref{app.smoothcrit}). These homomorphisms together induce $P\to R'$, such that the following diagram is commutative.
$$\xymatrix{
 & R'\ar[r]\ar[d]&A'\ar[d]\\
P\ar@{-->}@/^/[ur]\ar[r]&P/I_k\ar[r] & A
}
$$
Call $J=\ker(P\to R')$, and notice that $J\subseteq I_k$. If $J=I_k$ we are done, because the projection $R'\to P/I_k$ will have a section that we can use to define our lifting as the composite $R\to P/I_k \to R' \to A'$.

So suppose that $J$ is properly contained in $I_k$. Identifying $P/J$ with its image in $R'$, we have that $I_k/J\subseteq \ker(R'\to P/I_k)$, which is isomorphic to $\ker(A'\to A)\simeq k$, so that necessarily $I_k/J=\ker(R'\to P/I_k)$. Looking at the diagram with exact rows
$$
\xymatrix{
0 \ar[r] & I_k/J  \ar[r] & P/J \ar[r] \ar[d] & P/I_k \ar[r] & 0\\
0 \ar[r] & \ker(R'\to P/I_k) \ar@{=}[u] \ar[r] & R' \ar[r] & P/I_k \ar[r] \ar@{=}[u] & 0
}
$$
we get that $R'\simeq P/J$. It is also easy to check that $\m_P I_k \subseteq J$ (and we had already that $J\subseteq I_k$), and by RS we can find a lifting $\overline{\rho} \in \F(R')=\F(P/J)$ of $\rho_k$. Since by definition $I_{k+1}$ is the minimal ideal of $P$ with these properties, we have that $I_{k+1}\subseteq J$, so that the homomorphism $P\to R'$ factors through $P/I_{k+1}$. Now it is clear that the composite $R\to P/I_{k+1} \to R'\to A'$ is a lifting of the given $R\to A$, so we are done.
\end{proof}

Now we turn to the proof of Schlessinger's Theorem~\ref{schlessinger}. The key point is the following.

\begin{prop}\label{mini-uni}
Let $\fib$ be a deformation category, and $(R,\rho)$ a miniversal formal object of $\F$. If $\Inf(\rho_0)=0$ and $\F(k)$ is a trivial groupoid, then $(R,\rho)$ is a universal formal object of $\F$.
\end{prop}

To prove this we need a Lemma.

Let $\phi\colon A'\to A$ be a small extension with kernel $I$, and $B \in \artl$. Suppose we have two homomorphisms $f$, $g\colon B\to A'$ such that the composites $h=\phi\circ f=\phi\circ g\colon B\to A$ coincide. Then the difference $f-g\colon B\to I$ is a $\Lambda$-derivation (see Proposition~\ref{ext.der}), so easy calculations, which use also the fact that $A'\to A$ is a small extension, show that $(f-g)(\m_B^2)=0$ and $(f-g)(\m_\Lambda B)=0$; we can consider then the induced $k$-linear function
$$
\Delta(f,g)\colon \m_B/(\m_\Lambda B+ \m_B^2)\simeq \m_{\overline{B}}/\m_{\overline{B}}^2 \longrightarrow I\,.
$$
Notice that by $\Lambda$-linearity of $f$ and $g$, and the fact that $B$ is generated by $\m_B$ and $\Lambda$ as a ring, we have $f=g$ if and only if $\Delta(f,g)=0$.

Now take an object $\xi \in \F(B)$, and consider the pullbacks $f_*(\xi), g_*(\xi) \in \F(A')$, which are liftings of $h_*(\xi)$. In particular by Theorem~\ref{actionthm} we have an action of $I\otimes_k T_{\xi_0}\F$ on $\lif(h_*(\xi),A')$, where $\xi_0$ is the pullback of $\xi$ to $k$, and recall also that the formal deformation $(B,\xi)$ has an associated Kodaira--Spencer class $k_\xi \in \m_{\overline{B}}/\m_{\overline{B}}^2\otimes_k T_{\xi_0}\F$.

\begin{lemma}
With the notation of Remark~\ref{action}, we have
$$
[f_*(\xi)] - [g_*(\xi)] = (\Delta(f,g)\otimes\id)(k_\xi) \in I\otimes_k T_{\xi_0}\F\,.
$$
\end{lemma}

\begin{proof}
Set $V \eqdef \m_{\overline{B}}/\m_{\overline{B}}^2$, and consider $B'=B\oplus V$ with the obvious $\Lambda$-algebra structure, and the trivial small extension
$$
\xymatrix{
0\ar[r] & V\ar[r] & B'\ar[r] & B\ar[r]& 0.
}
$$
If $\pi\colon B\to k$ and $\pi'\colon B\to \overline{B} = B/\m_\Lambda B$ are the quotient maps, there is a derivation $D\colon B\to V$ that sends $b \in B$ into the class of $\pi'(b)-\pi(b)$ in $\m_{\overline{B}}/\m_{\overline{B}}^2$.

We consider the two homomorphisms $i$, $u\colon B\to B'$, defined by $i(b)=(b,0)$ and $u(b)=(b,D(b))$, and the one $F\colon B'\to A'$ given by
$$
F(b,x)=g(b)+\Delta(f,g)(x)\,.
$$
One can easily check that $F\circ i= g$ and $F\circ u =f$, and using Proposition~\ref{funct} (with $\phi=F\colon B'\to A'$) we get
$$
[f_*(\xi)] - [g_*(\xi)] = (\Delta(f,g)\otimes \id)([u_*(\xi)] - [i_*(\xi)])
$$
(since $F|_{V}=\Delta(f,g)$).

Now consider $\overline{B}_1=\overline{B}/\m_{\overline{B}}^2\simeq k\oplus V$, and the homomorphism $h\colon B'\to \overline{B}_1$ defined by $h(b,x)=\pi(b)+x$. If we call $\pi''\colon B\to \overline{B}_1$ the quotient map, we have $h\circ u=\pi''$, and $(h\circ i)(a)=\pi(a)\in \overline{B}_1$.

From this we get that $h_*(u_*(\xi))\simeq \overline{\xi}_1$ and $h_*(i_*(\xi))\simeq \xi_0|_{\overline{B}_1}$; noticing that $h|_V$ is the identity, using Proposition~\ref{funct} again we infer that
$$
[u_*(\xi)] - [i_*(\xi)] = [\overline{\xi}_1] - [\xi_0|_{\overline{B}_1}]\,.
$$
But now $[\overline{\xi}_1] - [\xi_0|_{\overline{B}_1}] = k_\xi$ by definition, and this concludes the proof.
\end{proof}

Now we are ready to prove Proposition~\ref{mini-uni}.

\begin{proof}[Proof of Proposition~\ref{mini-uni}]
By Proposition~\ref{kssurj} we already know that for any formal object $(S,\xi)$ of $\F$ there exists a morphism $(S,\xi)\to (R,\rho)$ ($\xi_0$ will be necessarily isomorphic to $\rho_0$, for $\F(k)$ is a trivial groupoid).

As to uniqueness, we have to show that any two morphisms of formal objects $f$, $g\colon (S,\xi)\to (R,\rho)$ are the same. Using Proposition~\ref{equiv.rel} we see that, since $\Inf(\rho_0)$ is trivial, we only need to show that the two homomorphisms $\phi,\psi\colon R\to S$ associated with $f$ and $g$ are equal.

Moreover it is sufficient to show that $\phi_n,\psi_n\colon R_n\to S_n$ are equal for every $n$, and we do this inductively. Obviously $\phi_0=\psi_0$, so suppose $\phi_{n-1}=\psi_{n-1}$. In this case $\phi_n,\psi_n\colon R_n\to S_n$ are the same map when composed with $S_n\to S_{n-1}$, so we can consider
$$
\Delta(\phi_n,\psi_n)\colon \m_{\overline{R}_n}/\m_{\overline{R}_n}^2\simeq \m_{\overline{R}}/\m_{\overline{R}}^2 \longrightarrow \m_S^{n}/\m_S^{n+1}\,.
$$
Since by assumption $(\phi_n)_*(\rho_n)$ and $(\psi_n)_*(\rho_n)$ are isomorphic as liftings of the object $(\phi_{n-1})_*(\rho_{n-1})=(\psi_{n-1})_*(\rho_{n-1})$, by the preceding Lemma we conclude that
$$
(\Delta(\phi_n,\psi_n)\otimes \id)(k_\rho)=0 \in \m_S^{n}/\m_S^{n+1}\otimes_k T_{\rho_0}\F
$$
where $k_\rho \in \m_{\overline{R}}/\m_{\overline{R}}^2\otimes_k T_{\rho_0}\F$ is the Kodaira--Spencer class of $\rho$.

This means that if we compose the adjoint map
$$
\Delta(\phi_n,\psi_n)^\vee\colon (\m_S^{n}/\m_S^{n+1})^\vee\arr (\m_{\overline{R}}/\m_{\overline{R}}^2)^\vee
$$
with the Kodaira--Spencer map $\kappa_\rho\colon T_\Lambda R\to T_{\rho_0}\F$ of $\rho$ we get the zero map. But now $\kappa_\rho$ is an isomorphism, so we conclude that $\Delta(\phi_n,\psi_n)=0$, from which follows that $\phi_n=\psi_n$, as we want to show.
\end{proof}

Schlessinger's Theorem is now an easy corollary of Theorem~\ref{miniexist} and Proposition~\ref{mini-uni}.

The following Proposition gives a useful criterion that will be used later to show that some formal deformations are miniversal.

\begin{prop}\label{miniversalcrit}
Let $\fib$ be a deformation category, and suppose that $(R,\rho)$ is a formal object of $\F$ such that:
\begin{itemize}
 \item $R$ is a power series ring over $\Lambda$.
\item The Kodaira--Spencer map $\kappa_\rho\colon T_\Lambda R\to T_{\rho_0}\F$ is an isomorphism.
\end{itemize}
Then $(R,\rho)$ is a miniversal formal object, and in particular $\rho_0$ is unobstructed (see Proposition~\ref{unobst-ver}).
\end{prop}

\begin{proof}
Since $\kappa_\rho$ is an isomorphism, we have that $T_{\rho_0}\F$ is a finite-dimensional $k$-vector space; by Theorem~\ref{miniexist} we can then find a miniversal object $(S,\xi)$ restricting to $\rho_0$ over $k$. By Proposition~\ref{kssurj} and versality of $(S,\xi)$ we have a morphism of formal objects $(R,\rho)\to (S,\xi)$, and since both of the Kodaira--Spencer maps $\kappa_\rho$ and $\kappa_\xi$ are isomorphisms, the $k$-linear map $T_\Lambda^\vee S\to T_\Lambda^\vee R$ induced on the cotangent spaces is an isomorphism too (by Proposition~\ref{kosfunct}).

Since $R$ is a power series ring over $\Lambda$, this implies that the homomorphism $S\to R$ is an isomorphism (see Corollary~\ref{psr.iso}), and then the morphism $(R, \rho)\to (S,\xi)$ is an isomorphism too, so $(R,\rho)$ is miniversal.
\end{proof}

\subsubsection{Applications to obstruction theories}\label{rkproof}

Now that we have proved the existence of miniversal deformations, we can give proofs of the Ran-Kawamata Theorem (Theorem~\ref{rk}) and of the anticipated formula for the dimension of the minimal obstruction space associated with an obstruction theory.

\begin{proof}[Proof of Theorem~\ref{rk}]\label{prf-rk}
Let $(R,\rho)$ be a miniversal deformation of $\xi_0$ coming from Theorem~\ref{miniexist}. In particular $R$ is a quotient $P/I$, where $P=k\ds{x_1,\hdots,x_n}$ and $n = \dim_k T_{\xi_0}\F$, and $I\subseteq \m_P^2$. We want to show that $I=(0)$, so that $R$ is a power series ring, and by Proposition~\ref{unobst-ver} the object $\xi_0$ will be unobstructed.

The first step is to prove that the module of continuous differentials $\diff \eqdef \wh{\diff}_{R}$ (see Appendix~\ref{appb}) is a free $R$-module. Since $R$ is local we can equivalently show that $\diff$ is a projective $R$-module, and to do this it suffices to check that for every surjection $M'\to M$ of $R$-modules of finite length the induced homomorphism $\Hom_R(\diff,M')\to \Hom_R(\diff,M)$ is surjective.

Let us take then a surjection $M'\to M$ of $R$-modules of finite length, and $n$ large enough for $M'$ and $M$ to be $R/\m_R^{n+1}$-modules. Set $A \eqdef R_n$, and consider a homomorphism $\phi \in \Hom_R(\diff,M)$. This will correspond to a $k$-derivation $R\to M$, which in turn is the same as a homomorphism of $R$-modules $R\to A \oplus M$ (this is a standard fact) that is moreover compatible with the two quotient maps to $A$. In other words the diagram
$$
\xymatrix{
R\ar[r]\ar[rd] & A \oplus M\ar[d]\\
& A
}
$$
is commutative.

Take then the pullbacks $\xi \in \F_{\xi_0}(A)$ and $\xi' \in \F_{\xi_0}(A \oplus M)$ of the miniversal deformation $\rho$ along the two homomorphisms above. The class of $\xi'$ is an element $[\xi'] \in F_\xi(M)$, and so by right-exactness of $F_\xi$ we can find a $[\xi''] \in F_\xi(M')$ that maps to $[\xi']$ via the canonical function $F_\xi(M')\to F_\xi(M)$. In other words $\xi'' \in \F_{\xi_0}(A \oplus M')$ is an object whose pullback to $A \oplus M$ is isomorphic to $\xi'$.

By versality of $(R,\rho)$ the homomorphism $R\to A \oplus M$ can then be lifted to $R\to A \oplus M'$
$$
\xymatrix@+10pt@C+10pt{
& (A \oplus M',\xi'')\ar@/_10pt/@{-->}[dl]\\
(R,\rho)& (A \oplus M,\xi') \ar[l] \ar[u]
}
$$
and this lifting corresponds to a $k$-derivation $R\to M'$, which in turn is the same as a homomorphism $\psi \in \Hom_R(\diff,M')$. This homomorphism $\psi$ will then be in the preimage of the chosen $\phi \in \Hom_R(\diff,M)$, and this proves that $\Hom_R(\diff,M')\to \Hom_R(\diff,M)$ is surjective. In conclusion $\diff=\wh{\diff}_R$ is a free $R$-module.

Now we deduce that $I=(0)$, and this will conclude the proof, as we already noticed. We consider the conormal sequence
\begin{equation}\label{conorm.rk}
\xymatrix{
I/I^2\ar[r]^-d & \wh{\diff}_P\otimes_P R\ar[r] & \wh{\diff}_R\ar[r] & 0
}
\end{equation}
(see Proposition~\ref{cont.conorm}) and notice that, since $P$ is a power series ring on $n$ indeterminates, the $R$-module  $\wh{\diff}_P\otimes_P R$ is free of rank $n$. Moreover if $m$ is the rank of $\wh{\diff}_R$, tensoring (\ref{conorm.rk}) with $k$ we obtain an isomorphism
$$
(\wh{\diff}_P\otimes_P R)\otimes_R k \simeq \wh{\diff}_R\otimes_R k
$$
(for the homomorphism $d$ becomes the zero map), and this tells us that $m=n$.

Therefore the surjective homomorphism $\wh{\diff}_P\otimes_P R \to \wh{\diff}_R$ of (\ref{conorm.rk}) has to be an isomorphism, and so $d\colon I/I^2\to \wh{\diff}_P\otimes_P R$ is the zero map. This means that the image of $I$ along the universal derivation $d\colon P\to \wh{\diff}_P$ is contained in the ideal $I\wh{\diff}_P$, and this implies that for any $f \in I$ and $i=1,\hdots,n$, the partial derivative $\partial f/\partial x_i$ is an element of $I$.

Since $\car k=0$, it is easy to see that this implies $I=(0)$ (for example considering an element of $I$ of minimal degree and recalling that $I\subseteq \m_P^2$), and so we are done.
\end{proof}

Now consider a deformation category $\fib$, an object $\xi_0 \in \F(k)$ such that $T_{\xi_0}\F$ is finite-dimensional, and an obstruction theory $(V_\omega, \omega)$ for $\xi_0$. By Theorem~\ref{miniexist} $\xi_0$ has a miniversal deformation $(R,\rho)$ where $R$ is a quotient $P/I$, with $P=\Lambda\ds{x_1,\hdots,x_n}$, $n = \dim_k T_{\xi_0}\F$ and $I\subseteq \m_\Lambda P+\m_P^2$.

We denote by $\mu(I)$ the minimal number of generators of the ideal $I \subseteq P$, which by Nakayama's Lemma is the same as $\dim_k(I/\m_P I)$. Finally let $\Omega_\omega$ denote the minimal obstruction space associated with $(V_\omega,\omega)$, as in Section~\ref{obs}.

\begin{prop}
The dimension of\/ $\Omega_\omega$ as a $k$-vector space coincides with $\mu(I)=\dim_k(I/\m_P I)$.
\end{prop}

\begin{proof}
We will show that there exists an isomorphism of $k$-vector spaces $\Omega_\omega\simeq (I/\m_p I)^\vee$, and this will imply the result. Set $\widetilde{R} \eqdef P/\m_PI$ (this is an object of $\compl$ as well), so that we have an exact sequence of $P$-modules
$$
\xymatrix{
0\ar[r] & I/\m_P I\ar[r] & \widetilde{R}\ar[r] & R\ar[r] & 0.
}
$$
Tensoring this with $P_n=P/\m_P^{n+1}$, we obtain
$$
\xymatrix{
I/\m_P I\ar[r]^-{\alpha_n} & \widetilde{R}_n \ar[r] & R_n\ar[r] & 0
}
$$
and by the Artin-Rees Lemma we see that $\ker\alpha_n = (I/\m_PI)\cap \m_{\widetilde{R}}^{n+1}=(0)$ for $n$ large enough.

For every such $n$ then the sequence
$$
\xymatrix{
0\ar[r] &I/\m_P I\ar[r]^-{\alpha_n} & \widetilde{R}_n \ar[r] & R_n\ar[r] & 0
}
$$
is a small extension, and we have an object $\rho_n \in \F(R_n)$, coming from the versal deformation $(R,\rho)$. We can consider then the obstruction
$$
\omega_n=\omega(\rho_n,\widetilde{R}_n) \in I/\m_P I\otimes_k \Omega_\omega\simeq \Hom_k((I/\m_P I)^\vee, \Omega_\omega)\,.
$$
Notice that this element does not depend on $n$ (large enough): this follows immediately from functoriality of the obstruction, and the fact that for every $n$ large enough we have a commutative diagram with exact rows
$$
\xymatrix{
0\ar[r] &I/\m_P I\ar[r]\ar@{=}[d]& \widetilde{R}_{n+1}\ar[d] \ar[r] & R_{n+1}\ar[d]\ar[r] & 0\\
0\ar[r] &I/\m_P I\ar[r] & \widetilde{R}_n \ar[r] & R_n\ar[r] & 0.
}
$$
From this we get a well-defined element $f \in \Hom_k((I/\m_P I)^\vee, \Omega_\omega)$, that is, a $k$-linear map $f\colon (I/\m_P I)^\vee \to \Omega_\omega$. We show now that $f$ is bijective.

First we show that it is injective. Take a nonzero $u \in (I/\m_P I)^\vee$, which is a surjective $k$-linear function $u\colon I/\m_P I\to k$, and set $K \eqdef \ker u$, which is an ideal of $\widetilde{R}_n$ (for $n$ large enough). We consider then $(I/\m_PI)/K\simeq k$ and $S_n=\widetilde{R}_n/K$, and the following commutative diagram with exact rows
$$
\xymatrix{
0\ar[r] &I/\m_P I\ar[r]\ar[d]^u& \widetilde{R}_{n}\ar[d] \ar[r] & R_{n}\ar@{=}[d]\ar[r] & 0\\
0\ar[r] &k\ar[r] & S_n \ar[r] & R_n\ar[r] & 0.
}
$$
(where the vertical arrows are the quotient maps) that gives a morphism between the two small extensions.

By definition of the isomorphism $I/\m_PI \otimes_k \Omega_\omega\simeq \Hom_k((I/\m_PI)^\vee,\Omega_\omega)$ and by functoriality of the obstruction $\omega$, we have
   \begin{align*}
   f(u)&=\omega_n(u)\\
   &=(u\otimes \id)(\omega_n)\\
   &=\omega(\rho_n,S_n) \in  k\otimes_k \Omega_\omega \simeq \Omega_\omega.
   \end{align*}Suppose that $f(u)=0$. Then there is a lifting $\eta_n \in \F(S_n)$ of $\rho_n \in \F(R_n)$, and by versality of $(R,\rho)$ the homomorphism $R\to R_n$ will lift to $R\to S_n$. On the other hand since $\m_{S_n}^{n+1}=(0)$ this last map will factor through $R_n$, and give then a splitting $R_n\to S_n$ of the small extension above.

Finally notice that this splitting (as well as $S_n\to R_n$) will induce an isomorphism between cotangent spaces of $R_n$ and $S_n$, and then (by part (ii) of Corollary~\ref{complll}) the map $S_n\to R_n$ is an isomorphism. But this is a contradiction, because the kernel of this map is isomorphic to $k$. In conclusion this shows that $f(u)\neq 0$, and so $f$ is injective.

We show that it is surjective. Take a vector $v \in \Omega_\omega$, and suppose it corresponds to the obstruction $\omega(\xi,A')$ associated with a small extension $A'\to A$ with kernel $J$ and an isomorphism $g\colon J\simeq k$, and an object $\xi \in \F_{\xi_0}(A)$.

By versality of $(R,\rho)$ and Proposition~\ref{kssurj} we have an arrow of formal objects $(A,\xi)\to (R,\rho)$, and since $A$ is artinian the homomorphism $R\to A$ will factor through $R_n$ for $n$ large enough (and the pullback of $\rho_n$ to $A$ is isomorphic to $\xi$). Moreover if we lift the homomorphism $P\to R\to A$ to $\phi\colon P\to A'$ using the fact that $P$ is a power series ring over $\Lambda$, then $\phi(I)$ will be contained in $J$, and consequently $\phi(\m_P I)=(0)$, so $\phi$ will factor through $\widetilde{R}$.

Taking $n$ large enough we get a commutative diagram with exact rows
$$
\xymatrix{
0\ar[r] & I/\m_P I \ar[d]^u\ar[r] & \widetilde{R}_n\ar[r]\ar[d] & R_n\ar[r]\ar[d] & 0\\
0 \ar[r] & J\ar[r] & A'\ar[r] & A\ar[r] &0
}
$$
where $u\colon I/\m_P I\to J\simeq k$ can be seen as an element of $(I/\m_P I)^\vee$. By functoriality of the obstruction (and the other arguments used above) we get
   \begin{align*}
   f(u)&=\omega_n(u)\\
   &=(u\otimes \id)(\omega_n)\\
   &=\omega(\xi,A')\in J\otimes_k \Omega_\omega\simeq \Omega_\omega
   \end{align*}
which corresponds to $v$. This shows that $f$ is surjective, and concludes the proof.
\end{proof}

Using this we get immediately another proof of Proposition~\ref{unobst-ver}: $R$ is a power series ring if and only if $I=(0)$, and this happens exactly when $\dim_k \Omega_\omega = 0$, that is, when $\xi_0$ is unobstructed.

The last Proposition has the following corollaries.

\begin{cor}
Let $\fib$ be a deformation category, $\xi_0 \in \F(k)$, and $(V_\omega,\omega)$ be an obstruction theory for $\xi_0$. If $T_{\xi_0}\F$ is finite-dimensional, then $\Omega_\omega$ is as well.
\end{cor}

\begin{cor}
Let $\fib$ be a deformation category, $\xi_0 \in \F(k)$ such that $T_{\xi_0}\F$ is finite-dimensional, and $(V_\omega,\omega)$ be an obstruction theory for $\xi_0$. Moreover let $(R,\rho)$ be a miniversal deformation of $\xi_0$. Then
\begin{align*}
\dim R&\geq \dim_k T_{\xi_0}\F - \dim_k \Omega_\omega\\
&\geq \dim_k T_{\xi_0}\F - \dim_k V_\omega.
\end{align*}
\end{cor}

\begin{proof}
The last inequality is clear, so we prove only the first, by showing that $\dim R\geq \dim_k T_{\xi_0}\F - \mu(I)$ for the miniversal deformation $(R,\rho)$ considered above, and using the preceding Proposition. Set $n \eqdef \dim_k T_{\xi_0}\F$.

Notice that we can reduce to the case $\Lambda=k$ by considering the canonical homomorphism $P=\Lambda\ds{x_1,\hdots,x_n}\to k\ds{x_1,\hdots,x_n}$, and the induced surjection $R=P/I\to k\ds{x_1,\hdots,x_n}/J$ where $J$ is the extension of $I$. Indeed if we know that $\dim(k\ds{x_1,\hdots,x_n}/J)\geq n-\mu(J)$, then we have
\begin{align*}
\dim R&\geq \dim(k\ds{x_1,\hdots,x_n}/J)\\
&\geq n-\mu(J)\\
&\geq n-\mu(I).
\end{align*}
So we can assume that $\Lambda=k$. Then by Krull's Hauptidealsatz $\dim_k(I/\m_P I)=\mu(I)\geq \het I$; because of this inequality and the fact that $P=k\ds{x_1,\hdots,x_n}$ is catenary, we get
\begin{align*}
\dim R&=\dim(P/I)\\
&=\dim P-\het I\\
&\geq \dim P - \mu(I) = n - \mu(I).\qedhere
\end{align*}
\end{proof}

This result can be applied to find a lower bound on the dimension of the base ring $R$ of a miniversal deformation.

\begin{examp}
Let $Z_0\subseteq \P^3_k$ be a smooth curve of genus $g$ and degree $d$, and $(R,\rho)$ a universal deformation of $Z_0$ in $\hilb^{\P^3_k}$ (we have a miniversal one since $\dim_k T_{Z_0}\hilb^{\P^3_k}=\dim_k \H^0(Z_0,\mathcal{N}_0)$ is finite, and it is universal because $\hilb^{\P^3_k}$ is fibered in sets). Recall that $\hilb^{\P^3_k}$ comes from a representable functor, so if we denote by $\Hilb^{\P^3_k}$ the Hilbert scheme of $\P^3_k$, the dimension of $R$ in this case is the same as $\dim_{Z_0}\Hilb^{\P^3_k}$.

By the preceding corollary we get
$$
\begin{array}{rcl}
\dim_{Z_0}\Hilb^{\P^3_k} & \geq &  \dim_k T_{Z_0}\hilb^{\P^3_k}  - \dim_k V_\omega\\
& = & \dim_k \H^0(Z_0,\mathcal{N}_0) - \dim_k \H^1(Z_0,\mathcal{N}_0)\\
& = &\chi(\mathcal{N}_0)
\end{array}
$$
(here we are considering the obstruction theory described in Section~\ref{obs}) where $\chi$ is the Euler characteristic and $\mathcal{N}_0$ is the normal sheaf of $Z_0$ in $\P^3_k$.

Now from the dual of the conormal sequence of $Z_0\subseteq \P^3_k$
$$
\xymatrix{
0\ar[r] & T_{Z_0}\ar[r] & T_{\P^3_k}|_{Z_0}\ar[r] & \mathcal{N}_0\ar[r] & 0
}
$$
we get $\chi(\mathcal{N}_0)=\chi(T_{\P^3_k}|_{Z_0})-\chi(T_{Z_0})$, and from the restriction of the dual of the Euler sequence
$$
\xymatrix{
0\ar[r] & \O_{Z_0}\ar[r] & \O_{Z_0}(1)^{\oplus 4}\ar[r] & T_{\P^3_k}|_{Z_0} \ar[r] & 0
}
$$
we obtain further that $\chi(\mathcal{N}_0)=4\chi(\O_{Z_0}(1))-\chi(\O_{Z_0})-\chi(T_{Z_0})$. Using the Riemann-Roch Theorem to calculate explicitly the three terms in the last expression, we get
   \begin{align*}
   \dim_{Z_0}\Hilb^{\P^3_k} &\geq \chi(\mathcal{N}_0)\\
   &=(4d+4-4g)-(1-g)-(2-2g+1-g)\\
   &=4d\,;
   \end{align*}
this gives a lower bound on $\dim_{Z_0}\Hilb^{\P^3_k}$ independent of the genus $g$.
\end{examp}

\begin{examp}\label{curvdim}
Consider a smooth projective curve $X_0$ over $k$. Since the $k$-vector space $T_{X_0}\catdef\simeq \H^1(X_0,T_{X_0})$ is finite-dimensional, $X_0$ has a miniversal deformation $(R,\rho)$, and since $\H^2(X_0,T_{X_0})=0$ we see that $X_0$ is unobstructed (by Theorem~\ref{smoothobstr}), and so $R$ is a power series ring, and $\dim R=\dim_k T_{X_0}\catdef$.

We can calculate this dimension explicitly: if $g$ is the genus of $X_0$, then $T_{X_0}$ has degree $2-2g$, and by the Riemann-Roch Theorem we get
   \begin{align*}
   \chi(T_{X_0}) &= \dim_k \H^0(X_0,T_{X_0}) - \dim_k \H^1(X_0,T_{X_0})\\
   &= 2-2g+1-g = 3-3g.
   \end{align*}
Now if $g\geq 2$, then $T_{X_0}$ has negative degree, so $\dim_k \H^0(X_0,T_{X_0}) = 0$ and
$$
\dim R =\dim_k \H^1(X_0,T_{X_0})=3g-3\,.
$$
On the other hand if $g=1$ we find $\dim_k \H^1(X_0,T_{X_0}) = 1$, and in the case $g=0$ we have $\dim_k \H^1(X_0,T_{X_0}) = 0$.

These values give the minimal number of parameters that are necessary to describe a versal deformation of $X_0$ for a given genus, and by the discussion in example \ref{dimcoarse} (and the fact that $\mathcal{M}_{g}$ is a smooth Deligne-Mumford stack for $g\geq 2$), they coincide with the dimension of the coarse moduli spaces $M_{g}$ for $g\geq 2$.
\end{examp}

\subsubsection{Hypersurfaces in $\A^n_k$}\label{affineminiversal}

As an example (which will be useful in the next section), we calculate a miniversal deformation of a generically smooth hypersurface $X_0 \subseteq \A^n_k$, using the results of Exercise~\ref{affine.hypersurf}. In particular, since $T_{X_0}\catdef$ is finite-dimensional if and only if $X_0$ has isolated singularities, we have to restrict to this case.

Suppose then that $X_0\subseteq \A^n_k$ is a hypersurface as above, with equation $f \in k[x_1,\hdots,x_n]$, and so defined by the ideal $I=(f)$ and with coordinate ring $A=k[x_1,\hdots,x_n]/I$. Recall that
$$
T_{X_0}\catdef\simeq k[x_1,\hdots,x_n]/(f,\partial f/\partial x_1,\hdots,\partial f/\partial x_n)\,.
$$
Let $m \eqdef \dim_k T_{X_0}\catdef$ (which is finite because $X_0$ has isolated singularities), and choose elements $g_1,\hdots,g_m \in \Lambda[x_1,\hdots,x_n]$ such that their images in $T_{X_0}\catdef$ form a basis.

\begin{rmk}
The integer $m$ is an important invariant of the singularities of the hypersurface, known as the \emph{Tjurina number}.
\end{rmk}

We consider then the power series ring $R=\Lambda\ds{t_1,\hdots,t_m}$, and the closed subscheme
$$
X \eqdef V(f'+t_1g_1+\cdots+t_mg_m)\subseteq \A^n_R
$$
where $f' \in \Lambda[x_1,\hdots,x_n]$ is a lifting of $f$. The subscheme $X$ induces a formal deformation $\wh{X} \eqdef \{X_i,f_i\}_{i \in \N}$ of $X_0$ over $R$, by taking $X_i$ to be the pullback of $X$ to $R_i=R/\m_R^{i+1}$ along the quotient map $R\to R_i$, and as arrows $f_i\colon X_i\to X_{i+1}$ the natural closed immersions.

\begin{prop}
The formal object $(R,\wh{X})$ of the deformation category $\catdef\to \artl^\op$ is miniversal.
\end{prop}

\begin{proof}
We use the criterion given by Proposition~\ref{miniversalcrit}: $R$ is a power series ring, so we only have to check that the Kodaira--Spencer map
$$
\kappa_{\wh{X}}\colon T_\Lambda R\arr T_{X_0}\catdef
$$
is an isomorphism. Recall that this map is the same as $\kappa_{\overline{X}_1}$, where $\overline{X}_1$ is the pullback of $X$ to $\overline{R}_1\simeq k\oplus \m_{\overline{R}}/\m_{\overline{R}}^2$ along the projection $R\to \overline{R}_1$.

In this particular case we have
$$
\overline{X}_1 = \spec(\overline{R}_1[x_1,\hdots,x_n]/(f+t_1\overline{g}_1+\cdots +t_m\overline{g}_m))
$$
where $\overline{g}_i$ is the image of $g_i$ in $k[x_1,\hdots,x_n]$ (and we still write $t_i$ for the class of $t_i$ in $\m_{\overline{R}}/\m_{\overline{R}}^2$). Since the images of $t_1,\hdots,t_m$ in $T_\Lambda^\vee R=\m_{\overline{R}}/\m_{\overline{R}}^2$ form a basis of the cotangent space, we can consider the dual basis $s_1,\hdots,s_m \in T_\Lambda R$. The Kodaira--Spencer map
$$
\kappa_{\overline{X}_1}\colon T_\Lambda R \arr T_{X_0}\catdef \simeq k[x_1,\hdots,x_n]/(f,\partial f/\partial x_1,\hdots,\partial f/\partial x_n)
$$
sends $s_i$ into the class of $g_i$ (as the reader can easily check).

Then by the choice of the $g_i$'s the map $\kappa_{\overline{X}_1}$ is an isomorphism (since the two spaces have the same dimension, and a basis goes to a basis), and this concludes the proof.
\end{proof}

\begin{examp}\label{nodal.ex}
Consider the union of the two axes
$$
X_0=V(xy)\subseteq \A^2_k=\spec k[x,y]\,.
$$
In this case the Jacobian ideal is $J=(x,y)\subseteq k[x,y]/(xy)$, and a basis of $T_{X_0}\catdef \simeq k[x,y]/(x,y)$ is given by the class of $-1$. A miniversal deformation of $X_0$ is then for example the one induced by $X=V(xy-t)\subseteq \A^2_{\Lambda\ds{t}}$.
\end{examp}

\begin{examp}\label{cubic.ex}
Assume $\car k\neq 2,3$, and consider the cuspidal curve
$$
X_0=V(y^2-x^3)\subseteq \A^2_k=\spec k[x,y]\,.
$$
In this case we have $J=(2y,3x^2)\subseteq k[x,y]/(y^2-x^3)$, and a basis of $T_{X_0}\catdef \simeq k[x,y]/(y,x^2)$ is given by the classes of $1$ and $x$. The formal object induced by the closed subscheme $X=V(y^2-x^3+t_1+t_2x) \subseteq \A^2_{\Lambda\ds{t_1,t_2}}$ is then a miniversal deformation.
\end{examp}

\subsection{Algebraization}\label{algebraization}

The next step in constructing (or studying) deformations, is to pass from formal ones to ``actual'' ones (over noetherian complete local rings). In other words given a formal deformation, which is a sequence of compatible deformations over the artinian quotients of the base ring, we ask if there is a ``true'' deformation over the base ring that restricts to the given ones over these quotients.

Here we consider in particular the case of deformations of schemes.

\begin{defin} Let $X_{0}$ be a proper scheme over $k$.
A formal deformation $(R,\wh{X})$ of $X_0$ (that is, a formal object of the deformation category $\catdef_{X_0}\to \artl^\op$ over $R$) is said to be \gr{algebraizable} if there exists a scheme $X$ flat and proper over $R$ inducing the formal deformation $\wh{X}$ by pullback.
\end{defin}

In other words for each $n$ we have a morphism $\phi_n\colon X_n \to X$ with a cartesian square
$$
\xymatrix{
X_n \ar[r]^{\phi_n}\ar[d]& X\ar[d]\\
\spec R_n\ar[r]& \spec R
}
$$
and the $\phi_n$ are compatible with the morphisms $f_n\colon X_n\to X_{n+1}$ of the formal deformation $\wh{X}$.

We call an $X$ as above an \gr{algebraization} of the formal deformation $\wh{X}$. The idea is that $X$ is an actual deformation of $X_0$ over $R$, whose approximations to the various orders correspond to the terms of the formal deformation $\wh{X}$.

The problem of algebraization (that is, the problem of the existence of an algebraization for a given formal deformation) is not solvable in general. The main result when dealing with it in particular cases is the following Theorem, due to Grothendieck.

Let $\Lambda$ be as usual, and $X$ a scheme over $\Lambda$; set $X_n \eqdef X|_{\spec \Lambda_n}$. Together with the obvious morphisms, the sequence $\{X_n,f_n\}_\nin$ gives a formal deformation $\wh{X}$ of $X_0$ over $\Lambda$.

We denote by $\coh(X)$ the category of coherent sheaves on $X$, and by $\coh(\wh{X})$ the category of formal coherent sheaves on $\wh{X}$: its objects are collections $\{\E_n,g_n\}_\nin$ of coherent sheaves $\E_n$ on $X_n$, with isomorphisms $g_n\colon  \E_n \simeq \E_{n+1}|_{X_n}$ (where this pullback is along the immersion $f_n\colon X_n\to X_{n+1}$), and an arrow $\{\E_n,g_n\}_\nin \to \{\G_n,h_n\}_\nin$ is a sequence $\{F_n\}_\nin$ of homomorphisms $F_n\colon \E_n\to \G_n$ of coherent sheaves on $X_n$, compatible with the isomorphisms $g_n,h_n$. This is an abelian category, even though in a not completely trivial way.

There is a natural functor $\Phi\colon \coh(X)\to \coh(\wh{X})$, sending a coherent sheaf $\E$ on $X$ into the formal coherent sheaf $\{\E|_{X_n},f_n\}_\nin$, where $f_n$ are the obvious isomorphisms identifying the pullback of $\E|_{X_{n+1}}$ to $X_n$ with the one of $\E$, and a homomorphism $F\colon \E\to \G$ goes to the sequence $\{F_n\}_\nin$ of homomorphisms induced on the pullbacks.

\begin{thm}[Grothendieck's existence Theorem]\label{groth.ex}
If $X$ is proper over $\Lambda$, the functor $\Phi$ is an equivalence of abelian categories.
\end{thm}

For a discussion of this Theorem, see for example \cite[Chapter~8]{FGA}.

From this Theorem we get an algebraization result for certain deformations of schemes. First of all we have the following corollary about embedded formal deformations.

\begin{cor}\label{algcor}
Let $X$ be a proper scheme over $\Lambda$, and consider the formal deformation $\wh{X}=\{X_n,f_n\}_\nin$ of $X_0$ as above. Consider a sequence $\{Y_n\}_\nin$ of closed subschemes $Y_n\subseteq X_n$, such that for every $n$ we have $Y_{n+1}\cap X_{n}=Y_n$ (where we see $X_n\subseteq X_{n+1}$ by means of the closed immersion $f_n$). Then there exists a closed subscheme $Y\subseteq X$ such that $Y_n=Y\cap X_n$ for any $n$.
\end{cor}

\begin{proof}
We use Theorem~\ref{groth.ex}: consider the formal coherent sheaf $\{\O_{Y_n},f_n\}_\nin$, where $f_n$ are the obvious isomorphisms. By the Theorem we have a coherent sheaf $\E$ on $X$, and a sequence of isomorphisms $\phi_n\colon \E|_{X_n}=\E\otimes_{\O_X}\O_{X_n}\simeq \O_{Y_n}$ compatible with the projections $\O_{X_{n+1}}\to \O_{X_n}$ and $\O_{Y_{n+1}}\to \O_{Y_n}$.

Moreover we have an arrow $\{\O_{X_n},g_n\}_\nin \to \{\O_{Y_n},f_n\}_\nin$ of formal sheaves on $\wh{X}$, given by the surjections $\O_{X_n}\to \O_{Y_n}$ defining the closed subschemes $Y_n$. This arrow corresponds (by the Theorem again) to a homomorphism $\psi\colon \O_X\to \E$ of coherent sheaves on $X$, such that for every $n$ the diagram
$$
\xymatrix{
\O_X|_{X_n}\ar[r]^{\psi|_{X_n}}\ar@{=}[d] & \E|_{X_n}\ar[d]^{\phi_n}\\
\O_{X_n}\ar[r] & \O_{Y_n}
}
$$
is commutative.

Notice finally that since the functor $\Phi$ of Grothendieck's Theorem is an equivalence of abelian categories, and $\{\O_{X_n},g_n\}_\nin \to \{\O_{Y_n},f_n\}_\nin$ has trivial cokernel in $\coh(\wh{X})$, we get that $\psi$ is surjective. The kernel of $\psi\colon \O_X\to \E$ defines then a closed subscheme $Y\subseteq X$ with structure sheaf $\E$, such that $Y\cap X_n=Y_n$ for every $n$.
\end{proof}

\begin{ex}
Using Grothendieck's existence theorem, show that an algebraization of a formal deformation of $X_{0}$ is unique, up to a unique isomorphism.
\end{ex}

Now we go further, and consider abstract deformations.

\begin{prop}\label{algebr}
Let $X_0$ be a projective scheme over $k$ such that $\H^2(X_0,\O_{X_0})=0$, and suppose that $\wh{X} = \{X_n,f_n\}_\nin$ is a formal deformation of $X_0$ over $\Lambda$.

Then $\wh{X}$ is algebraizable.
\end{prop}

The algebraization that we will construct in the proof is moreover projective over $R$.

\begin{proof}
We start by showing that the natural restriction homomorphism $\pic(X_n)\to \pic(X_{n-1})$ is surjective. For a fixed $n$ we have an exact sequence of sheaves of groups
$$
\xymatrix{
0\ar[r] & \m_\Lambda^n/\m_\Lambda^{n+1}\otimes_k \O_{X_0}\ar[r] & \O_{X_n}^*\ar[r] & \O_{X_{n-1}}^*\ar[r] & 0
}
$$
where the first map is defined by $u\mapsto 1+u$ (over any open subset), and we see $\m_\Lambda^n/\m_\Lambda^{n+1}\otimes_k \O_{X_0}$ as the ideal sheaf of $X_{n-1}$ in $X_n$.

Taking the cohomology long exact sequence and recalling that by hypothesis
$$
\H^2(X_0,\m_\Lambda^n/\m_\Lambda^{n+1}\otimes_k \O_{X_0})\simeq \m_\Lambda^n/\m_\Lambda^{n+1}\otimes_k \H^2(X_0,\O_{X_0})
$$
is trivial, we see that the homomorphism $\H^1(X_0, \O_{X_n}^*)\to \H^1(X_0,\O_{X_{n-1}}^*)$ corresponding to the restriction $\pic(X_n)\to \pic(X_{n-1})$ is surjective.

Now take a very ample invertible sheaf $\L_0$ on $X_0$, such that $\H^1(X_0,\L_0)=0$, and let $s_0,\hdots,s_m$ be a basis of $\H^0(X_0,\L_0)$ as a $k$-vector space, defining the closed immersion $X_0\to \P^m_k$. By surjectivity of $\pic(X_n)\to \pic(X_{n-1})$ we can lift $\L_0$ successively to $X_n$, obtaining thus a sequence $\{\L_n\}_{n \in \N}$ of compatible invertible sheaves on the formal deformation $\wh{X}$; moreover we can also lift the basis above at each step.

In fact tensoring the exact sequence
$$
\xymatrix{
0\ar[r] & \m_\Lambda^n/\m_\Lambda^{n+1}\otimes_k \O_{X_0} \ar[r] & \O_{X_n}\ar[r] & \O_{X_{n-1}}\ar[r] & 0
}
$$
with $\L_n$, we get
$$
\xymatrix{
0\ar[r] & \m_\Lambda^n/\m_\Lambda^{n+1}\otimes_k \L_0 \ar[r] & \L_n\ar[r] & \L_{n-1}\ar[r] & 0.
}
$$
Noticing that
$$
\H^1(X_0, \m_\Lambda^n/\m_\Lambda^{n+1}\otimes_k \L_0)\simeq  \m_\Lambda^n/\m_\Lambda^{n+1}\otimes_k \H^1(X_0,\L_0)
$$
is trivial, and taking the cohomology long exact sequence of the last short one, we see that the restriction homomorphism $\H^0(X_0,\L_n)\to \H^0(X_0,\L_{n-1})$ is surjective, and so we can lift inductively $s_0,\hdots,s_m$ to elements $s_0^n,\hdots,s_m^n \in \H^0(X_0,\L_n)$.

Moreover the sections $(s_0^n,\hdots,s_m^n)$ will not have base points (because if they did, these points would also be base points of $(s_0,\hdots,s_m)$), and then for every $n$ we have an induced morphism $\phi_n\colon X_n \to \P^m_{\Lambda_n}$; since $\phi_0\colon X_0\to \P^m_k$ is a closed immersion, every $\phi_n$ will be as well.

This makes the sequence $\{X_n\}_\nin$ into a sequence of closed subschemes $X_n \subseteq \P^m_{\Lambda_n}$ compatible with the immersions $\P^m_{\Lambda_n}\subseteq \P^m_{\Lambda_{n+1}}$. Corollary~\ref{algcor} gives then a closed subscheme $X\subseteq \P^m_{\Lambda}$ restricting to $X_n$ over $\Lambda_n$. If we show that $X$ is flat over $\Lambda$, then it will be an algebraization of $\wh{X}$.

By \cite[Theorem~24.3]{Mat}, the locus of points at which $X$ is flat over $\Lambda$ is an open subset of $X$; consider its complement $Z$, a closed subset of $X$. Since $X \arr \spec \Lambda$ is proper, hence closed, if $Z$ is not empty its image will be a closed non-empty subset of $\spec \Lambda$, which must therefore contain the maximal ideal of $\Lambda$. So, to show that $Z = \emptyset$ it is enough to show that $X \arr \spec \Lambda$ is flat at any point of $X_{0}$.

But if $p \in Z\cap X_0$, since $\O_{X_n,p}\simeq \O_{X,p}/\m_{\Lambda}^{n+1}\O_{X,p}$ is flat over $\Lambda_n$ for every $n$, it follows from the local flatness criterion (Theorem~\ref{local-flatness}) that $X \to \spec\Lambda$ is flat at $p$, and this concludes the proof.
\end{proof}

\begin{examp}
Here is an example of a formal deformation of a smooth projective scheme over $\mathbb C$ that is not algebraizable. To do this, we will take as $X_0$ a smooth quartic surface in $\P^3_\mathbb{C}$, such that the Picard group $\pic(X_0)$ is cyclic, generated by the invertible sheaf $\O_{X_0}(1)$. One can check that in this case $\H^2(X_0,\O_{X_0})\simeq \mathbb{C}$, so that the hypotheses of the last Theorem are not satisfied.

To know that such a surface exists, we need the following facts.

\begin{thm}[Noether-Lefschetz]
Let $d\geq 4$, and $\P^N_\mathbb{C}$ be the projective space of surfaces of degree $d$ in $\P^3_\mathbb{C}$. Then there exists countably many hypersurfaces
$$
H_1,H_2,\hdots,H_n,\hdots \subseteq \P^N_\mathbb{C}
$$
such that if $X_0 \in \P^N_\mathbb{C}\setminus \bigcup_i H_i$, then $\pic(X_0)$ is cyclic and generated by $\O_{X_0}(1)$.
\end{thm}

For a discussion of this, see for example \cite{Griff}.

\begin{thm}[Baire]
In a locally compact and Hausdorff topological space, a countable intersection of open dense subsets is itself dense.
\end{thm}

Combining these two Theorems, we get a quartic surface $X_0\subseteq \P^3_\mathbb{C}$, such that $\pic(X_0)$ is cyclic generated by $\O_{X_0}(1)$.

\begin{ex}
Show that $\H^2(X_0,T_{X_0})=0$.
\end{ex}

In particular, using Theorem~\ref{smoothobstr}, this shows that $X_0$ is unobstructed.

From Proposition~\ref{hypers} we know that the differential of the forgetful morphism $F\colon \hilb^{\P^3_k}\to \catdef$ at $X_0$ is not surjective, so we can take a first-order deformation $X_\eps \to \spec\mathbb[\eps]$, such that there does not exists a closed immersion $X_\eps \subseteq \P^3_{\mathbb{C}[\eps]}$ extending $X_0\subseteq \P^3_\mathbb{C}$.

Moreover such a deformation has trivial Picard group: the exact sequence of shaves of groups
$$
\xymatrix{
0\ar[r] & \O_{X_0}\simeq \O_{X_0}\otimes_k (\eps)\ar[r] & \O_{X_\eps}^*\ar[r] & \O_{X_0}^*\ar[r] & 0
}
$$
yields
$$
\xymatrix{
0=\H^1(X_0,\O_{X_0})\ar[r] & \pic(X_\eps)\ar[r] & \pic(X_0)\ar[r] & \H^2(X_0,\O_{X_0})\simeq \mathbb{C}.
}
$$
Now since $\pic(X_0)$ is cyclic infinite and $\mathbb{C}$ is torsion-free, we conclude that the map $\pic(X_\eps)\to \pic(X_0)$ must be zero, and then $\pic(X_\eps)=0$.

From the fact that $X_0$ is unobstructed, we can find a formal deformation $\wh{X}=\{X_n,f_n\}_\nin$ of $X_0$ over $\mathbb{C}\ds{t}$, with term of order one isomorphic to $X_\eps$.

\begin{prop}
The formal deformation $\wh{X}$ is not algebraizable, that is, there does not exist a flat and proper scheme $X$ over $\mathbb{C}\ds{t}$ inducing the formal deformation $\wh{X}$.
\end{prop}

\begin{proof}
Assume such an $X$ exists, and take an open affine subscheme $U\subseteq X$. If we denote by $D=X\setminus U$ the complement of $U$, then every irreducible component of $D$ has codimension $1$ (see for example \cite[Corollaire 21.12.7]{EGAIV}), and then, with the structure of reduced closed subscheme, it can be seen as an effective divisor on $X$.

Now notice that $X$ is smooth over $\mathbb{C}\ds{t}$, so in particular Cartier divisors correspond to invertible sheaves: if $Z$ is the locus where $X$ is not smooth, then $Z$ is a closed subset of $X$, and (if it is not empty) it must intersect the central fiber $X_0$ (since $X\to \spec \mathbb{C}\ds{t}$ is proper, and then the image of $Z$ contains the maximal ideal of $\mathbb{C}\ds{t}$). But $X_0$ is smooth over $\mathbb{C}$, and this contradiction shows that $Z=\emptyset$.

Now, since $X_0\nsubseteq U$, we have that $D\cap X_0$ is an effective divisor on $X_0$, and it is not trivial (i.e. the associated invertible sheaf $\O_{X_0}(D\cap X_0)$ is not isomorphic to $\O_{X_0}$). Finally consider the commutative diagram
$$
\xymatrix@C-20pt{
\pic(X)\ar[rd]\ar[rr]&&\pic(X_\eps)\ar[ld]\\
&\pic(X_0)&
}
$$
with the maps are the natural pullback homomorphisms. We showed above that the function $\pic(X)\to \pic(X_0)$ is not zero, since the invertible sheaf $\O_X(D)$ goes to $\O_{X_0}(D\cap X_0)$, which is not trivial, but on the other hand we proved that $\pic(X_\eps)=0$, which gives a contradiction.
\end{proof}

\end{examp}

\section{Deformations of nodal curves}\label{capcurv}

As a simple example of application of the general theory described up to this point, we consider in this last section deformations of affine and projective curves with a finite number of nodes.

By studying this particular case we will show how knowing a miniversal deformation of a local model for a singularity helps in giving a local (formal) description of any global deformation of such a singularity. Finally we will give an algebraization result for projective curves with a finite number of nodes that relies on the general results of the preceding section.

Let us start with some generalities on the behavior of deformations under \'etale morphisms.

\subsection{Pullbacks of deformations under \'etale morphisms}

Let $f\colon Y_{0} \arr X_{0}$ be an \'etale morphism of $k$-schemes. Let $X_{A}$ be a lifting of $X_{0}$ to some $A \in (\Art_\Lambda)$. The pullback functor from the category of \'etale morphism to $X_{A}$ to that of \'etale morphisms to $X_{0}$ is an equivalence (see for example \cite[Theorem~3.23]{Mil}); hence there is an \'etale morphism $f^{*}X_{A} \arr X_{A}$, unique up to a unique isomorphism, whose restriction to $X_{0}$ is isomorphic to $Y_{0} \arr X_{0}$. We obtain a morphism of fibered categories $f^{*}\colon \catdef_{X_{0}} \arr \catdef_{Y_{0}}$ by sending each lifting $X_{A}$ into $f^{*}X_{A}$.

\begin{prop}\label{etale-pullback}
Let us assume that the following conditions are satisfied.

\begin{enumeratea}

\item The scheme $X_{0}$ is affine, of finite type, local complete intersection, and smooth outside of a finite number of closed points $p_{1}$, \dots,~$p_{r}$.

\item The scheme $Y_{0}$ is also affine. Furthermore, for each $i$ the inverse image $f^{-1}(p_{i})$ is formed of a unique point $q_{i} \in Y_{0}$ with $k(p_{i}) = k(q_{i})$.

\end{enumeratea}

Then the following hold.

\begin{enumeratei}

\item The differential $d_{X_0}f^{*}\colon T_{X_{0}}\catdef_{X_{0}} \arr T_{Y_{0}}\catdef_{Y_{0}}$ is an isomorphism.

\item The morphism $f^{*}$ is formally smooth (Definition~\ref{formally-smooth}).

\item If $(R,\rho)$ is a miniversal deformation of $X_0$, then $f^*(R,\rho)$ (with the obvious meaning) is a miniversal deformation of $Y_0$.

\end{enumeratei}
\end{prop}

\begin{proof}
Set $X_{0} = \spec R$ and $Y_{0} = \spec S$. We have $T_{X_{0}}\catdef_{X_{0}} = \Ext^{1}_{R}(\Omega_{R/k}, R)$ and $T_{Y_{0}}\catdef_{X_{0}} = \Ext^{1}_{S}(\Omega_{S/k}, S)$. Furthermore, there is an obstruction theory for $X_{0}$ with space $\Ext^{2}_{R}(\Omega_{R/k}, S)$, and analogously for $Y_{0}$.

Since $S$ is flat over $R$, there are natural isomorphism
   \[
   \Ext^{i}_{R}(\Omega_{R/k}, R)\otimes_{R}S \simeq 
   \Ext^{i}_{S}(\Omega_{S/k}, S)
   \]
for each $i$. When $i > 0$, the $R$-module $\Ext^{i}_{R}(\Omega_{R/k}, R)$ is supported on $\{p_{1}, \dots, p_{r}\}$. It is easily seen by induction on the length that for every $R$-module with support on $\{p_{1}, \dots, p_{r}\}$, the homomorphism $M \arr M\otimes_{R}S$ is an isomorphism; hence we obtain canonical isomorphisms
   \begin{align*}
   \Ext^{i}_{R}(\Omega_{R/k}, R) & \simeq
   \Ext^{i}_{R}(\Omega_{R/k}, R)\otimes_{R}S\\
   &\simeq \Ext^{i}_{S}(\Omega_{S/k}, S)
   \end{align*}
By Theorem~\ref{kodaira} we have
   \[
   T_{X_{0}}\catdef_{X_{0}} = \Ext^{1}_{R}(\Omega_{R/k}, R)\text{\quad and \quad} T_{Y_{0}}\catdef_{Y_{0}} = \Ext^{1}_{S}(\Omega_{S/k}, S)\,;
   \]
it is easy to see that the canonical isomorphism $\Ext^{1}_{R}(\Omega_{R/k}, R) \simeq \Ext^{1}_{S}(\Omega_{S/k}, S)$ corresponds to the differential $d_{X_{0}}f^{*}$. This proves (i).

Let us show that $f^{*}$ is formally smooth. This amounts to the following: if we are given a small extension $A' \arr A$ of $\Lambda$-algebras with kernel $I$, a lifting $X_{A}$ of $X_{0}$ to $A$, and a lifting $Y_{A'}$ of $f^{*}X_{A}$ to $A'$, then there is a lifting $X_{A'}$ to $A'$, such that $f^{*}X_{A'}$ is isomorphic to $Y_{A}$ as a lifting of $f^*{X_{A}}$. By Theorem~\ref{schemesobstr}, the obstructions to lifting $X_{A}$ and $f^{*}X_{A}$ correspond under the isomorphism $I\otimes_{k}\Ext^{2}_{R}(\Omega_{R/k}, R) \simeq I\otimes_{k}\Ext^{2}_{S}(\Omega_{S/k}, S)$; since $f^{*}X_{A}$ has a lifting, by hypothesis, we see that $X_{A}$ must also have a lifting $X_{A'}$. 

The lifting $f^{*}X_{A'}$ is not necessarily isomorphic to $Y_{A}$ as a lifting of $f^{*}X_{A}$; however, the difference is measured by an element of $I\otimes_{k}T_{Y_{0}}\catdef_{Y_{0}}$, which, by (i), comes from an element of $I\otimes_{k}T_{X_{0}}\catdef_{X_{0}}$; hence we can modify $X_{A'}$ so that $f^*X_{A'}$ is isomorphic to $Y_{A'}$ as a lifting of $f^{*}X_{A}$. This completes the proof of (ii).

For part (iii), let $(R,\rho)$ be a miniversal deformation of $X_{0}$. By Exercise~\ref{form.smooth}, this is equivalent to saying that the corresponding morphism $\rho\colon (\Art_R)^{\op} \arr \catdef_{X_{0}}$ is formally smooth, with bijective differential. By composing with $f^{*}$ we obtain a morphism $(\Art_R)^{\op} \arr \catdef_{Y_{0}}$ which is once again formally smooth with bijective differential, because of parts (i) and (ii), so it is a miniversal deformation.
\end{proof}

\subsection{Nodal curves}

Let $X$ be a curve over $k$.

\begin{defin}
A closed point $p \in X$ is a \gr{rational node} if $p$ is a rational point, and the complete local ring $\wh{\O}_{X,p}$ is isomorphic to $k\ds{x,y}/(xy)$ as a $k$-algebra.
\end{defin}

We consider generically smooth curves, having only rational nodes as singularities.

\begin{defin}
By \gr{nodal curve} we mean a curve $X$ over $k$ that is smooth outside of a finite number of closed points $p_1,\hdots,p_n$ that are rational nodes.
\end{defin}

We give a criterion to recognize rational nodes, assuming $\car k\neq 2$: suppose $X$ is a curve over $k$, and that the complete local ring of $X$ in $p \in X(k)$ has a 2-dimensional tangent space. Then $\widehat{\O}_{X,p}$ is isomorphic to $k\ds{x,y}/(f)$ for some element $f \in k\ds{x,y}$. Write $f_i$ for the homogeneous term of degree $i$ of $f$: then suppose $f_0=f_1=0$. Assume that that $f_2$ is a quadratic form equivalent to $xy$ over $k$. Then there is an isomorphism $k\ds{x,y}/(f)\simeq k\ds{x,y}/(xy)$. In other words, every rational singular point with multiplicity two and with two rational distinct tangent lines is a rational node.

%
%
%
%

\subsection{Deformations of affine nodal curves}

The first case we consider is the one of an affine nodal curve $X_0$ over $k$, with a single rational node $p$. Since the complete local ring should give some control on the local structure of a scheme at the corresponding point, and by definition we have an isomorphism of the complete local ring $\wh{\O}_{X_0,p}$ with $k\ds{x,y}/(xy)$, which is the complete local ring at the origin of $Y_0 \eqdef V(xy)\subseteq \A^2_k$, one could hope to link the deformations of $X_0$ and the ones of $V(xy)$ using this isomorphism, which is what we will do now.

The starting point is the following Theorem of Michael Artin (see \cite{Art}).

\begin{thm}\label{artinthm}
Suppose $X,Y$ are schemes of finite type over a base scheme $S$, also of finite type over a field $k$. Let $p\in X$ and $q \in Y$ be two points with a fixed isomorphism $f\colon k(p)\simeq k(q)$ over $S$, and call $s$ the image of $p$ and $q$ in $S$. Then $f$ extends to an isomorphism $\wh{\O}_{X,p}\simeq \wh{\O}_{Y,q}$ of $\wh{\O}_{S,s}$-algebras if and only if there exists a scheme $Z$ over $S$ with two \'{e}tale morphisms $Z\to X$ and $Z\to Y$, fitting in the commutative diagram below.
$$
\xymatrix@-5pt@C-10pt@R-8pt{
\spec k(p)\ar[dr]\ar@/_20pt/[dd] & & \spec k(q)\ar[ll]_\sim\ar[ld] \ar@/^20pt/[dd]\\
& Z\ar[rd]\ar[ld] &\\
X\ar[dr] & & Y \ar[ld]\\
& S&
}
$$
\end{thm}

We apply the preceding Theorem with $S=\spec k$, $X=X_0$ our curve, $Y=Y_0=V(xy)\subseteq \A^2_k$, $p$ the rational node of $X_0$, and $q$ the origin in $Y_0$.

Since we have an isomorphism $\wh{\O}_{X_0,p}\simeq k\ds{x,y}/(xy)=\wh{\O}_{Y_0,q}$ extending the identity $k\simeq k$ on the residue fields, we conclude that there exist a scheme $Z$ over $k$ with two \'{e}tale morphisms $f\colon Z\to X_0, g\colon Z\to Y_0$, and a rational point $z$ of $Z$ that gets mapped to $p$ and $q$ respectively, and fitting in a commutative diagram as above.

We will use these two \'{e}tale maps to link the deformations of $X_0$ with the ones of the ``standard'' nodal curve $Y_0$. Precisely, applying the functor $g^*$ to the standard miniversal deformation $\wh{Y}=\{Y_n,f_n\}_\nin$ of $Y_0$ (see Example~\ref{nodal.ex}) we get a miniversal deformation $\wh{Z}=\{Z_n,g_n\}_\nin$ of $Z$: here $Z_n$ can be defined inductively as a scheme over $k$ with an \'{e}tale morphism $Z_n\to Y_n$ that fits in the cartesian diagram
$$
\xymatrix{
Z_{n-1}\ar[r]^{g_{n-1}}\ar[d]& Z_n \ar[d]\\
Y_{n-1}\ar[r] & Y_n.
}
$$

Since $X_0$ has isolated singularities we can also consider a miniversal deformation $\wh{X}=\{X_n,h_n\}$ of $X_0$, say over $S \in \compl$, and apply the functor $f^*$. This way we get another miniversal deformation $\wh{Z}'=\{Z_n',g_n'\}_\nin$ of $Z$, defined the same way as the one induced by $\wh{Y}$.

Since two miniversal deformations of the same scheme over $k$ are isomorphic (Proposition~\ref{miniiso}), we have an isomorphism $(S,\wh{Z}')\to (R,\wh{Z})$, which consists of an isomorphism of $\Lambda$-algebras $\phi\colon R\to S$, together with isomorphisms $\alpha_n\colon Z_n'\to Z_n$ over $\phi_n$, and compatible with the immersions $g_n$ and $g_n'$.

Now we can consider the inverse $\psi\colon S\to R$ of $\phi$, and the pullback $X_n'$ of $X_n$ along $S_n\to R_n$. This, together with the induced arrows $h_n'\colon X_n'\to X_{n+1}'$, gives another formal deformation $\wh{X}'=\{X_n',h_n'\}_\nin$ of $X_0$ over $R$ that is easily seen to be miniversal too. Moreover the morphisms $Z_n'\to X_n$ and $Z_n'\to \spec S_n\to \spec R_n$ induce an \'{e}tale morphism $Z_n'\to X_n'$ over $\spec R_n$.

Putting everything together, for every $n$ we have a commutative diagram
$$
\xymatrix@C-15pt@-5pt{
& Z_n'\ar[rr]_{\alpha_n}^\sim\ar[ld]\ar[rdd] & & Z_n\ar[rd]\ar[ldd] & \\
X_n'\ar[rrd] & & & &Y_n\ar[lld]\\
&&\spec R_n&&
}
$$
where $Z_n'\to X_n'$ and $Z_n\to Y_n$ are \'{e}tale, and moreover the morphisms in this diagram are compatible with the closed immersions $h_n',g_n,g_n',f_n$ and $\spec R_n\to \spec R_{n+1}$.

This gives us (by Theorem~\ref{artinthm}) a sequence of compatible isomorphisms
$$
\lambda_n\colon \wh{\O}_{X_n',p}\to \wh{\O}_{Y_n,q}\simeq R_n\ds{x,y}/(xy-t)
$$
in the sense that for every $n$ the diagram
$$
\xymatrix{
\wh{\O}_{X_{n+1}',p}\ar[r]^{\lambda_{n+1}}\ar[d] & \wh{\O}_{Y_{n+1},q}\ar[d]\\
\wh{\O}_{X_n',p}\ar[r]^{\lambda_n} & \wh{\O}_{Y_n,q}
}
$$
commutes, where the vertical maps are the projections.

From this discussion we get the following result, which gives a description of the complete local ring of a global deformation of a curve around a rational node, which is a ``formal'' description of the deformation around the node.

\begin{prop}\label{nodaldef}
Suppose that $f\colon X\to S$ is a flat morphism of finite type, the fiber $X_0=f^{-1}(s_0)$ over a point $s_0 \in S$ is a curve over $k(s_0)$ with isolated singularities, and $p \in X_0$ is a rational node. Then there exists $u \in \wh{\O}_{S,s_0}$ and an isomorphism of $\wh{\O}_{S,s_0}$-algebras
$$
\wh{\O}_{X,p}\simeq \wh{\O}_{S,s_0}\ds{x,y}/(xy-u)\,.
$$
\end{prop}

\begin{proof}
Since the statement is local, we can assume that $p$ is the unique singular point of $X_0$. We take $\Lambda=\wh{\O}_{S,s_0}$, and consider the formal deformation $\wh{X}=\{X_n,f_n\}_\nin$ over $\Lambda$ defined by $X_n=f^{-1}(\spec\Lambda_n)$, where we see $\spec\Lambda_n \to \spec\Lambda\to S$ as the $n$-th infinitesimal neighborhood of $s_0$; the morphisms $f_n$ are the induced closed immersions.

Since $(\Lambda,\wh{X})$ is a formal deformation, we have a morphism of formal objects $(\Lambda,\wh{X})\to (\Lambda\ds{t},\wh{X}')$ (where $\wh{X}'$ is the miniversal deformation of $X_0$ constructed above), that is, a homomorphism of $\Lambda$-algebras $\Lambda\ds{t}\to \Lambda$ and closed immersions $X_n\to X_n'$ such that for every $n$ the diagram
$$
\xymatrix{
X_n\ar[r]\ar[d]& X_n'\ar[d]\\
\spec \Lambda\ar[r] & \spec \Lambda\ds{t}
}
$$
is cartesian.

We call $u$ the image of $t$ along the homomorphism $\Lambda\ds{t}\to \Lambda$, which we can see as the quotient map $\Lambda\ds{t}\to \Lambda\ds{t}/(t-u)\simeq \Lambda$. Using the preceding discussion and pulling back to $\spec \Lambda$ we get a sequence of compatible isomorphisms of $\Lambda$-algebras
$$
\wh{\O}_{X_n,p}\simeq \Lambda\ds{t}_n\ds{x,y}/(xy-t,t-u)\,.
$$
Finally, passing to the projective limit, this sequence induces an isomorphism
$$
\wh{\O}_{X,p}\simeq \Lambda\ds{t}\ds{x,y}/(xy-t,t-u)\simeq \Lambda\ds{x,y}/(xy-u)
$$
which is what we want.
\end{proof}

Proposition~\ref{nodaldef} can be generalized to the case of local complete intersections with isolated singularities.

\begin{examp}\label{cusp}
If instead of $xy=0$ we take $y^2-x^3=0$ as standard singularity (and we assume that $\car k\neq 2,3$), then by Example~\ref{cubic.ex} we know that a miniversal deformation of $V(y^2-x^3)\subseteq \A^2_k$ is given by the pullbacks $X_n=X|_{R_n}$ of $X=V(y^2-x^3+t+ux)\subseteq \A^2_R$ to the quotients of $R=\Lambda\ds{t,u}$, together with the induced immersions $X_n\to X_{n+1}$.

Then in the same way one can prove: if $f\colon X\to S$ is a flat morphism of finite type such that $X_0=f^{-1}(s_0)$ (with $\car(k(s_0))\neq 2,3$) is a curve with isolated singularities over $k(s_0)$, and $p \in X_0$ is a rational point such that $\wh{\O}_{X_0,p}\simeq k(s_0)\ds{x,y}/(y^2-x^3)$, then there exist $v,w \in \wh{\O}_{S,s_0}$ and an isomorphism of $\wh{\O}_{S,s_0}$-algebras
$$
\wh{\O}_{X,p}\simeq\wh{\O}_{S,s_0}\ds{x,y}/(y^2-x^3+v+wx)\,.
$$
\end{examp}

Now we analyze briefly the more general case of affine curves with more than one node. In this case one can repeat the argument we described for the curve with one node, using instead of $Y_0=V(xy)\subseteq \A^2_k$ a disjoint union of copies of $Y_0$, one for each node of $X_0$. We can get a miniversal deformation of this disjoint union from the fact that its deformation category will be a product of some copies of the deformation category of $Y_0$, and once we have a miniversal deformation we can repeat the argument above.

Here is the final results (analogous to the ones for curves with one node) that one can get by the analysis above.

\begin{prop}\label{miniversnod}
Let $X_0$ be an affine nodal curve over $k$, with $r$ rational nodes $p_{1}$, \dots,~$p_{r}$ and smooth elsewhere. Then there is a miniversal deformation $\wh{X}=\{X_n,f_n\}_\nin$ of $X_0$ over $\Lambda\ds{t_1,\hdots,t_r}$, with compatible isomorphisms of $\Lambda$-algebras
$$
\wh{\O}_{X_n,p_i}\simeq \Lambda\ds{t_1,\hdots,t_r}_n \ds{x,y}/(xy-t_i)
$$
for every $n$ and $i$.
\end{prop}

\begin{prop}\label{nodaldef1}
Let $f\colon X\to S$ be a flat morphism of finite type, and suppose that the fiber $X_0=f^{-1}(s_0)$ over a point $s_0 \in S$ is a nodal curve over $k(s_0)$ with $r$ nodes $p_1,\hdots,p_r \in X_0$. Then there exist $u_1,\hdots,u_r \in \Lambda$ and isomorphisms of $\O_{S,s_0}$-algebras
$$
\wh{\O}_{X,p_i}\simeq \wh{\O}_{S,s_0}\ds{xy}/(xy-u_i)\,.
$$
\end{prop}

\subsection{Deformations of projective nodal curves}

Now we turn to the case of projective nodal curves. Let $X_0$ be a projective nodal curve over $k$, and call $p_1,\hdots,p_r \in X_0$ its rational nodes. In this case one shows, using the local-to-global spectral sequence for Ext groups, that $\Ext^2_{\O_{X_0}}(\diff_{X_0},\O_{X_0})=0$, so (by Theorem~\ref{schemesobstr}) $X_0$ is unobstructed.

Now take an open affine subscheme $U_0\subseteq X_0$, containing all the nodes $p_1,\hdots,p_r$. Since the open immersion $i\colon U_0\to X_0$ is \'{e}tale, we have an induced restriction functor $i^*\colon \catdef_{X_0}\to \catdef_{U_0}$, which in this case is a ``true'' restriction (in the sense that $i^*(X)$ is the open subscheme of $X$ with underlying topological space $U_0$). Then one shows that the differential
$d_{X_0}i^*\colon T_{X_0}\catdef\to T_{U_0}\catdef$ of this restriction functor is surjective.

\begin{ex}
Prove that $\Ext^2_{\O_{X_0}}(\diff_{X_0},\O_{X_0})=0$, and that the differential
   \[
   d_{X_0}i^*\colon T_{X_0}\catdef\arr T_{U_0}\catdef
   \]
of the restriction functor is surjective.
\end{ex}

Finally from these two facts, using a reasoning similar to that in the proof of Proposition~\ref{immers}, one shows that every formal deformation of $U_0$ is isomorphic to the restriction of a formal deformation of $X_0$.

Using this and the results on algebraization of Section~\ref{capform}, we get the following Proposition.

\begin{prop}\label{projalgebr}
Let $X_0$ be a projective nodal curve over $k$, with rational nodes $p_1,\hdots,p_r \in X_0$, and $u_1,\hdots,u_r \in \Lambda$ be arbitrary elements. Then there exists a flat and projective scheme $X$ over $\Lambda$, having closed fiber isomorphic to $X_0$, and such that
$$
\wh{\O}_{X,p_i}\simeq \Lambda\ds{x,y}/(xy-u_i)\,.
$$
for every node $p_i \in X_0$.
\end{prop}

\begin{proof}
Let $\Lambda\ds{t_1,\hdots,t_r}$ be the base ring of the miniversal deformation of $U_0$ of Proposition~\ref{miniversnod}, and consider the homomorphism of $\Lambda$-algebras $\Lambda\ds{t_1,\hdots,t_r}\to \Lambda$ defined by $t_i\mapsto u_i$. This induces by pullback a formal deformation $\wh{U}=\{U_n,f_n\}_\nin$ of $U_0$ over $\Lambda$.

By the remarks above, we can find a formal deformation $\wh{X}=\{X_n,g_n\}_\nin$ such that the restriction $\wh{i^*}(\wh{X})$ is isomorphic to $\wh{U}$. Since $X_0$ is projective and $\H^2(X_0,\O_{X_0})=0$, by Theorem~\ref{algebr} the formal deformation $\wh{X}$ is algebraizable, that is, we can find a flat and projective scheme $X\to \spec \Lambda$ inducing $\wh{X}$.

In particular $X$ has closed fiber isomorphic to $X_0$, and since it restricts to a formal deformation isomorphic to $\wh{U}$ constructed above, by Proposition~\ref{nodaldef1} we have
$$
\wh{\O}_{X,p_i}\simeq \Lambda\ds{x,y}/(xy-u_i)\,.
$$
for every node $p_i$.
\end{proof}

In particular we deduce the following corollary, which shows that if $\Lambda$ is one-dimensional, we can always deform $X_0$ to a smooth curve.

\begin{cor}
Let $X_0$ be a projective nodal curve as in the preceding Proposition, and suppose that $\dim \Lambda = 1$ (for example $\Lambda=k\ds{t}$). Then there exists a flat and projective morphism $X\xrightarrow{\pi} \spec \Lambda$ such that the closed fiber is isomorphic to $X_0$, and $X \setminus X_0 \to \spec \Lambda\setminus \{\m_\Lambda\}$ is smooth.
\end{cor}

\begin{proof}
Let $u \in \m_\Lambda$ be a system of parameters for $\Lambda$, and take the deformation $X\to \spec \Lambda$ of $X_0$ given by Proposition~\ref{projalgebr}, with $u_i=u$ for every $i$.

Let $U$ be the open subset of $X$ on which the coherent sheaf $\diff_{X/\Lambda}$ is locally free of rank $1$ (or equivalently where $X$ is smooth over $\Lambda$). We want to show that $U=X\setminus \{p_1,\hdots,p_r\}$.

Consider an irreducible component $V$ of $X\setminus U$, with generic point $p \in V\subseteq X$. Since $V$ is closed in $X$ and $X\to \spec \Lambda$ is proper, we must have $V\cap X_0\neq \emptyset$ (because the image of $X\to \spec \Lambda$ contains the maximal ideal $\m_\Lambda$). Since $X_0$ is smooth outside $p_1,\hdots,p_r$, there exists an $i$ such that $p_i \in V\cap X_0$; we will show that $V=\{p_i\}$, and this will conclude the proof.

We consider the complete local ring $R=\wh{\O}_{X,p_i}\simeq \Lambda\ds{x,y}/(xy-u)$, and the module of continuous K\"{a}hler differentials $\wh{\diff}_{R/\Lambda}$ (see Appendix~\ref{appb}), which is an $R$-module with two generators $dx, dy$, and the relation $ydx+xdy=0$. This can also be seen as the cokernel of the homomorphism $R\to R\oplus R$ given by multiplication by the vector $(y,x)$.

If $p \in \spec R$ does not contain the ideal $(x,y)$, then one of $x,y$ is invertible in $p$, and then $\wh{\diff}_{R/\Lambda}$ is locally free of rank 1 over $R_p$. Since the radical of $(x,y)$ is $\m_R$, we conclude that $\wh{\diff}_{R/\Lambda}|_{\spec R\setminus \{\m_R\}}$ is locally free of rank 1.

Now the natural morphism $\O_{X,p_i}\to \wh{\O}_{X,p_i}=R$ is faithfully flat, and then $\spec R\to \spec\O_{X,p_i}$ is flat and surjective. Moreover the inverse image of the closed point $\m_{p_i}$ is $\{\m_R\}$, and so we can restrict the morphism above to
$$
\spec R\setminus \{\m_R\} \arr \spec\O_{X,p_i}\setminus \{\m_{p_i}\}
$$
that is flat and surjective too. From this, and the fact that the pullback of $\diff_{X/\Lambda}$ to $\spec R\setminus \{\m_R\}$ is locally free of rank 1 (since it is isomorphic to $\wh{\diff}_{R/\Lambda}|_{\spec R\setminus \{\m_R\}}$, as one easily checks), we get that its pullback to $\spec\O_{X,p_i}\setminus \{\m_{p_i}\}$ along the morphism
$$
\spec\O_{X,p_i}\setminus \{\m_{p_i}\}\arr X
$$
is also locally free of rank 1.

Finally, the generic point $p$ of $V$ is in the image of the morphism above (since this image it is the set of the generic points of irreducible components of $X$ containing $p_i$), but the stalk $\diff_{X/\Lambda,p}$ is not free of rank 1 by hypothesis. From this we get that the maximal ideal $\m_{p_i}$ goes to $p$, or in other words $p=p_i$, and $V=\{p_i\}$ (since $p_i$ is closed), as we claimed.
\end{proof}

\appendix

\section{Linear functors}\label{appa}

In this Appendix we give some results on functors from categories of modules (or vector spaces) that preserve finite products. Throughout this Appendix $A$ will be a noetherian ring (commutative and with identity, as usual).

Let $F\colon \fmod\to \set$ be a functor. If $M, N\in\fmod$ , the two projections $M\oplus N\to M$ and $M\oplus N\to N$ induce two functions $F(M\oplus N)\to F(M)$ and $F(M\oplus N) \to F(N)$, and in turn these induce $\phi_{M,N}\colon F(M\oplus N)\to F(M)\times F(N)$.

\begin{defin}\label{def:preserve-products}
A functor $F\colon \fmod\to \set$ is said to \gr{preserve finite products} if the function $\phi_{M,N}$ is bijective for every $M,N \in \fmod$, and $F(0)\neq \emptyset$.
\end{defin}


\begin{defin}
A functor $F\colon \fmod\to \mod$ is said to be $A$\gr{-linear} if for every $M,N \in\fmod$ the function
$$
\Hom_A(M,N)\arr \Hom_A(F(M),F(N))
$$
is a homomorphism of $A$-modules.
\end{defin}

It is easy to see that if $F\colon \fmod\to \mod$ is $A$-linear, then the induced functor $\fmod\to \set$ preserves finite products, and the bijection $\phi_{M,N}\colon F(M\oplus N)\to F(M)\times F(N)$ is actually an isomorphism of $A$-modules (with the product structure on the target).

The following Proposition shows that if a functor $F\colon \fmod\to \set$ preserves finite products, then each $F(M)$ has a canonical structure of $A$-module.

\begin{prop}\label{app.lifting}
Let $F\colon \fmod\to \set$ be a functor that preserves finite products. Then there exists a unique $A$-linear lifting $\widetilde{F}\colon \fmod\to \mod$ of $F$, that is, an $A$-linear functor such that its composite with the forgetful functor $\mod\to \set$ is $F$.
\end{prop}

With ``unique'' above we mean really unique, not only up to isomorphism.

\begin{proof}
We define a structure of $A$-module on every $F(M)$, and call $\widetilde{F}(M)$ the set $F(M)$ with this $A$-module structure. First, notice that $F(0)$ (where $0$ is the zero $A$-module) has exactly one element. In fact the two projections $0\oplus 0\to 0$ are the same function, and then the same is true for the two induced functions $F(0\oplus 0)\to F(0)$. But now we know that $F(0\oplus 0)\simeq F(0)\times F(0)$, and for the projections on the two factors to be the same function, we must have that $F(0)$ has at most one element. Finally, it has at least one, since $F(0)\neq \emptyset$ by hypothesis.

Now fix $M \in \fmod$, and notice that we have a unique homomorphism $0\to M$. We define the image of the induced $F(0)\to F(M)$ to be the zero vector of $F(M)$.

Next, we define scalar multiplication: if $a \in A$, we have a homomorphism $\mu_A\colon M\to M$ given by scalar multiplication by $a$. We define then scalar multiplication by $a$ in $F(M)$ to be the induced function $F(\mu_A)\colon F(M)\to F(M)$.

Finally, we define addition. Consider the ``sum'' homomorphism $+\colon M\oplus M\to M$ defined by $(m,n)\mapsto m+n$; this induces a function $F(M\oplus M)\to F(M)$, which, using the bijection $\phi_{M,M}\colon F(M\oplus M)\simeq F(M)\times F(M)$, gives a function $F(M)\times F(M)\to F(M)$. We define the sum in $F(M)$ by means of the last function.

Now one should verify that the data defined give a structure of $A$-module on $F(M)$, and that if $M\to N$ is a homomorphism, then the induced $F(M)\to F(N)$ is a homomorphism too. The method to verify the various axioms (and also the last fact about homomorphisms) is to rewrite everything using of commutative diagrams, and then use the appropriate functorialities to conclude. We leave these verifications (as well as those for uniqueness and linearity) to the reader.
\end{proof}

Now we turn to natural transformations.

\begin{defin}
Let $F$, $G\colon \fmod\to \mod$ be two functors. A natural transformation $\alpha\colon F\to G$ is said to be $A$-linear if for every $M\in\fmod$ the function $\alpha_M\colon F(M)\to G(M)$ is $A$-linear.
\end{defin}

The following Proposition is useful when one has to show that a bijection is an isomorphism of modules, and can be proved by using the defining property of a natural transformation, and the structure of module given in the previous Proposition.

\begin{prop}\label{app.nat.linear}
Let $F$, $G\colon \fmod\to \set$ be two functors that preserve finite products, $\widetilde{F},\widetilde{G}\colon \fmod\to \mod$ the two $A$-linear liftings coming from the preceding Proposition, and $\alpha\colon F\to G$ a natural transformation. Then for every $M \in \fmod$ the function $\alpha_M\colon \widetilde{F}(A)\rightarrow \widetilde{G}(A)$ is $A$-linear, and so $\alpha$ induces an $A$-linear natural transformation $\widetilde{\alpha}\colon \widetilde{F}\to \widetilde{G}$.
\end{prop}

Finally, we see that if $A=k$ is a field, then $k$-linear functors are particularly simple to describe.

\begin{prop}\label{app.natural}
Let $G\colon \fvect\to \vect$ be a $k$-linear functor. Then for every $V\in\fvect$ there is a functorial isomorphism
$$
G(V)\simeq V\otimes_k G(k)\,.
$$
In particular $G$ is an exact functor (carries exact sequences into exact sequences), since the functor $-\otimes_k G(k)$ is.
\end{prop}

\begin{proof}
We define a natural transformation $\tau\colon -\otimes_k G(k)\to G$ as follows: for $V \in \fvect$ we define $\tau(V)\colon V\otimes_k F(k)\to F(V)$ by
$$
\tau(V)(v\otimes \alpha)=F(\phi_v)(\alpha)
$$
where $\phi_v\colon k\to V$ is the $k$-linear function sending $1$ into $v$. It is readily checked that $\tau$ is a natural transformation.

We check that each $\tau(V)$ is an isomorphism. If $V=k$, then $\tau_k\colon k\otimes_k F(k)\to F(k)$ is easily seen to be the canonical isomorphism defined by $a\otimes \alpha\mapsto a\cdot \alpha$.

If $V=k^n$, then we have a commutative diagram
$$
\xymatrix@C+10pt{
k^n\otimes_k F(k)\ar[r]^>>>>>>>>{\tau_{k^n}}\ar[d] & F(k^n)\ar[d] \\
(k\otimes_k F(k))^n\ar[r]^>>>>>>>{(\tau_k)^n} & F(k)^n
}
$$
where the left vertical arrow is the canonical isomorphism, the right hand vertical one is the isomorphism given by $k$-linearity of $F$ to $k^n\simeq k\oplus \cdots \oplus k$ (applied $n-1$ times), and the bottom one is an isomorphism because $\tau_k$ is. It follows that $\tau_{k^n}$ is an isomorphism too.

Finally, for a general $V \in \fvect$, we take an isomorphism $V\simeq k^n$ where $n$ is the dimension of $V$, and reduce this case to the preceding one.
\end{proof}

\section{Noetherian complete local rings}\label{appb}

In this Appendix we gather some definitions and facts about noetherian complete local algebras over $\Lambda$ (which is as usual a noetherian complete local ring) with residue field $k$, which we apply in Sections \ref{capform} and \ref{capcurv}. We denote the category of such rings by $\compl$, where as usual we consider only local homomorphisms (which are also precisely the continuous ones with respect to the natural topologies).

\subsection*{Vertical tangent and cotangent spaces}

First of all we set up some notation. Let $R \in \compl$, and $\m_R\subseteq R$ be the maximal ideal as usual; there is another important ideal of $R$, the extension $\m_\Lambda R\subseteq \m_R \subseteq R$ of the maximal ideal of $\Lambda$. In this situation, we denote by $R_n$ the quotient $R/\m_R^{n+1}$, which is an object of $\artl$, and by $\overline{R}$ the quotient $R/\m_\Lambda R$,  an object of $(\Comp_k)$. If $\phi\colon R\to S$ is a homomorphism in $\compl$, we will denote by $\phi_n\colon R_n\to S_n$ and $\overline{\phi}\colon \overline{R}\to \overline{S}$ the induced ones.

So $\overline{R}_n \in \art$ will be the quotient $\overline{R}/\m_{\overline{R}}^{n+1}\simeq R_n/\m_\Lambda R_n\simeq R_n\otimes_\Lambda k$, and in particular we have
$$
\overline{R}_1\simeq k\oplus \m_{\overline{R}_1} \simeq k\oplus \m_{\overline{R}}/\m_{\overline{R}}^2
$$
because $\m_{\overline{R}_1}^2=(0)$.

\begin{defin}
The \gr{vertical cotangent space} of $R$ is the finite-dimensional $k$-vector space
$$
T^\vee_\Lambda R=\m_R/(\m_\Lambda R+\m_R^2)\,.
$$
Its dual
$$
T_\Lambda R=(\m_R/(\m_\Lambda R+\m_R^2))^\vee
$$
is called the \gr{vertical tangent space} of $R$.
\end{defin}

The name ``vertical tangent space'' comes from the fact that $T_\Lambda R$ is the tangent space of the fiber of the morphism $\spec R\to \spec \Lambda$ over the maximal ideal (which is $\spec\overline{R}$), at the only closed point. In fact one easily checks that there is a canonical isomorphism
$$
T_\Lambda R=(\m_R/(\m_\Lambda R+\m_R^2))^\vee\simeq (\m_{\overline{R}}/\m_{\overline{R}}^2)^\vee
$$

As one expects, there is a related notion of differential of a homomorphism $\phi\colon R\to S$ in $\compl$. This comes from the fact that $\phi(\m_R)\subseteq \m_S$ and $\phi(\m_\Lambda R+\m_R^2)\subseteq \m_\Lambda S+\m_S^2$, so $\phi$ induces a $k$-linear map
$$
\phi_*\colon \m_R/(\m_\Lambda R+\m_R^2)\arr \m_S/(\m_\Lambda S+\m_S^2)
$$
between the cotangent spaces, which we call the \gr{codifferential} of $\phi$.

Dualizing, we get another $k$-linear map
$$
d\phi\colon (\m_S/(\m_\Lambda S+\m_S^2))^\vee \arr (\m_R/(\m_\Lambda R+\m_R^2))^\vee
$$
that we call \emph{the differential of $\phi$}. It is the differential of the morphism induced by $\phi$ between the closed fibers $\spec\overline{S}\to \spec\overline{R}$.

These constructions are clearly functorial, in the sense that differential and codifferential of a composite coincides with the composites of the differentials and codifferentials respectively.

We have the following important Proposition.

\begin{prop}\label{codiff.surj}
Let $R,S \in \compl$, and $\phi\colon R\to S$ be a homomorphism. If the codifferential $\phi_*\colon T^\vee_\Lambda R\to T^\vee_\Lambda S$ is surjective, than $\phi$ itself is surjective.
\end{prop}

\begin{proof}
Let us consider first the homomorphisms of $k$-algebras $\overline{\phi}_n\colon \overline{R}_n\to \overline{S}_n$ induced by $\phi$. We show inductively that $\overline{\phi}_n$ is surjective for every $n$.

To do this, notice that surjectivity of the codifferential $\phi_*\colon \m_{\overline{R}}/\m_{\overline{R}}^2\to \m_{\overline{S}}/\m_{\overline{S}}^2$ implies that of the map
$$
f_n\colon \m_{\overline{R}}^n/\m_{\overline{R}}^{n+1}\arr \m_{\overline{S}}^n/\m_{\overline{S}}^{n+1}
$$
induced by $\phi_n$, for any $n$ (as is easily checked). Now we come to $\overline{\phi}_n$: if $n=1$, we have that $\overline{\phi}_1\colon k\oplus T^\vee_\Lambda R\to k\oplus T^\vee_\Lambda S$ is surjective because the codifferential $\phi_*$ is, by hypothesis. Suppose that we know that $\overline{\phi}_{n-1}$ is surjective; we have a commutative diagram with exact rows
$$
\xymatrix{
0\ar[r]& \m_{\overline{R}}^n/\m_{\overline{R}}^{n+1}\ar[r]\ar[d]^{f_n} & \overline{R}_n \ar[r]\ar[d]^{\overline{\phi}_n} & \overline{R}_{n-1}\ar[r]\ar[d]^{\overline{\phi}_{n-1}} & 0\\
0\ar[r] & \m_{\overline{S}}^n/\m_{\overline{S}}^{n+1}\ar[r]& \overline{S}_n \ar[r] & \overline{S}_{n-1}\ar[r] & 0
}
$$
and by diagram chasing the surjectivity of $\overline{\phi}_{n-1}$ and $f_n$ implies that of $\overline{\phi}_n$.

Now consider $\phi_n\colon R_n\to S_n$; we show that all these homomorphisms are surjective as well. Notice that $R_n$ and $S_n$ are finite as $\Lambda$-modules, because they have a finite filtration (given by the powers of the maximal ideal), such that successive quotients are finite-dimensional $k$-vector spaces.

Recall also that $\overline{R}_n\simeq R_n\otimes_\Lambda k$ and $\overline{S}_n\simeq S_n\otimes_\Lambda k$, and $\overline{\phi}_n$ is the homomorphism induced by $\phi_n$. Since $\overline{\phi}_n$ is surjective, we can apply Nakayama's Lemma and deduce that $\phi_n$ is surjective too.

Finally, we pass to $\phi\colon R\to S$, which is the projective limit of the homomorphisms $\phi_n$. If we set $K_n \eqdef \ker(R_n\to S_n)$, we have for every $n$ an exact sequence
$$
\xymatrix{
0\ar[r] & K_n\ar[r] & R_n\ar[r] & S_n\ar[r] & 0
}
$$
that together give an exact sequence of projective systems. Since in our case $R_n$ is artinian (and so $K_n$ is as well), the Mittag-Leffler condition (for every $n$ the image of $K_{n+k}\to K_n$ is the same for all $k$'s large enough) is certainly satisfied, and then the induced homomorphism
$$
\invlim \phi_n=\phi\colon \invlim R_n\simeq R\longrightarrow \invlim S_n\simeq S
$$
is surjective.
\end{proof}

From the last Proposition we get the following corollary.

\begin{cor}\label{complll}
Let $R,S \in \compl$, and $\phi\colon R\to S$ be a homomorphism such that the codifferential $\phi_*\colon T^\vee_\Lambda R\to T^\vee_\Lambda S$ is surjective. Then:

\begin{enumeratei}

\item If $\ell(R_n)=\ell(S_n)$ for all $n$ (where $\ell(-)$ denotes the length as a $\Lambda$-module), then $\phi$ is an isomorphism.

\item If $\psi\colon S\to R$ is a homomorphism, and the codifferential $\psi_*\colon T^\vee_\Lambda S\to T^\vee_\Lambda R$ is surjective, then $\phi$ is an isomorphism.

\item If $R$ and $S$ are isomorphic, then $\phi$ is an isomorphism.

\end{enumeratei}

\end{cor}

\begin{proof}
The first assertion follows from the fact that $\ell(R_n)=\ell(S_n)$ implies $\ell(K_n)=0$ (with the notation of the preceding proof), and consequently that each $\phi_n\colon R_n\to S_n$ is an isomorphism. In conclusion $\phi=\invlim \phi_n$ is an isomorphism as well.

For the second statement, if $\psi\colon S\to R$ is a homomorphism with surjective codifferential $\psi_*\colon T^\vee_\Lambda S\to T^\vee_\Lambda R$, by the proof of the preceding Proposition we deduce that $\psi_n\colon S_n\to R_n$ is surjective for every $n$, and this, together with the fact that $\phi_n\colon R_n\to S_n$ is surjective as well, implies that $\ell(R_n)=\ell(S_n)$, so we can apply the first part of the corollary.

This last argument clearly proves (iii) as well (because if $\psi\colon S\to R$ is an isomorphism, then in particular the codifferential will be surjective), and this concludes the proof.
\end{proof}

Notice that it is not sufficient to have a surjective map $T^\vee_\Lambda S\to T^\vee_\Lambda R$ to conclude that $\phi$ above is an isomorphism, but we must have a homomorphism $S\to R$ with surjective codifferential.

In particular the fact that $\phi_*$ is an isomorphism does not imply that $\phi$ itself is.

\subsection*{Power series rings}

Now we turn to power series rings over $\Lambda$. For any $n$, the power series ring on $n$ indeterminates $R \eqdef \Lambda\ds{x_1,\hdots,x_n}$ is an object of $\compl$. Since the ideal $\m_\Lambda R\subseteq R$ coincides with the kernel of the natural homomorphism $\Lambda\ds{x_1,\hdots,x_n}\to k\ds{x_1,\hdots,x_n}$ (as one easily checks, using noetherianity of $\Lambda$), we get that $\overline{R} \simeq k\ds{x_1,\hdots,x_n}$. In particular
$$
T^\vee_\Lambda R\simeq \m_{k\ds{x_1,\hdots,x_n}}/\m_{k\ds{x_1,\hdots,x_n}}^2
$$
is a $k$-vector space of dimension $n$, with basis $[x_1],\hdots,[x_n]$.

The next Proposition shows that power series rings have properties similar to those of polynomial rings, with respect to complete algebras.

\begin{prop}\label{psr}
Let $R \in \compl$, and $a_1,\hdots,a_n \in \m_R$. Then there exists a unique homomorphism $\phi\colon \Lambda\ds{x_1,\hdots,x_n} \to R$ such that $\phi(x_i)=a_i$.
\end{prop}

\begin{proof}
By the properties of polynomial rings, for every $k$ we have a unique homomorphism
$$
\phi_k\colon \Lambda\ds{x_1,\hdots,x_n}_k\simeq \Lambda[x_1,\hdots,x_n]/\m_{\Lambda[x_1,\hdots,x_n]}^{k+1}\longrightarrow R_k
$$
sending $[x_i]$ into $[a_i]$. By completeness we get a homomorphism
$$
\phi \eqdef \invlim \phi_k \colon \invlim\Lambda\ds{x_1,\hdots,x_n}_k\simeq \Lambda\ds{x_1,\hdots,x_n} \longrightarrow \invlim R_k \simeq R
$$
such that $\phi(x_i)=a_i$.

Moreover if $\psi\colon \Lambda\ds{x_1,\hdots,x_n}\to R$ is a homomorphism with this property, then for every $k$ the induced homomorphism
$$
\psi_k\colon \Lambda[x_1,\hdots,x_n]/\m_{\Lambda[x_1,\hdots,x_n]}^{k+1}\arr R_k
$$
sends $[x_i]$ into $[a_i]$, and so coincides with $\phi_k$ above. This implies $\psi=\invlim \psi_k =\invlim \phi_k=\phi$ and concludes the proof.
\end{proof}

The following is an immediate consequence of Proposition~\ref{psr} and part (ii) of Corollary~\ref{complll}.

\begin{cor}\label{psr.iso}
Let $R \in \compl$, and assume we have a homomorphism $\phi\colon R\to \Lambda\ds{x_1,\hdots,x_n}$ such that the codifferential $\phi_*\colon T^\vee_\Lambda R\to T^\vee_\Lambda \Lambda\ds{x_1,\hdots,x_n}$ is an isomorphism. Then $\phi$ is an isomorphism.
\end{cor}

\begin{proof}
Let us take elements $a_1,\hdots,a_n \in \m_R$ such that $\phi_*([a_i])=[x_i]$, and such that $[a_1],\hdots,[a_n]$ form a basis of $T^\vee_\Lambda R$. By Proposition~\ref{psr} we can find then a homomorphism $\psi\colon \Lambda\ds{x_1,\hdots,x_n}\to R$ such that $\psi(x_i)=a_i$; its codifferential will then be surjective, and part (ii) of Corollary~\ref{complll} lets us conclude that $\phi$ is an isomorphism.
\end{proof}

From this we get a description of noetherian complete local rings as quotients of power series rings.

\begin{cor}\label{quot.psr}
Every $R \in \compl$ is a quotient of the ring $\Lambda\ds{x_1,\hdots,x_n}$ for some $n$. Moreover, the minimum such $n$ is the dimension of the vertical cotangent space $T^\vee_\Lambda R$ of $R$.
\end{cor}

\begin{proof}
Set $n \eqdef \dim_k T^\vee_\Lambda R$, and consider elements $a_1,\hdots,a_n \in \m_R$ such that $[a_1],\hdots,[a_n]$ is a basis of $T^\vee_\Lambda R$. By Proposition~\ref{psr} we can define a homomorphism $\phi\colon \Lambda\ds{x_1,\hdots,x_n}\to R$ such that $\phi(x_i)=a_i$; its codifferential will then be surjective, and by Proposition~\ref{codiff.surj} $\phi$ will be surjective too. In other words, $R$ is a quotient of $\Lambda\ds{x_1,\hdots,x_n}$.

On the other hand if $\phi\colon \Lambda\ds{x_1,\hdots,x_r}\to R$ is surjective then the codifferential $\phi_*\colon T^\vee_\Lambda \Lambda\ds{x_1,\hdots,x_r}\to T^\vee_\Lambda R$ is surjective too, and this implies that $r\geq n$.
\end{proof}

Finally we prove a criterion that characterizes power series rings as formally smooth algebras in $\compl$.

\begin{thm}\label{app.smoothcrit}
Let $R \in \compl$. Then $R$ is a power series ring if and only if for every surjection $A'\to A$ in $\artl$ and every homomorphism $R\to A$, we can find a lifting $R\to A'$.
\end{thm}

\begin{proof}
If $R$ is a power series ring, then Proposition~\ref{psr} implies that we can lift homomorphisms along small extensions.

Conversely, suppose that the lifting property holds, and take a homomorphism $\phi\colon \Lambda\ds{x_1,\hdots,x_n}\to R$ such that the codifferential $\phi_*\colon T^\vee_\Lambda \Lambda\ds{x_1,\hdots,x_n}\to T^\vee_\Lambda R$ is an isomorphism (using the last corollary, for example).

Now notice that the quotient map $\Lambda\ds{x_1,\hdots,x_n}_1\to T^\vee_\Lambda \Lambda\ds{x_1,\hdots,x_n}$ is a surjection in $\artl$, and then by hypothesis we can lift the homomorphism $R\to T^\vee_\Lambda R \xrightarrow{(\phi_*)^{-1}}T^\vee_\Lambda \Lambda\ds{x_1,\hdots,x_n}$ to $R\to \Lambda\ds{x_1,\hdots,x_n}_1$
$$
\xymatrix{
&R\ar[d] \ar@/_10pt/@{-->}[ld] \\
\Lambda\ds{x_1,\hdots,x_n}_1\ar[r] &T^\vee_\Lambda \Lambda\ds{x_1,\hdots,x_n}.
}
$$
Likewise, since the quotient map $\Lambda\ds{x_1,\hdots,x_n}_k\to \Lambda\ds{x_1,\hdots,x_n}_{k-1}$ is a surjection in $\artl$, we can lift inductively the homomorphism $R\to \Lambda\ds{x_1,\hdots,x_n}_{k-1}$ to a homomorphism $R\to \Lambda\ds{x_1,\hdots,x_n}_k$.

Finally, taking the projective limit of the sequence of compatible homomorphisms above, we obtain a homomorphism $\psi\colon R\to \Lambda\ds{x_1,\hdots,x_n}$ such that the codifferential $\psi_*\colon  T^\vee_\Lambda R\to T^\vee_\Lambda \Lambda\ds{x_1,\hdots,x_n}$ is an isomorphism (the inverse of $\phi_*$), and by Proposition~\ref{psr.iso} this implies that $\phi$ is an isomorphism, so $R$ is a power series ring.
\end{proof}

Actually the criterion can be strengthened by replacing ``surjection'' $A'\to A$ by ``small extension''. To see this it suffices to factor a surjection as a composite of small extensions and lift the homomorphism successively, as usual.

\subsection*{Continuous K\"{a}hler differentials}

In this section we introduce a module of differentials for objects of $\compl$ that is much more useful than the standard one.

Let $R$ be an object of $\compl$. Then we have the usual module of K\"{a}hler differentials $\diff_{R/\Lambda}$ with the universal $\Lambda$-derivation $d\colon R\to \diff_{R/\Lambda}$, which has the following universal property: if $D\colon R\to M$ is a $\Lambda$-derivation then there is a unique homomorphism of $R$-modules $f\colon \diff_{R/\Lambda}\to M$ such that $D=f\circ d$.

For some applications this module is too large: for example, one can show that $\diff_{k\ds{x}/k}$ is not finitely generated over $k\ds{x}$, since the field of fractions $k\dr{x}$ of $k\ds{x}$ has infinite transcendence degree over $k$, and
$$
\diff_{k\ds{x}/k}\otimes_{k\ds{x}} k\dr{x}=\diff_{k\dr{x}/k}
$$
is not finitely generated over $k\dr{x}$.

Because of this we define another module of differentials that is better behaved. We consider derivations $D\colon R\to M$ where $M$ is a module that is separated with respect to the $\m_R$-adic topology, that is, the intersection of the submodules $\{\m_R^i M\}_{i \in \N}$ is the zero submodule. For example, by one of Krull's Theorems, finitely generated $R$-modules are separated.

We want then a finitely generated $R$-module $\wh{\diff}_{R/\Lambda}$ with a derivation $d\colon R\to \wh{\diff}_{R/\Lambda}$, such that for every derivation $D\colon R\to M$, where $M$ is a separated $R$-module, there exists a homomorphism $f\colon \wh{\diff}_{R/\Lambda}\to M$ such that $D=f\circ d$.

Write $R$ as a quotient of a power series ring (Corollary~\ref{quot.psr})
$$
R\simeq \Lambda\ds{x_1,\hdots,x_n}/I
$$
and suppose that $I=(f_1,\hdots,f_k)$. We consider the free $R$-module on $n$ elements $dx_1,\hdots,dx_n$, and its submodule $J$ generated by the elements
$$
df_i=\left[\frac{\partial f_i}{\partial x_1}\right]dx_1+\cdots+\left[\frac{\partial f_i}{\partial x_n}\right]dx_n
$$
for $i=1,\hdots,k$; we define then
$$
\wh{\diff}_{R/\Lambda}=(Rdx_1\oplus \cdots \oplus Rdx_n)/J\,.
$$
Moreover we have a derivation $d\colon R\to \wh{\diff}_{R/\Lambda}$ given by
$$
d([g])=\left[\left[\frac{\partial g}{\partial x_1}\right]dx_1+\cdots+\left[\frac{\partial g}{\partial x_n}\right]dx_n\right]
$$
for $[g] \in R$, which is easily seen to be well-defined.

\begin{prop}
The $R$-module $\wh{\diff}_{R/\Lambda}$ and the derivation $d\colon R\to \wh{\diff}_{R/\Lambda}$ have the universal property above.
\end{prop}

\begin{proof}
We sketch the idea of the proof, without going into details. Let $D\colon R\to M$ be a $\Lambda$-derivation of $R$ into a separated $R$-module $M$.

We start by defining $Rdx_1\oplus \cdots \oplus Rdx_n\to M$ by saying that $dx_i$ goes to $D([x_i])$, and then extending by linearity. To see that this induces a homomorphism on the quotient $\wh{\diff}_{A/\Lambda}$, the key point is to see that the derivation $D$ is completely determined by $D([x_i])$ for $i=1,\hdots,n$.

This is clearly true for $D([p])$ where $p$ is a polynomial, by using the Leibnitz rule repeatedly. The fact that $D$ is uniquely determined on power series follows from the fact that derivations are continuous with respect to the $\m_R$-adic topology, and from separateness of $M$.
\end{proof}

\begin{defin}
The $R$-module $\wh{\diff}_{R/\Lambda}$ is called the \gr{module of continuous K\"{a}hler differentials} of $R$, and $d$ is the \gr{universal continuous derivation}.
\end{defin}

The Proposition above implies in particular that changing the presentation of $R$ as a quotient of a power series ring we get isomorphic modules of continuous differentials.

Now suppose that $R,S \in \compl$, and that $\phi\colon R\to S$ is a surjection with kernel $I\subseteq R$. Because of the universal property of $\wh{\diff}_{R/\Lambda}$ and the fact that the composite $R\to S\to \wh{\diff}_{S/\Lambda}$ is a $\Lambda$-derivation, we get a homomorphism of $R$-modules $\wh{\diff}_{R/\Lambda}\to \wh{\diff}_{S/\Lambda}$, which tensoring by $S$ induces a homomorphism of $S$-modules $f\colon \wh{\diff}_{R/\Lambda}\otimes_R S\to \wh{\diff}_{S/\Lambda}$.

Moreover the universal derivation $d\colon R\to \wh{\diff}_{R/\Lambda}$ induces as usual a homomorphism of $S$-modules $I/I^2\to \wh{\diff}_{R/\Lambda}\otimes_R S$ that we still denote by $d$. The following Proposition is proved in the same way as its analogue for the standard modules of differentials.

\begin{prop}\label{cont.conorm}
If $\phi\colon R\to S$ is a surjection in $\compl$, then the sequence of $S$-modules
$$
\xymatrix{
I/I^2\ar[r]^-d & \wh{\diff}_{R/\Lambda}\otimes_R S\ar[r]^-f & \wh{\diff}_{S/\Lambda}\ar[r] & 0
}
$$
is exact.
\end{prop}

The sequence above is called the \gr{conormal sequence} associated with the homomorphism $\phi$.

\section{Some other facts and constructions}\label{appc}

In this Appendix we gather some other miscellaneous standard results and constructions that are used throughout this work.

\subsection*{Fibered products of categories}

Let $\F,\G,\mathcal{H}$ be three categories, with two functors $F\colon \F\to \mathcal{H}$ and $G\colon \G\to \mathcal{H}$. We want to define a ``fibered product'' category $\F\times_{\mathcal{H}}\G$ with two functors $\pi_\F\colon \F\times_{\mathcal{H}}\G\to \F$ and $\pi_\G\colon \F\times_{\mathcal{H}}\G\to\G$, such that the composites $F\circ \pi_\F$ and $G\circ \pi_\G$ are isomorphic as functors $\F\times_{\mathcal{H}}\G\to \mathcal{H}$, and such that for any other category $\C$ with two functors $\phi_\F\colon \C\to \F$ and $\phi_\G\colon \C\to \G$ and a fixed isomorphism of functors $F\circ \phi_\F\simeq G\circ \phi_\G$ there exists a dotted functor as in the diagram below
$$
\xymatrix@+10pt{
\C\ar@/^10pt/[rrd]^{\phi_\F}\ar@/_10pt/[ddr]_{\phi_\G}\ar@{-->}[rd]^\psi&&\\
&\F\times_{\mathcal{H}}\G\ar[r]^>>>>>>>>{\pi_\F}\ar[d]_{\pi_\G}&\F\ar[d]^F\\
&\G\ar[r]_G & \mathcal{H}
}
$$
such that $\pi_\F\circ \psi=\phi_\F$ and $\pi_\G\circ \psi=\phi_\G$ (which are actual equalities, and not merely isomorphisms of functors).

We define such a category $\F\times_{\mathcal{H}}\G$ as follows:
\begin{description}
\item[Objects] are triplets $(X,Y,f)$ where $X \in \F, Y \in \G$ and $f\colon F(X)\to G(X)$ is an isomorphism in the category $\mathcal{H}$.
\item[Arrows] from $(X,Y,f)$ to $(Z,W,g)$ are pairs $(h, k)$ of arrows $h\colon X\to Z$ of $\F$ and $k\colon Y\to W$ of $\G$, such that the diagram
$$
\xymatrix@+5pt{
F(X)\ar[r]^{F(h)}\ar[d]_f & F(Z)\ar[d]^g\\
G(Y)\ar[r]_{F(k)} & G(W)
}
$$
is commutative.
\end{description}

Composition of arrows is defined in the obvious way, as well as the two functors $\pi_\F, \pi_\G$; for example $\pi_\F\colon \F\times_{\mathcal{H}}\G\to \F$ sends an object $(X,Y,f)$ into $X \in \F$, and an arrow $(h,k)$ to $h$.

Moreover notice that the composites $F\circ \pi_\F$ and $G\circ \pi_\G$ are clearly isomorphic: starting from $(X,Y,f) \in \F\times_{\mathcal{H}}\G$ we have $(F\circ \pi_\F)(X,Y,f)=F(X)$ and $(G\circ \pi_\G)(X,Y,f)=G(Y)$, so $f\colon F(X)\to G(Y)$ gives the desired isomorphism. The compatibility property on arrows ensures that these isomorphisms altogether give a natural transformation.

\begin{prop}
The category $\F\times_{\mathcal{H}}\G$ with the functors $\pi_\F,\pi_\G$ has the property stated above.
\end{prop}

\begin{proof}
Suppose we have a category $\C$ with two functors $\phi_\F\colon \C\to \F$ and $\phi_\G\colon \C\to \G$, and a fixed isomorphism of functors $\alpha\colon F\circ \phi_\F\simeq G\circ \phi_\G$. We define a functor $\psi\colon \C\to \F\times_{\mathcal{H}}\G$ as follows: if $X \in \C$, we set $\psi(X) \eqdef (\phi_\F(X),\phi_\G(X),\alpha(X))$, and an arrow $f\colon X\to Y$ of $\C$ goes to the arrow $(\phi_\F(f),\phi_\G(f))$ of $\F\times_\mathcal{H}\G$.

It is immediate to check that $\psi$ is well-defined, and that $\pi_\F\circ \psi=\phi_\F$ and $\pi_\G\circ \psi=\phi_\G$.
\end{proof}

\begin{defin}
The category $\F\times_{\mathcal{H}}\G$ is called the \gr{fibered product} of $\F$ and $\G$ over $\mathcal{H}$.
\end{defin}

\begin{rmk}
The property that we used as starting point to define the fibered product looks much like a universal property (which should define it up to equivalence), apart from the fact that there is no uniqueness required on the functor $\psi$. On the other hand we defined the fibered product explicitly, and we will not need this ``uniqueness'' part.

Nevertheless, we remark that it is possible to give a universal property that identifies the fibered product up to equivalence, but the natural setting in which this property is stated is that of 2-categories.
\end{rmk}

\subsection*{The local flatness criterion}

The following Theorem gives an important flatness criterion.

\begin{thm}[Local flatness criterion]\label{local-flatness}
Let $A$ be a ring, $I\subseteq A$ a proper ideal, and $M$ an $A$-module. If either

\begin{enumeratei}

\item $I$ is nilpotent, or

\item $A$ is a noetherian local ring and $M$ is a finitely generated $B$-module, where $B$ is a noetherian local ring with a local homomorphism $\phi\colon A\to B$ and the two structures of module on $M$ are compatible with $\phi$

\end{enumeratei}

then the following conditions are equivalent:
\begin{itemize}
\item $M$ is a flat $A$-module.
\item $M/IM$ is flat over $A/I$ and $\Tor^A_1(M,A/I)=0$.
\item $M/I^nM$ is flat over $A/I^n$ for every $n\geq 1$.
\end{itemize}
\end{thm}

See \cite[$\S$~22]{Mat} for this.

%

\subsection*{A base change result}

Let $X$ be a scheme over a noetherian ring $A$, and $\E$ be a coherent sheaf on $X$. We want to understand the relation between the $A$-modules $\H^i(X,M\otimes_A \E)$ and $M\otimes_A \H^i(X,\E)$ (this is a particular case of the ``base change problem'').

There is a natural homomorphism
$$
\phi_M^i\colon M\otimes_A \H^i(X,\E)\arr \H^i(X,M\otimes_A \E)
$$
that is defined as follows. An element $m\in M$ corresponds to a homomorphism of $A$-modules $m\colon A\to M$, defined by $a\mapsto a\cdot m$. We can consider then the homomorphism $m\otimes \id\colon \E\simeq A\otimes_A \E\to M\otimes_A \E$, which induces a homomorphism in cohomology $(m\otimes \id)_*\colon \H^i(X,\E)\to \H^i(X,M\otimes_A \E)$.

We define then a function $F\colon M\times \H^i(X,\E)\to \H^i(X,M\otimes_A \E)$ by $F(m,\alpha)=(m\otimes \id)_*(\alpha)$; this function is $A$-bilinear in both variables, and so it induces a homomorphism of $A$-modules $\phi_M^i\colon M\otimes_A \H^i(X,\E)\to \H^i(X,M\otimes_A \E)$.

We have the following classical result.

\begin{thm}
Let $X$ be a proper scheme over $A$, $\E$ a coherent sheaf on $X$, flat over $A$, and assume that for every closed point $p \in \spec A$ and a fixed $i$ the homomorphism
$$
\phi_{k(p)}^i\colon k(p)\otimes_A \H^i(X,\E)\arr \H^i(X,k(p)\otimes_A \E)
$$
is surjective. Then for every $A$-module $M$ the homomorphism $\phi_M^i$ is an isomorphism.
\end{thm}

\begin{defin}
If the conclusion of this theorem holds for a coherent sheaf $\E$ and a natural number $i$, we say that the cohomology group $\H^i(X,\E)$ \gr{satisfies base change}.
\end{defin}

For a discussion of base change and the theorem above, see for example Sections 7.7 and 7.8 of \cite{EGAIII}, or \cite[III, 12]{Har}.


The following theorem tells us that sheaves of differentials satisfy base change in a particular case (see \cite[Theorem 5.5 (i)]{Del}).
\begin{thm}[Deligne]\label{deligne}
Let $X$ be a proper and smooth scheme over a noetherian $\Q$-algebra $A$, and consider the coherent sheaf of K\"{a}hler differentials $\diff_{X/A}$, and its exterior powers $\diff_{X/A}^i=\bigwedge^i\diff_{X/A}$. Then all the cohomology groups $\H^i(X,\diff_{X/A}^j)$ satisfy base change.
\end{thm}

\section{The Rim--Schlessinger condition for flat modules and algebras}
\label{appd}

In this Appendix we prove Proposition~\ref{schemes.rs}. For any ring $A$ we denote by $\Modfl A$ the category of flat $A$-modules. Given a ring homomorphism $A \arr B$, we have a functor $\Modfl A \arr \Modfl B$ obtained by tensoring with $B$ over $A$. Suppose that we have a diagram
   \[
   A' \stackrel{\rho'}{\larr} A \stackrel{\rho''}{\longleftarrow} A''
   \]
of ring homomorphisms, such that $\rho''\colon A'' \arr A$ is surjective with nilpotent kernel; we denote this diagram by $A_{\bullet}$. Let $B$ be the fibered product $A'\times_{A}A''$. An object of the fibered product $\Modfl A'\times_{\Modfl A} \Modfl A''$ consists of a triplet $(M', M'', \alpha_{M})$, where $M'$ and $M''$ are flat modules over $A'$ and $A''$ respectively, and $\alpha_{M}$ is an isomorphism of $A$-modules $\alpha_{M}\colon M' \otimes_{A'}A \simeq M'' \otimes_{A''}A$. An arrow $(M_{1}', M_{1}'', \alpha_{M_{1}}) \arr (M_{2}', M_{2}'', \alpha_{M_{2}})$ is a pair $(f', f'')$ of homomorphisms $f'\colon M'_{1} \arr M'_{2}$ and $f''\colon M''_{1} \arr M''_{2}$ that compatible with the $\alpha$'s, in the obvious sense. We denote the category $\Modfl A'\times_{\Modfl A} \Modfl A''$ by $\Modfl A_{\bullet}$; its objets will be called \emph{flat $A_{\bullet}$-modules}.

The natural functor $\Phi\colon \Modfl B  \arr \Modfl A_{\bullet}$ sends a flat $B$-module $N$ into the triplet $N\otimes_{B}A_{\bullet} \eqdef (N\otimes_{B}A', N\otimes_{B}A'', \alpha_{N})$, where $\alpha_{N}$ is the isomorphism $(N\otimes_{B}A') \otimes_{A'} A \simeq (N\otimes_{B}A'')\otimes_{A''}A$ obtained by composing the canonical isomorphisms $(N\otimes_{B}A') \otimes_{A'} A \simeq N\otimes_{B}A$ and $N\otimes_{B}A \simeq (N\otimes_{B}A') \otimes_{A'} A$.

One can also consider the category $\Algfl A$ of flat $A$-algebras, define the category $\Algfl A_{\bullet}$ as the fibered product $\Algfl A' \times_{\Algfl A} \Algfl A''$, and define a functor $\Phi\colon \Algfl B \arr \Algfl A_{\bullet}$ in a completely similar fashion.

\begin{prop}\label{rs-alg}
The functors $\Modfl B \arr \Modfl A_{\bullet}$ and $\Algfl B \arr \Algfl A_{\bullet}$ are equivalences of categories.
\end{prop}

\begin{proof}
Let us construct a quasi-inverse $\Psi\colon \Modfl A_{\bullet} \arr \Modfl B$. Given a flat $A\bullet$-module $M = (M', M'', \alpha_{M})$, we define $\Psi M$ as the kernel of the surjective homomorphism $\delta_{M}\colon M'\times M''\arr M''\otimes_{A''}A$ defined by $(m', m'') \mapsto \alpha_{M}(m'\otimes 1) - m''\otimes 1$. The product $M' \times M''$ is an $A' \times A''$-module, and $\Psi M$ is a sub-$B$-module. This defines a functor from $\Modfl A_{\bullet}$ to the category of $B$-modules.

Notice that by tensoring the exact sequence of $B$-modules
   \[
   \xymatrix{
   0 \ar[r]& B \ar[r]& A' \times A'' \ar[r]^-{\rho' - \rho''}& A \ar[r]& 0
   }
   \]
with a flat $B$-module $N$ we get an exact sequence
   \[
   \xymatrix{
   0 \ar[r]& N \ar[r]& (N\otimes_{B}A') \times (N\otimes_{B}A'')
   \ar[r]& N \otimes_{B} A \ar[r]& 0
   }\,,
   \]
which yields a functorial isomorphism $\Psi\Phi N \simeq N$ of $B$-modules. 

So we have left to show two things:

\begin{enumeratea}

\item $\Psi M$ is always flat, and 

\item for all $M$ in $\Modfl A'\times_{\Modfl A} \Modfl A''$, there is a functorial isomorphism $\Phi\Psi M \simeq M$.

\end{enumeratea}

So, let $M = (M', M'', \alpha_{M})$ be a flat $A_{\bullet}$-module. The inclusion $\Psi M \subseteq M' \times M''$ induces $B$-linear maps $\Psi M \arr M'$ and $\Psi M \arr M''$, which in turn induce homomorphisms $\phi'_{M}\colon \Psi M\otimes_{B}A' \arr M'$ and $\phi''_{M}\colon \Psi M\otimes_{B}A'' \arr M''$. We will show that $\Psi M$ is a flat $B$-module, and that $\phi'_{M}$ and $\phi''_{M}$ are isomorphisms. The functorial isomorphism $\Phi\Psi M \simeq M$ is $(\phi'_{M}, \phi''_{M})$. (Notice that $\phi'_{M}$ and $\phi''_{M}$ are isomorphism when $M$ is of the form $\Phi N$, as was shown above).

Choose a free $B$-module $F_{0}$ with a surjective homomorphism of $A'$-modules $f'\colon F_{0}\otimes_{B}A' \arr M'$. The induced homomorphism of $A$-modules
   \[
   \alpha \circ (f\otimes \id_{A})\colon  F_{0} \otimes_{B}
   A \arr M''\otimes_{A''}A
   \]
lifts to a homomorphism of $A''$ modules $f''\colon F_{0}\otimes_{B}A'' \arr M''$, which is still surjective, because the kernel of $A'' \arr A$ is nilpotent. Then $(f',f'')$ gives a surjective homomorphism $F \eqdef \Phi F_{0} \arr M$ in $\Modfl A_{\bullet}$. Denote by $K'$ and $K''$ the kernels of $f'$ and $f''$ respectively. Because of the flatness of $M'$ and $M''$, the two sequences
   \[
   \xymatrix{
   0 \ar[r]& K'\otimes_{A'}A \ar[r]& F'\otimes_{A'}A \ar[r]&
   M'\otimes_{A'}A \ar[r]& 0 
   }
   \]
and
   \[
   \xymatrix{
   0 \ar[r]& K''\otimes_{A''}A \ar[r]& F''\otimes_{A''}A \ar[r]&
   M'\otimes_{A''}A \ar[r]& 0 
   }
   \]
are exact; hence the isomorphism $\alpha_{F}\colon F'\otimes_{A'}A \simeq F''\otimes_{A''}A$ restricts to an isomorphism $\alpha_{K}\colon K'\otimes_{A'}A \simeq K''\otimes_{A''}A$. This defines a flat $A_{\bullet}$-module $K \eqdef (K', K'', \alpha_{K})$. We have a commutative diagram
   \[
   \xymatrix{
   &0\ar[d]&0\ar[d]&0\ar[d]&\\
   0\ar[r]& \Psi K \ar[r]\ar[d] & K' \times K'' \ar[r]^-{\delta_{K}}\ar[d] &
   K''\otimes_{A''}A \ar[r]\ar[d] & 0\\
   0\ar[r]& \Psi F \ar[r]\ar[d] & F'
   \times F'' \ar[r]^-{\delta_{F}}\ar[d] &
   F''\otimes_{A''}A \ar[r]\ar[d] & 0\\
   0\ar[r]& \Psi M \ar[r]\ar[d] & M' \times M'' \ar[r]^-{\delta_{M}}\ar[d] &
   M''\otimes_{A''}A \ar[r]\ar[d] & 0\\
   &0&0&0&\\
   }
   \]
with exact rows, whose middle and right hand columns are exact. This implies that the left hand column is also exact. Tensoring this with $A'$ we get a  commutative diagram
   \[
   \xymatrix{
   &\Psi K\otimes_{B}A' \ar[r]\ar[d]^-{\phi'_{K}}&
   \Psi F\otimes_{B}A' \ar[r]\ar[d]^-{\phi'_{F}}&
   \Psi M\otimes_{B}A' \ar[r]\ar[d]^-{\phi'_{M}}& 0\\
   0\ar[r] & K' \ar[r]& F' \ar[r]& M' \ar[r]& 0
   }
   \]
with exact rows. Since $\phi'_{F}$ is easily seen to be an isomorphism, we see that $\phi'_{M}$ is also surjective. But here $M$ is a arbitrary flat $A_{\bullet}$-module, so $\phi'_{K}$ must also be surjective; and by diagram chasing we get that $\phi'_{M}$ is an isomorphism. The argument for $\phi''_{M}$ is completely analogous.

Then $\phi'_{K}$ is also an isomorphism; hence the sequence
   \[
   \xymatrix{
   0\ar[r]&\Psi K\otimes_{B}A' \ar[r]&
   \Psi F\otimes_{B}A' \ar[r]&
   \Psi M\otimes_{B}A' \ar[r]& 0
   }
   \]
is exact. Since $F_{0} = \Psi F$ is free, this means that $\tor_{1}^{A'}(\Psi M, A') = 0$. The projection $B \arr A'$ is surjective with nilpotent kernel, since it is obtained by base change from $\rho''\colon A'' \arr A$; hence by the local criterion of flatness (Theorem~\ref{local-flatness}) we have that $\Psi M$ is flat over $B$, and this completes the proof for the categories of flat modules.

For the case of flat algebras, it is sufficient to notice that for any $A_{\bullet}$ algebra $M$, the module $\Psi M$ has a natural algebra structure, and that the natural isomorphisms $\Phi\Psi \simeq \id$ and $\Psi\Phi \simeq \id$ defined above preserve the algebra structures.
\end{proof}

\begin{proof}[Proof of Proposition~\ref{schemes.rs}]
The fact that the restriction of the functor
   \[
   \Phi\colon \catdef(A')\times_{\catdef(A)}\catdef(A'')\to 
   \catdef(A'\times_A A'')
   \]
to affine schemes is an equivalence follows immediately from Proposition~\ref{rs-alg}.

Let us extend the definition of the quasi-inverse
$$
\Psi\colon \catdef(A')\times_{\catdef(A)}\catdef(A'')\to \catdef(A'\times_A A'')
$$
of $\Phi$ that we described above to non-necessarily affine schemes. Take an object
$$
(X',X'',\theta) \in \catdef(A')\times_{\catdef(A)}\catdef(A'')
$$
that is, a pair of flat schemes $X'\to \spec A'$ and $X''\to \spec A''$ with an isomorphism $\theta\colon X''|_{\spec A}\simeq X'|_{\spec A}$ of the pullbacks to $A$. We can see this object as the following diagram
$$
\xymatrix@-15pt{
X'\ar@{|->}[dd]&&X''\ar@{|->}[dd]\\
&X=X'|_{\spec A}\ar@{|->}[dd]\ar[ur]\ar[ul]&\\
A'\ar[rd]&&A''\ar[ld]\\
&A&
}
$$
where the morphism $X'|_{\spec A}\to X''$ is the composite of the inverse of $\theta$ and of the closed immersion $X''|_{\spec A}\to X''$.

We consider then the sheaf of $A'\times_A A''$-algebras $\O_{X'}\times_{\O_{X}}\O_{X''}$ on the topological space $X$. The locally ringed space $\widetilde{X}=(X,\O_{X'}\times_{\O_{X}}\O_{X''})$ is a scheme by the affine case, and moreover it is flat over $A'\times_A A''$, since flatness is a local property. We set $\Psi(X',X'',\theta) \eqdef \widetilde{X}$.

By the universal property of fibered products one can easily see that an arrow $(X',X'',\theta)\to (Y',Y'',\nu)$ gives a morphism $\widetilde{X}\to \widetilde{Y}$, and routine verifications show that $\Phi$ and $\Psi$ are quasi-inverse to each other.
\end{proof}

\bibliographystyle{amsalpha}
\bibliography{def}

\end{document}